\documentclass[11pt]{article}
\usepackage[T1]{fontenc}
\usepackage[in]{fullpage}
\usepackage{url}            
\usepackage{booktabs}       
\usepackage{amsfonts}       
\usepackage{pifont}
\usepackage{subcaption}
\usepackage{endnotes}
\usepackage{mathrsfs}
\usepackage{float}

\usepackage[table,xcdraw]{xcolor}
\definecolor{airforceblue}{rgb}{0.36, 0.54, 0.66}
\definecolor{amber}{rgb}{1.0, 0.75, 0.0}
\definecolor{burntorange}{rgb}{0.8, 0.33, 0.0}
\definecolor{dimgray}{rgb}{0.41, 0.41, 0.41}
\usepackage{mathtools}

\usepackage{tikz}
\usetikzlibrary{fit}
\tikzset{%
	highlight/.style={rectangle,blend mode = multiply,draw=blue!90!black,thick,rounded corners = 0.3 mm,inner sep=0.5pt}
}

\usepackage{microtype}
\usepackage{graphicx}
\usepackage{booktabs} 

\usepackage{amsmath,amssymb,amsthm}
\usepackage[ruled,linesnumbered]{algorithm2e}
\usepackage{bbold}
\usepackage[sans]{dsfont}
\usepackage{pifont} 
\def\O#1{\text{\ding{\the\numexpr#1+171}}}

\usepackage{thmtools}
\usepackage{thm-restate}
\usepackage{multirow}
\usepackage{tabularx}
\usepackage{textcomp}

\usepackage[shortlabels]{enumitem}

\usepackage{thmtools}

\declaretheoremstyle[parent=section]{definitionwithend}
\declaretheorem[style=definitionwithend]{corollary}
\declaretheorem[style=definitionwithend]{condition}
\declaretheorem[style=definitionwithend]{theorem}

\declaretheorem[style=definitionwithend]{definition}
\declaretheorem[style=definitionwithend]{assumption}
\declaretheorem[style=definitionwithend]{example}
\declaretheorem[style=definitionwithend]{remark}

\declaretheorem[style=definitionwithend]{lemma}

\usepackage{tikz}
\usetikzlibrary{shadows}
\usetikzlibrary{arrows,positioning} 

\usepackage{contour}
\usepackage[normalem]{ulem}
\usepackage{float}

\contourlength{0.8pt}

\numberwithin{equation}{section}

\usepackage{bm}

\usepackage{wrapfig}

\usepackage[dvipsnames]{xcolor}
\definecolor{cmtblue}{HTML}{1F4E79}
\definecolor{cmtpurple}{HTML}{6A1B9A}
\definecolor{cmtorange}{HTML}{B26A00}
\definecolor{cmtred}{HTML}{B00020}

\definecolor{amethyst}{rgb}{0.6, 0.4, 0.8}
\newcount\Comments
\newcommand{\kibitz}[2]{\ifnum\Comments=1{\textcolor{#1}{\textsf{\footnotesize #2}}}\fi}

\usepackage{authblk}

\newcommand{\bz}{\mathbf 0}

\newcommand{\R}{\mathbb R}

\newcommand{\eps}{\epsilon}

\DeclareMathOperator{\prox}{prox}
\DeclareMathOperator{\proj}{proj}

\DeclareMathOperator{\dist}{dist}

\DeclareMathOperator*{\argmin}{argmin}
\DeclareMathOperator*{\argmax}{argmax}
\DeclareMathOperator{\FOAM}{FOAM}

\def\O#1{\text{\ding{\the\numexpr#1+171}}} 

\usepackage{threeparttable}
\allowdisplaybreaks[4] 

\def\L{{\mathbb{L}}}

\def\R{{\mathbb{R}}}

\def\cD{{\cal D}}

\def\cI{{\cal I}}

\def\cO{{\cal O}}

\def \cR{{\cal \cR}}
\def\cS{{\cal S}}

\def\cX{{\cal X}}
\def\cY{{\cal Y}}

\def\x{\bm{x}}
\def\y{\bm{y}}
\def\z{\bm{z}}
\def\bz{{\mathbf 0}}

\def\1{{\mathbf 1}}
\def\prox{{\rm{prox}}}

\def\dist{{\operatorname{dist}}}
\def\bx{\bm{x}}
\def\by{\bm{y}}

\definecolor{clemson-orange}{RGB}{234,106,32}
\definecolor{chicago-maroon}{RGB}{128,0,0}
\definecolor{northwestern-purple}{RGB}{82,0,99}
\definecolor{cornell-red}{RGB}{179,27,27}
\definecolor{sauder-green}{RGB}{171,180,0}
\definecolor{harveymudd-gold}{RGB}{178,139,51}

\usepackage[colorlinks,citecolor=northwestern-purple,urlcolor=chicago-maroon,linkcolor=chicago-maroon]{hyperref}
\usepackage[nameinlink]{cleveref}
\crefname{assumption}{Assumption}{Assumptions}
\crefname{lemma}{Lemma}{Lemmas}
\crefname{theorem}{Theorem}{Theorems}
\crefname{corollary}{Corollary}{Corollaries}
\crefname{proposition}{Proposition}{Propositions}
\crefname{condition}{Condition}{Conditions}
\crefname{claim}{Claim}{Claims}
\crefname{procedure}{Procedure}{Procedures}
\crefname{algorithm}{Algorithm}{Algorithms}
\crefname{figure}{Figure}{Figures}
\crefname{remark}{Remark}{Remarks}
\crefname{section}{Section}{Sections}
\crefname{procedure}{Procedure}{Procedures}
\crefname{example}{Example}{Examples}
\crefname{definition}{Definition}{Definitions}
\crefname{table}{Table}{Tables}
\crefname{equation}{}{}
\crefname{enumi}{}{}
\crefname{conjecture}{Conjecture}{Conjectures}
\crefname{step}{Step}{Steps}
\crefname{footnote}{Footnote}{Footnotes}
\crefname{appendix}{Appendix}{Appendices}
\Crefname{appendix}{Appendix}{Appendices}
\AtBeginEnvironment{appendices}{\crefalias{section}{appendix}}
\AtBeginEnvironment{appendices}{\crefalias{subsection}{appendix}}

\definecolor{bittersweet}{rgb}{1.0, 0.44, 0.37}

\usepackage{todonotes}
\usepackage{natbib}
\usepackage{tcolorbox}
\newtcbox{\alertinline}[1][red]
  {on line, arc = 0pt, outer arc = 0pt,
    colback = #1!20!white, colframe = #1!50!black,
    boxsep = 0pt, left = 1pt, right = 1pt, top = 2pt, bottom = 2pt,
    boxrule = 0pt, bottomrule = 1pt, toprule = 1pt}
\usepackage{newtxtext}
\newtcolorbox{textbox}[1]{
    sharp corners,
    boxsep=0mm,
    toptitle=2mm,
    lefttitle=0mm,
    colframe=black!3,
    colback=black!3,
    title={\rule[-2pt]{4.5pt}{10pt}\hspace*{1.5mm}#1},
    fonttitle=\bfseries\itshape\sffamily,
    coltitle=black,
    halign=flush left,
}
\usepackage[titletoc]{appendix}
\usepackage{lipsum}

\usepackage{setspace}

\usepackage{makecell} 

\def\O#1{\text{\ding{\the\numexpr#1+171}}}
\usepackage{authblk} 
\begin{document}
\title{{
\bf Smoothing Meets Perturbation: Unified and Tight Analysis for Nonconvex-Concave Minimax Optimization}}
\author{
Jiajin Li\thanks{Sauder School of Business, University of British Columbia, Vancouver BC, Canada. \texttt{jiajin.li@sauder.ubc.ca}}, \quad 
Mahesh Nagarajan\thanks{Sauder School of Business, University of British Columbia, Vancouver BC, Canada. \texttt{mahesh.nagarajan@sauder.ubc.ca}}, \quad 
Siyu Pan\thanks{Sauder School of Business, University of British Columbia, Vancouver BC, Canada. \texttt{siyu.pan@sauder.ubc.ca}}, \quad 
Nanxi Zhang\thanks{Ivey Business School, University of Western Ontario, London ON, Canada. \texttt{nzhang@ivey.ca}} \thanks{Authors are listed in alphabetical order.}
}
\date{\today}

\maketitle
\vspace{-5mm}

\begin{abstract}This paper studies smooth nonconvex-concave minimax optimization and two acceleration mechanisms for single-loop first-order methods: dual perturbation and smoothing. Although both techniques improve convergence guarantees, their relative advantages remain unclear due to the distinction between game stationarity (GS) and optimization stationarity (OS). We provide a tight characterization of their iteration complexities under both notions. We show that smoothing accelerates convergence to both GS and OS, whereas dual perturbation improves the rate only for GS and does not accelerate OS. Matching lower bounds based on hard instances establish the tightness of these rates. Motivated by this separation, we propose \emph{Perturbed Smoothed GDA}, a single-loop method combining both techniques. It improves the complexity for GS over existing single-loop methods while preserving the state-of-the-art rate for OS, and further admits asymptotic convergence to 0-GS, which is not available for vanilla \emph{Smoothed GDA}.

\end{abstract}

\section{Introduction}\label{sec:intro}
In this paper, we consider the smooth nonconvex–concave (NC-C) minimax problem   
\begin{equation}
\min_{\x \in \mathcal{X}} \max_{\y \in \mathcal{Y}} f(\x,\y), \label{eq:prob}
\end{equation}  
where $f:\R^n\times \R^d\rightarrow\R$ is $\ell$-smooth with respect to both $\x$ and $\y$,
$\cX\subseteq \R^n$ is a nonempty, closed convex set, and $\cY \subseteq \R^d$ is a nonempty, compact, convex set. 
Problem \cref{eq:prob} has received increasing attention due to its extensive applications in machine learning and operations research, including training generative adversarial networks \citep{arjovsky2017wasserstein,goodfellow2020generative}, robust
training of deep neural networks~\citep{sinha2018certifying}, multi-agent reinforcement learning \citep{omidshafiei2017deep,dai2018sbeed} and game theory (e.g., finding first-order Nash equilibria (FNE), see \citep{nouiehed2019solving,ostrovskii2021efficient}).

Motivated by the computational efficiency and the needs of large-scale learning, we primarily focus on single-loop first-order methods. 
The simplest algorithm of this class is gradient descent–ascent (\emph{GDA}).
In \emph{GDA}, each iteration consists of a gradient‐descent step on the minimization variable $\x$ and a gradient‐ascent step on the maximization variable $\y$.
Although vanilla \emph{GDA} may fail to converge in the general smooth NC-C setting, \citet{lin2020gradient} show that a two‐timescale variant (\emph{TS-GDA})—with a substantially larger step‐size for ascent than for descent—restores convergence guarantees. However, this guarantee comes with a slow convergence rate: TS-GDA requires $\cO(\eps^{-6})$ iterations to reach both $\eps$-game stationary ($\epsilon$-GS) and $\eps$-optimization stationary ($\epsilon$-OS) points (cf. \citep[Definition A.5]{lin2020near}, \citep[Definition 3.1]{li2025nonsmooth}). 
The slow convergence of \emph{TS-GDA} stems from the fact that it  essentially behaves as an inexact subgradient method applied to the nonsmooth value function $\Phi(\x)\coloneqq
\max_{\y\in\mathcal Y} f(\x,\y)+\iota_{\cX}(\x)$, where $\iota_{\cX}(\cdot)$ is an indicator function that takes value $0$ if $\x\in\cX$  and $+\infty$ otherwise.

To improve the convergence rate, two acceleration techniques have recently been proposed to overcome this nonsmooth barrier. 
The first is \emph{Smoothing via the Moreau envelope}: The Moreau envelope of $\Phi$ provides a smooth surrogate with Lipschitz-continuous gradients~\citep{davis2019stochastic}.
Building on this idea, \citet{zhang2020single} propose the \emph{Smoothed GDA} method, which replaces the original objective with a regularized surrogate $f(\x,\y)+\frac{r}{2}\|\x-\z\|^2$ ($r>\ell$), where $\z$ serves as an approximate proximal point of $\x$ with respect to the value function $\Phi$. Rather than solving the proximal subproblem exactly, \citet{zhang2020single} perform a single \emph{GDA} update on $(\x,\y,\z)$ in each iteration, leading to an improved iteration complexity of $\mathcal{O}(\epsilon^{-4})$ for finding an $\epsilon$-GS point~\citep{zhang2020single}.
A recent work \citep{li2025nonsmooth} further shows  that \emph{Smoothed GDA} also attains an $\epsilon$-OS point in $\mathcal{O}(\epsilon^{-4})$ iterations.

The second acceleration technique is \emph{dual perturbation}.
This approach adds a small quadratic regularizer to the dual objective, thereby making it strongly concave. 
Such a perturbation is closely related to the regularization technique of \citep{nesterov2013gradient}.
By the envelope theorem, this modification also yields a smooth value function $ \max_{\y\in\cY} f(\cdot,\y)-\frac{\eps}{2}\|\y\|^2$. Moreover, due to the compactness of $\cY$, any $\epsilon$-GS of the perturbed problem is also a $\Theta(\epsilon)$-GS point of the original problem.
Applying \emph{GDA} to this perturbed formulation yields the \emph{Perturbed GDA} algorithm, which attains an iteration complexity of $\mathcal{O}(\epsilon^{-4})$ for finding an $\epsilon$-GS solution~\citep{xu2023unified,lu2020hybrid,xu2026riemannianalternatingdescentascent}.

However, in the existing literature, direct comparisons of iteration complexities between dual perturbation and smoothing can be misleading, as the literature often fails to clearly distinguish between two distinct notions of stationarity: 
game stationarity and optimization stationarity.
This oversight has fostered the inaccurate impression that both smoothing and dual perturbation yield the same $\mathcal{O}(\epsilon^{-4})$ rate for smooth problems under single-loop first-order methods. 
In reality, \citet{li2025nonsmooth} clarifies that \emph{Smoothed GDA} achieves both $\epsilon$-GS and $\epsilon$-OS within $\mathcal{O}(\epsilon^{-4})$ iterations, and the dual-perturbation schemes developed in \citep{xu2023unified,lu2020hybrid} establish an $\mathcal{O}(\epsilon^{-4})$ rate only for $\epsilon$-GS. To the best of our knowledge, the iteration complexity of \emph{Perturbed GDA} for attaining an $\epsilon$-OS point has not been characterized in the existing literature.

\vskip 5pt
\noindent\textbf{Our contributions.}
This distinction between GS and OS is more than a notational nuance; GS and OS impose intrinsically different complexity requirements, as already evidenced by their different complexity in the multi-loop setting. 
\citet{lin2020near} show that multi-loop first-order methods attain an $\epsilon$-GS solution with iteration complexity $\mathcal{O}(\epsilon^{-2.5}\log^2(\tfrac{1}{\epsilon}))$, whereas attaining an $\epsilon$-OS solution requires $\mathcal{O}(\epsilon^{-3}\log^2(\tfrac{1}{\epsilon}))$ iterations. 
This disparity in complexity shows that OS and GS are distinct for smooth NC-C minimax problems,
which raises the following question for single-loop first-order methods:
\begin{enumerate}[label=(Q\arabic*)]
    \item \bf For single-loop first-order algorithms, is there a complexity gap between  the iteration complexity of finding an $\epsilon$-GS and $\epsilon$-OS? \label{q:gs-os-gap}
\end{enumerate}
In this paper, we provide a precise characterization addressing \cref{q:gs-os-gap}. 
We first show that for \emph{Perturbed GDA}, although it achieves an iteration complexity of $\cO(\epsilon^{-4})$ for computing an $\epsilon$-GS, its complexity degrades to $\cO(\epsilon^{-6})$ for attaining an $\epsilon$-OS. We further prove that these bounds are tight by constructing matching hard instances on which the algorithm indeed requires this many iterations. This establishes a tight separation between GS and OS for \emph{Perturbed GDA}.
 In contrast, \emph{Smoothed GDA} achieves the same iteration complexity under the two stationarity notions, attaining both $\epsilon$-GS and $\epsilon$-OS within $\cO(\epsilon^{-4})$ iterations. We also establish matching lower bounds, certifying that these rates are tight.




These results lead to the following conceptual picture of dual perturbation and smoothing: 
Dual perturbation achieves an accelerated rate under the GS notion,  but its complexity deteriorates  under the OS notion.  In contrast, smoothing attains the same accelerated rate for both GS and OS stationarity.  This contrast implies that the two techniques accelerate the game stationarity through distinct mechanisms, and raises the possibility of a synergy effect of the two techniques. 

This motivates the second question we tackle in this paper. 

\begin{enumerate}[label=(Q2)]
    \item {\bf 
 Can combining dual perturbation and smoothing improve the iteration complexity beyond either technique alone?
    }\label{q2}
\end{enumerate}


To answer \cref{q2}, we introduce \emph{Perturbed Smoothed GDA}, a single-loop algorithm that uses both smoothing and dual perturbation techniques.
Our analysis shows that this algorithm achieves an $\epsilon$-OS in $\mathcal{O}(\epsilon^{-4})$ iterations, and this rate is tight under our assumptions. 
This result confirms the earlier insight that dual perturbation does not improve the iteration complexity to $\epsilon$-OS as the iteration complexity remains identical to that of \emph{Smoothed GDA}.
On the other hand, our \emph{Perturbed Smoothed GDA}  achieves an $\epsilon$-GS in $\mathcal{O}(\epsilon^{-3})$ iterations, a strict improvement over both \emph{Smoothed GDA} and \emph{Perturbed GDA}, which each requires $\mathcal{O}(\epsilon^{-4})$ iterations. This result gives a positive answer to \cref{q2} and shows that smoothing and perturbation can be synergistic in accelerating convergence to game stationarity.

\begin{table}[H]
  \centering
  \small
  \begin{tabular}{l  cc  cc}
  \toprule
  \multirow{2}{*}{Single Loop Algorithms} 
    & \multicolumn{2}{c}{Optimization Stationarity} 
    & \multicolumn{2}{c}{Game Stationarity} \\ 
  \cmidrule(lr){2-3} \cmidrule(lr){4-5}
    & Upper Bound & Lower Bound 
    & Upper Bound & Lower Bound \\
  \midrule
  \makecell[l]{\emph{TS-GDA}\\ \scriptsize\citep{lin2020gradient,boct2020alternating,mahdavinia2022tight}}
    & \(\cO\bigl(\tfrac{\ell^3L^2D_{\cY}^2\Delta_\Phi}{\epsilon^6}\bigr)\) 
    & \(\Omega\bigl(\tfrac{\ell^3L^2D_{\cY}^2\Delta_\Phi}{\epsilon^6}\bigr)\)
    & \(\cO\bigl(\tfrac{\ell^3L^2D_{\cY}^2\Delta_\Phi}{\epsilon^6}\bigr)\) 
    & \(\Omega\bigl(\tfrac{\ell^3L^2D_{\cY}^2\Delta_\Phi}{\epsilon^6}\bigr)\)  \\[2pt]  

  \makecell[l]{\emph{Smoothed GDA} \\ \scriptsize\citep{zhang2020single,li2025nonsmooth}}
    & \(\cO\bigl(\tfrac{ \ell^3D_{\cY}^{2}\Delta_{\Psi_2}}{\epsilon^{4}}\bigr)\) 
    & \cellcolor{gray!15}{\(\Omega\bigl(\tfrac{\ell^3D_{\cY}^{2}\Delta_{\Psi_2}}{\epsilon^{4}}\bigr)\)}
    & \(\cO\bigl(\tfrac{\ell^{3}D_{\cY}^{2}\Delta_{\Psi_2}}{\epsilon^{4}}\bigr)\)
    & \cellcolor{gray!15}{\(\Omega \bigl(\tfrac{\ell^{3}D_{\cY}^{2}\Delta_{\Psi_2}}{\epsilon^{4}}\bigr)\)}\\[2pt]

  \makecell[l]{\emph{Perturbed GDA}\\ \scriptsize\citep{xu2023unified,lu2020hybrid}}
    & \cellcolor{gray!15}{\(\cO\bigl(\tfrac{\ell^5D_{\cY}^{4}\Delta_{\Psi_1}}{\epsilon^{6}}\bigr)\)}
    & \cellcolor{gray!15}{\(\Omega\bigl(\tfrac{\ell^5D_{\cY}^{4}\Delta_{\Psi_1}}{\epsilon^{6}}\bigr)\)} 
    & \(\cO\bigl(\tfrac{\ell^{3}D_{\cY}^{2}\Delta_{\Psi_1}}{\epsilon^{4}}\bigr)\) 
    & \cellcolor{gray!15}{\(\Omega \bigl(\tfrac{\ell^{3}D_{\cY}^{2}\Delta_{\Psi_1}}{\epsilon^{4}}\bigr)\)}\\[2pt]

  \rowcolor{gray!15}
  \makecell[l]{\textbf{\emph{Perturbed Smoothed GDA}}\\ \scriptsize\textit{(this work)}}
    & \(\cO\bigl(\tfrac{\ell^3D_{\cY}^{2}\Delta_{\Psi_2}}{\epsilon^{4}}\bigr)\) 
    & \(\Omega\bigl(\tfrac{\ell^3D_{\cY}^{2}\Delta_{\Psi_2}}{\epsilon^{4}}\bigr)\) 
    & \(\cO\bigl(\tfrac{\ell^{2}D_{\cY}\Delta_{\Psi_2}}{\epsilon^{3}}\bigr)\) 
    & \(\Omega\bigl(\tfrac{\ell^{2}D_{\cY}\Delta_{\Psi_2}}{\epsilon^{3}}\bigr)\) \\
    
  \bottomrule 
  \end{tabular}
  \caption{Gray‐shaded entries denote results newly developed in this paper. Here, $L$ is the Lipschitz constant of $f$; $D_{\cY}$ stands for the diameter of $\cY$;  {$\Delta_\Phi$, $\Delta_{\Psi_1}$, and $\Delta_{\Psi_2}$ measure the initialization of the algorithm, and their formal definitions are given in \cref{def:initial_gap}.}}
  \label{tab:combined}
\end{table}
\vspace{-2mm}

This iteration complexity difference can be explained by the structural relation between $\epsilon$-GS and $\epsilon$-OS. It is known that $\epsilon$-OS implies $\Theta(\epsilon)$-GS, but not conversely, see, e.g., \citep[Theorem 7.1]{li2025nonsmooth}; thus, OS is a strictly stronger criterion for convergence. Correspondingly, smoothing and dual perturbation act on different components of the problem. Smoothing modifies the  value function and directly targets OS, which automatically secures GS guarantees. In contrast, dual perturbation regularizes the inner maximization objective and primarily facilitates the convergence to GS, but the lack of a reverse implication ($\epsilon$-\text{GS} $\not\Rightarrow $$\epsilon$-\text{OS}) prevents this benefit from extending to OS.
Therefore, the two techniques operate through distinct yet complementary mechanisms: Smoothing promotes convergence to OS (and by extension GS), whereas dual perturbation mainly accelerates convergence to GS by rendering the inner maximization (dual) objective strongly concave. When combined, these effects yield a strictly better complexity for finding an $\epsilon$-GS point than either technique alone.

 Beyond the improvement in iteration complexity, the perturbation component in \emph{Perturbed Smoothed GDA} also enables us to establish an asymptotic convergence result, i.e., we show that \emph{Perturbed Smoothed GDA} converges to 0-GS when its iteration number goes to infinity.
Such an asymptotic convergence result is not available for  \emph{Smoothed GDA} under the current analysis.

Importantly, the smoothing-perturbation synergy is structural and is not restricted to single-loop methods. 
Building on the same principle, we propose a double-loop first-order algorithm, \emph{Perturbed Smoothed FOAM}, that combines dual perturbation and smoothing within a multi-loop framework. 
This combination transforms the original NC-C problem into a sequence of strongly convex-strongly concave (SC-SC) subproblems, thereby enabling the use of optimal first-order methods within each iteration. 
In particular, we employ \emph{FOAM} \citep{kovalev2022first} as an efficient inner solver, although the improvement does not rely on a specific choice of solver.
Our algorithm achieves iteration complexities $\cO(\epsilon^{-2.5}\log(\frac{1}{\epsilon}))$ for $\epsilon$-GS and $\cO(\epsilon^{-3}\log(\tfrac{1}{\epsilon}))$ for $\epsilon$-OS. 
In comparison, most existing multi-loop methods either attain weaker guarantees or rely on more involved triple-loop constructions to obtain their rates \citep{nouiehed2019solving,thekumparampil2019efficient,lin2020near,kong2021accelerated,zhao2024primal}. 
A recent double-loop method proposed in \citep{lu2026solving} achieves the rate $\cO(\epsilon^{-2.5}\log(\tfrac{1}{\epsilon}))$ for $\epsilon$-GS, but does not provide a corresponding guarantee for $\epsilon$-OS. 
Thus, \emph{Perturbed Smoothed FOAM} matches the best-known complexity for $\epsilon$-GS while improving the state-of-the-art complexity for $\epsilon$-OS among first-order methods~\citep{zhang2025avoid}.

Finally, on the technical side, we develop a unified Lyapunov-function-based framework that subsumes the analysis of existing single-loop methods, including \emph{Smoothed GDA} and \emph{Perturbed GDA}, as well as our proposed \emph{Perturbed Smoothed GDA}. 
The analysis for \emph{Perturbed Smoothed GDA} avoids the cumbersome two-stage arguments for \emph{Smoothed GDA} used in prior works such as \citep{zhang2020single,li2025nonsmooth}. 
 As a result, we are able  to prove an asymptotic convergence result for  \emph{Perturbed Smoothed GDA}.
Furthermore, the framework extends naturally to the double-loop setting and applies to our proposed \emph{Perturbed Smoothed FOAM}.
Overall, this yields a streamlined proof structure that not only simplifies the analysis but explicitly disentangles the distinct  roles of smoothing and perturbation in the convergence dynamics.


\vskip 5pt
\noindent\textbf{Notation.} 
We use bold lowercase letters (e.g., $\x,\y$) to denote vectors, and calligraphic uppercase letters (e.g., $\cX,\cY$) to denote sets. For a closed and convex set $\cX \subseteq \mathbb{R}^n$, 
the indicator function $\iota_{\cX}:\mathbb{R}^n \to \{0,+\infty\}$ is defined by $\iota_{\cX}(\x) = 0$ if $\x \in \cX$ and $+\infty$ otherwise. 
We write $\cD_{\cS}\coloneqq\sup_{\x,\y \in \cS}\|\x-\y\|$ for the diameter of $\cS$, and $\proj_{\cX}(\z)$ for the projection of a point $\z$ onto $\cX$. Given $r>0$, the proximal operator of a function (if well-defined) $f:\mathbb{R}^n \to \mathbb{R}$ at a point $\z$ is defined as $\prox_{r\cdot f}(\z)\coloneqq\argmin_{\x\in \cX} f(\x)+\tfrac{1}{2r}\|\x-\z\|^2$. We use $\dist(\x, \mathcal{S}):=\inf_{\z\in\mathcal{S}}\|\x-\z\|$ to denote the distance from $\x$ to set $\mathcal{S}$. For a positive integer $T$, we use the shorthand notation $[T] := \{1,\ldots,T\}$ to denote the index set of the first $T$ positive integers. Throughout this paper, $\epsilon \in (0,1)$ denotes a sufficiently small positive number. We write $a_k \searrow 0$ to denote that $a_k$ is a non-increasing sequence converging to $0$. 
\vskip 5pt
\noindent\textbf{Structure of the paper.}
The paper is organized as follows. 
In \Cref{sec:pre}, we introduce the preliminaries, including assumptions, stationarity notions, and the two acceleration tools: dual perturbation and smoothing. 
The main body of the paper focuses on single-loop methods: In \Cref{sec:main}, we present the main results, including the proposed \emph{Perturbed Smoothed GDA}  together with all iteration-complexity bounds and their tight analyses. 
In \Cref{sec:framework}, we develop a unified convergence analysis framework for  these single-loop algorithms. 
In \Cref{sec:analysis}, we provide a tight analysis that sharpens these iteration complexities. 
In \Cref{sec:double}, we further show that the two acceleration techniques and the unified analysis from \Cref{sec:framework} also apply to a double-loop setting, resulting in a new algorithm, \emph{Perturbed Smoothed FOAM},  improving the best known iteration complexities. 
Finally, \Cref{sec:close} concludes the paper.

\section{Preliminaries}\label{sec:pre}
In this section, we collect the main assumptions and definitions used in the analysis. We begin with a smoothness assumption on the objective function $f$ and then describe the properties required of the dual function. These assumptions are imposed throughout the paper.
\begin{assumption}[$\ell$-smooth]
    \label{l:smooth}
	The function $f$ is continuously differentiable and there exists a positive constant $\ell>0$ such that for all $\x,\x' \in \mathcal{X}$ and $\y,\y'\in \mathcal{Y}$, we have
		\[
		\begin{aligned}
			&\|\nabla_{\x}f(\x,\y)-\nabla_{\x}f(\x',\y')\| \leq \ell(\|\x-\x'\|+\|\y-\y'\|),~\mbox{and}\\
			&\|\nabla_{\y}f(\x,\y)-\nabla_{\y}f(\x',\y')\| \leq \ell(\|
   \x-\x'\|+\|\y-\y'\|).
		\end{aligned}
		\]
\end{assumption} 
\begin{assumption}\label{ass:bound}
For each $\x \in \cX$, the dual function $f(\x,\cdot)$ is concave. The dual domain  $\cY\subseteq \R^d$ is convex and bounded with diameter $D_{\cY}>0$. Without loss of generality (WLOG), we assume that $\textrm{\textbf{0}} \in \cY$.	
\end{assumption} 
Our goal   
is to find approximate stationary points of problem \cref{eq:prob}.
We introduce two standard stationarity measures~(cf. \citep[Definition 3.1]{li2025nonsmooth}, \citep[Definition A.5]{lin2020near}) that are used throughout the rest of the paper. 
\begin{definition}[Stationarity points]
\label{defi:primal-dual} 
Let $\epsilon>0$ be given.  
\begin{enumerate}[label=(\roman*)] 
\item The point $\x\in\mathcal{X}$ is an $\epsilon$-optimization-stationary point {\rm($\epsilon$-OS)} of problem \cref{eq:prob} if 
\[\left\|\prox_{\frac{1}{2\ell} \Phi}(\x)-\x\right\|\leq  \frac{\epsilon}{2\ell},\]
where $\Phi:\R^n \rightarrow \R \cup \{+\infty\}$ is given by $\Phi(\x) := \max_{\y' \in \cY} f(\x,\y') + \iota_{\cX}(\x)$.

 
\item  The pair $(\x,\y)\in \cX \times \cY$ is an $\epsilon$-game-stationary point {\rm ($\epsilon$-GS)} of problem \cref{eq:prob} if 
			\[
\dist(\bz,\nabla_{\x}f(\x,\y)+\partial\iota_\cX(\x))\leq \epsilon \quad \text{and} \quad \dist(\bz,-\nabla_{\y}f(\x,\y)+\partial\iota_\cY(\y))\leq \epsilon.
			\]
\end{enumerate}
\end{definition} 
Next, we introduce two  acceleration techniques for \emph{TS-GDA}  developed in~\citep{lin2025two}. 
\subsection{Dual Perturbation}\label{sec:prelim:perturb}
The idea of dual perturbation is to introduce a small quadratic regularizer of the dual variable to the objective function. This yields the following perturbed, nonconvex-strongly-concave (NC-SC) minimax problem
\begin{equation}
    \min_{\x \in \mathcal{X}}    \max_{\y \in \mathcal{Y}} f_t(\x,\y), ~\mbox{where}~ f_t(\x,\y)= f(\x,\y) -\frac{r_{\y}^t}{2}\|\y\|^2.
    \label{eq:perturb}
\end{equation} 
Throughout this paper, we let $r_{\y}^t >0$ be non-increasing in $t$.
This perturbation renders the objective strongly concave in $\y$ at each iteration, thereby enabling the use of standard techniques for NC-SC minimax problems.

Intuitively, if we choose $r_{\y}^t = \mathcal{O}(\epsilon)$ for all $t \ge 0$,
the perturbation term alters the dual gradient by at most $\mathcal{O}(\epsilon)$.
As a result, any $\epsilon$-GS point of~\cref{eq:prob} remains an $\mathcal{O}(\epsilon)$-GS point of~\cref{eq:perturb},
so that convergence guarantees established for the perturbed problem~\cref{eq:perturb}
translate to the original problem~\cref{eq:prob} up to an $\mathcal{O}(\epsilon)$ error.
Moreover, if the perturbation parameter $r_{\y}^t$ gradually diminishes to zero,
the iterates are expected to asymptotically recover stationarity points for the original problem~\cref{eq:prob}.

Motivated by these observations, the algorithms in \citep{xu2023unified,lu2020hybrid}
essentially apply vanilla \emph{GDA} to~\cref{eq:perturb}, which we formalize below as the \emph{Perturbed GDA} algorithm.

\begin{algorithm}[H]
		\caption{Perturbed GDA}\label{alg:gda}
		\KwData{Initial $\x^0,\y^0$, and step sizes $c^t,\alpha^t>0$}
		\For{$t=0,\ldots,T$}{
	     $\x^{t+1}=\proj_{\mathcal{X}}(\x^{t}-c^t\nabla_{\x}f_t(\x^{t},\y^{t}))$;\\
		 $\y^{t+1}=\proj_{\mathcal{Y}}(\y^{t}+\alpha^t \nabla_{\y}f_{t}(\x^{t+1},\y^{t}))$.
		}
	\end{algorithm} 
    
Since the perturbed objective $f_t$ is strongly concave in $\y$ at each iteration,
one may expect \emph{Perturbed GDA} to improve upon the
$\cO(\epsilon^{-6})$ iteration complexity of vanilla \emph{TS-GDA}
established in \citep{lin2020gradient,boct2020alternating}.
This is indeed the case: \citet{xu2023unified,lu2020hybrid} show that \emph{Perturbed GDA} finds an
$\epsilon$-GS of~\cref{eq:prob} in 
$\cO(\epsilon^{-4})$ iterations, demonstrating a clear improvement
in terms of GS.  
However, we note that such an improvement is only for $\epsilon$-GS, and the convergence rate of \emph{Perturbed GDA} with respect to
$\epsilon$-OS remains unexplored in the literature.

\subsection{Smoothing}\label{sec:prelim:smooth}
Unlike dual perturbation, the smoothing technique modifies the optimization landscape through a fundamentally different mechanism. 
It is based on the Moreau envelope.
For a given function $\Phi$ and any $r<\tfrac{1}{\ell}$, its Moreau envelope $\Phi_r$ is defined as
\[
  \Phi_{r}(\x)
  \coloneqq \min_{\z\in\R^n}\Bigl\{\Phi(\z) + \tfrac1{2r}\|\x - \z\|^2\Bigr\}.
\]
The Moreau envelope $\Phi_r(\x)$ is not merely a smooth approximation of $\Phi$, but also a surrogate of the original problem: Its minimizers coincide with those of $\Phi$, and its gradient provides a natural measure of stationarity for $\Phi$~\citep{davis2019stochastic}. The smoothness of $\Phi_r(\x)$ further makes it a convenient and well-structured surrogate for algorithm design.
Using this idea, \citet{zhang2020single} introduce the smoothed surrogate function
\[
  F(\x,\y,\z) := f(\x,\y) + \frac{r_{\x}}{2}\|\x - \z\|^2,
\]
where $r_{\x}>0$ is a smoothing parameter. 
In their algorithm, they update $(\x,\z)$ by gradient descent and $\y$ by gradient ascent with
respect to $F$.  
This algorithm is formally presented in the following \emph{Smoothed GDA.}

\begin{algorithm}[H]
		\caption{Smoothed GDA} \label{alg:smoothed}
		\KwData{Initial point $(\x^0,\y^0,\z^0)$, and parameters $r_{\x}>\ell$, $c^t>0$, $\alpha^t>0$, $\beta^t\in(0,1)$}
		\For{$t=0,\ldots,T$}{
	     $\x^{t+1}=\proj_{\mathcal{X}}(\x^{t}-c^t\nabla_{\x}F(\x^{t},\y^{t},\z^{t}))$;\\
		 $\y^{t+1}=\proj_{\mathcal{Y}}(\y^{t}+\alpha^t \nabla_{\y}F(\x^{t+1},\y^{t},\z^{t}))$;\\
        $\z^{t+1}=\z^{t}+\beta^t(\x^{t+1}-\z^{t})$;
		}
\end{algorithm} 
    
Recently, \citet{li2025nonsmooth} further show that \emph{Smoothed GDA} achieves both $\epsilon$-GS and $\epsilon$-OS within $\mathcal{O}(\epsilon^{-4})$ iterations. The intuition for such an improvement is as follows:
Note that the definition of {Moreau envelope} dovetails naturally with the definition of $\epsilon$-OS in \cref{defi:primal-dual}.  
This structure enables \emph{Smoothed GDA} to converge to an $\epsilon$-OS faster. 
In other words, the acceleration due to smoothing for $\epsilon$-OS is ``direct". On the other hand, 
 because any $\epsilon$-OS can be converted into an
$\Theta(\epsilon)$-GS in $\mathcal{O}(\epsilon^{-2})$ iterations
(cf. \citep[Proposition~4.13]{lin2020gradient}), the acceleration induced by smoothing on $\epsilon$-OS also carries over to $\epsilon$-GS.
In other words, smoothing improves the convergence rate to an $\epsilon$-GS in an ``indirect" manner.

\section{Main Results}\label{sec:main}
This section presents our main algorithmic and theoretical results. We first propose a new method, \emph{Perturbed Smoothed GDA}, which combines dual perturbation and smoothing in a single-loop first-order algorithm. We then establish iteration complexity bounds for computing $\epsilon$-GS and $\epsilon$-OS solutions across several representative single-loop methods, and provide matching lower bounds showing that these rates are tight. 
Our analysis makes explicit how the effect of smoothing and perturbation interact and explains why their combination leads to an improved convergence guarantee to $\epsilon$-GS for NC–C minimax optimization. 
Finally, we establish asymptotic convergence to 0-GS for both \emph{Perturbed GDA} and \emph{Perturbed Smoothed GDA}, a guarantee that is not available for vanilla \emph{Smoothed GDA} under existing analyses.

\subsection{Perturbed Smoothed GDA}\label{subsec:ps-gda}
Similar to previous methods, our algorithm also starts with a surrogate function.
Specifically, at each iteration $t$, we combine dual perturbation and smoothing and construct the following surrogate function:
\begin{equation}\label{eq:tilde-F}
F_t(\x,\y,\z) := f(\x,\y) + \underbrace{\frac{r_{\x}}{2}\|\x-\z\|^2}_{\text{Smoothing}}\underbrace{-\frac{r_{\y}^t}{2}\|\y\|^2}_{\text{Dual perturbation}}.
\end{equation}
We apply gradient descent on variables $\bx$ and $\z$, and gradient ascent on variable $\y$ in each iteration. This update rule leads to the following \emph{Perturbed Smoothed GDA} algorithm. 

\begin{algorithm}[H]
		\caption{Perturbed Smoothed GDA}\label{alg:sgda}
		\KwData{Initial point $(\x^0,\y^0,\z^0)$, and parameters $r_{\x}>\ell,r_{\y}^t>0$, $c^t> 0$, $\alpha^t>0$, $\beta^t\in(0,1)$}
		\For{$t=0,\ldots,T$}{
	     $\x^{t+1}=\proj_{\mathcal{X}}(\x^{t}-c^t\nabla_{\x}F_t(\x^{t},\y^{t},\z^{t}))$;\\
		 $\y^{t+1}=\proj_{\mathcal{Y}}(\y^{t}+\alpha^t \nabla_{\y}F_t(\x^{t+1},\y^{t},\z^{t}))$;\\
        $\z^{t+1}=\z^{t}+\beta^t(\x^{t+1}-\z^{t})$;
		}
\end{algorithm}

\subsection{Convergence Results and their Tight Analysis}
\label{subsec:iter-complex}

Before presenting our main iteration‐complexity theorems, we introduce two preliminaries that facilitate our subsequent analysis. First, in \cref{def:initial_gap} we formalize the initial gap for each algorithmic family, namely, \emph{Perturbed GDA}, \emph{Smoothed GDA} and \emph{Perturbed Smoothed GDA}. Second, we specify the step‐size conditions which ensure basic descent estimates (See \cref{sec:framework} for details). These technical bounds are provided for completeness and can be safely skipped in a first reading.

\begin{definition}[Initial gaps]
\label{def:initial_gap}
We define two initial gaps, i.e., $\Delta_{\Psi_1}$ and $\Delta_{\Psi_2}$ as below: 
\begin{enumerate}[label=(\roman*)] 
    \item 
    The first gap is defined as
    \[
    \Delta_{\Psi_1}:={\Psi_1^0}(\x^0,\y^0)-\min_{\x\in\cX}\max_{\y\in\cY}f_0(\x,\y),\]
    where ${\Psi_1^0}(\x,\y):=2\max_{\y\in\cY}f_0(\x,\y)-f_0(\x,\y)$, and $f_0(\cdot,\cdot)$ is defined in \cref{eq:perturb}.
    \item The second gap is defined as \[\Delta_{\Psi_2}:={\Psi_2^0}(\x^0,\y^0,\z^0)-\min_{\x\in\cX}\max_{\y\in\cY}f_0(\x,\y),\] where ${\Psi_2^0}(\x,\y,\z) := F_0(\x,\y,\z) - 2\min_{\x\in\cX}F_0(\x,\y,\z) + 2\min_{\x\in\cX}\max_{\y\in\cY}F_0(\x,\y,\z)$, and $F_0(\cdot,\cdot,\cdot)$ is defined in \cref{eq:tilde-F}.
\end{enumerate} 
\end{definition}
\begin{remark}
By construction, $\Delta_{\Psi_1}$ measures the initial gap associated with
\emph{Perturbed GDA}, and $\Delta_{\Psi_2}$ measures the initial gap for both
\emph{Smoothed GDA} and \emph{Perturbed Smoothed GDA}.  
For \emph{Smoothed GDA}, $\Delta_{\Psi_2}$ is defined by setting $r_{\y}^0 = 0$ in $F_0$.
\end{remark}
Next, we summarize the step‐size conditions for different algorithms: 

\begin{condition}[Step-size conditions for Perturbed GDA]\label{con-pgda}
Let $r_{\y}^t > 0$ for all $t \ge 0$. The step sizes $\{\alpha^t\}_{t\ge 0}$ and $\{c^t\}_{t\ge 0}$ are chosen such that, for all $t$, we have
\begin{align*}
       & \alpha^t=\Theta\left(\frac{1}{\ell}\right) \quad ~\mbox{and}~ \quad0<\alpha^t \leq \frac{1}{4(\ell+r_{\y}^t)}; \\
       & c^t=\Theta\left(\frac{(r_{\y}^t)^2}{\ell^3}\right)  \quad~\mbox{and}~ \quad0 <c^t\leq \min\left\{\frac{(r_{\y}^t)^2\alpha^t}{16 \ell^2},\ \frac{r_{\y}^t}{\ell(3r_{\y}^t+2\ell)},\ \frac{3r_{\y}^t}{128\ell^2}\right\}.
\end{align*}

\end{condition}

\begin{condition}[Step-size conditions for (Perturbed) Smoothed GDA]
\label{con-psgda}
    Let $r_{\x} = \Theta(\ell)$ with $r_{\x} > 3\ell$. 
The step sizes $\{\alpha^t\}_{t\ge0}$, $\{c^t\}_{t\ge0}$, and $\{\beta^t\}_{t\ge0}$ are chosen such that, for all $t \ge 0$, 
    \begin{align*}
    & \alpha^t=\Theta\left(\frac{1}{\ell}\right)\quad ~\mbox{and}~ \quad0<\alpha^t \leq \min\left\{\frac{1}{11(\ell+r_{\y}^t)},\ \frac{(c^t)^2(r_{\x}-\ell-r_{\y}^t)^2}{4(\ell+r_{\y}^t)(1+c^t(r_{\x}-\ell-r_{\y}^t))^2} \right\};\\
    & c^t=\Theta\left(\frac{1}{\ell}\right)\quad ~\mbox{and}~ \quad0<c^t<\frac{1}{r_{\x}+\ell+r_{\y}^t}; \\
    &0<\beta^t  \leq \min \left\{\frac{1}{36},\ \frac{(r_{\x}-\ell-r_{\y}^t)^2}{384r_{\x}(r_{\x}+\ell+r_{\y}^t)^2}\right\}.
    \end{align*}
Moreover, the parameter $\beta^t$ and $r_{\y}^t$  are chosen differently for \emph{Smoothed GDA} and \emph{Perturbed Smoothed GDA}:
    \begin{enumerate}[label=(\roman*)] 
\setlength{\itemsep}{5pt} 
\item \emph{Smoothed GDA}: For all $t\ge0$,\[
r_{\y}^t=0\quad \text{and} \quad \beta^t=\Theta\left(\frac{\epsilon^2}{\ell^2 D_{\cY}^2}\right).\]
\item \emph{Perturbed Smoothed GDA}: Let $r_{\y}^t > 0$ for all $t \ge 0$, and define
\[
\omega^t := \frac{1}{\sqrt{r_{\y}^t}} \cdot 
\frac{(r_{\x}-\ell-r_{\y}^t)+\alpha^t(\ell+r_{\y}^t)(3r_{\x}-2\ell-2r_{\y}^t)}
{\alpha^t (r_{\x}-\ell-r_{\y}^t)^{3/2}}.
\]
Then
\[
\beta^t = \Theta\!\left(\frac{r_{\y}^t}{\ell}\right),
\quad \text{and} \quad
0 < \beta^t \le \frac{1}{384\,r_{\x}\alpha^t(\omega^t)^2}.
\]
\end{enumerate}

\end{condition}

With these step-size conditions in place, we present our main iteration-complexity results under a constant choice $r_{\y}^t$.
We choose this setting for two reasons: First, our lower bounds (\cref{thm:tight-alter,thm:tight-psgda}) are established under a constant $r_{\y}^t$. Stating the upper bounds under the same setting yields a direct tightness guarantee. 
Second, the core ideas of our analytical framework for iteration complexity analysis are most cleanly shown when $r_{\y}^t$ is constant.

\begin{theorem}[Iteration complexity of Perturbed GDA]    
\label{thm:pgda-nc-c-epi} 
Let the sequence $\{(\x^{t},\y^{t})\}_{t\ge 0}$ be generated by \emph{Perturbed GDA} with step sizes satisfying \cref{con-pgda}.
Given any $\epsilon>0$, we have 
\begin{enumerate}[label=(\roman*)]
\item (Game Stationarity)
Let $r_{\y}:= \Theta\left(\tfrac{\epsilon}{D_{\cY}}\right)$ and set $r_{\y}^t=r_{\y}$ for all $t \ge 0$,
then there exists an index $t \leq \cO\left(\tfrac{\ell^3D_{\cY}^2\Delta_{\Psi_1}}{\epsilon^4}\right)$ such that $(\x^{t+1},\y^{t+1})$ is an $\epsilon$-GS of problem \cref{eq:prob}. 
 \item
 (Optimization Stationarity) 
Let $r_{\y}:= \Theta\left(\tfrac{\epsilon^2}{\ell D_{\cY}^2}\right)$ and set $r_{\y}^t=r_{\y}$ for all $t \ge 0$,
then there exists an index $t\leq \cO\left(\tfrac{\ell^5D_{\cY}^4\Delta_{\Psi_1}}{\epsilon^6}\right)$ such that $\x^{t+1}$ is an $\epsilon$-OS of problem \cref{eq:prob}. 
\end{enumerate}
\end{theorem}

\begin{theorem}[Tightness analysis of  Perturbed GDA] \label{thm:tight-alter}
Consider \emph{Perturbed GDA} run with the step-size conditions in
\cref{thm:pgda-nc-c-epi}.  Let $(\x^{T},\y^{T})$ be the output after $T$ iterations.
Then there exist an $\ell$-smooth function and an initialization $(\x^0,\y^0)$
for which the following hold:
\begin{enumerate}[label=(\roman*)] 
\item (Game Stationarity)  \emph{Perturbed GDA} 
requires at least $T= \Omega\left(\tfrac{\ell^3D_{\cY}^2\Delta_{\Psi_1}}{\epsilon^4}\right)$ 
iterations to guarantee that the point $\left(\x^{T},\y^{T}\right)$ is an $\epsilon$-GS of  problem \cref{eq:prob}. 
\item (Optimization Stationarity) \emph{Perturbed GDA} 
requires at least $T = \Omega\left(\tfrac{\ell^5D_{\cY}^4\Delta_{\Psi_1}}{\epsilon^6}\right)$ 
iterations to guarantee that the point $\x^{T}$ is an $\epsilon$-OS of  problem \cref{eq:prob}.
\end{enumerate}
\end{theorem}

\begin{remark}
 (i) A result analogous to \cref{thm:pgda-nc-c-epi}~(i) has been obtained in \citep{lu2020hybrid,xu2023unified}. However, our proof relies on a different analytical framework (see \cref{sec:framework});
 (ii) As noted in the introduction, prior works do not provide iteration complexity guarantees for $\eps$-OS. In \cref{thm:pgda-nc-c-epi}~(ii), we establish the first iteration complexity bound in terms of $\eps$-OS, which matches the iteration complexity of vanilla \emph{TS-GDA}. 
 Together with our tight analysis in \cref{thm:tight-alter}~(ii), these results show that while dual perturbation accelerates the convergence to $\epsilon$-GS, it does not improve the convergence to $\epsilon$-OS. 
In other words, \cref{thm:pgda-nc-c-epi,thm:tight-alter} show that the convergence performance of \emph{Perturbed GDA} on $\epsilon$-GS and $\epsilon$-OS is totally different. 
This observation clarifies that the benefits of dual perturbation are limited to GS and do not extend to OS in single-loop first-order methods for NC–C minimax optimization.
\end{remark}

\begin{theorem}[Iteration complexity of {\it \bf{Perturbed Smoothed GDA}}]
\label{thm:psgda-epi}

Let the sequence $\{(\x^{t},\y^{t},\z^{t})\}_{t\ge 0}$ be generated by \emph{Perturbed Smoothed GDA}. Given any $\epsilon>0$, we let step sizes satisfy \cref{con-psgda}. Then, we have 
\begin{enumerate}[label=(\roman*)] 

\item (Game Stationarity)  
Let $r_{\y}:= \Theta\left(\tfrac{\epsilon}{D_{\cY}}\right)$ and set $r_{\y}^t=r_{\y}$ for all $t \ge 0$,
then there exists an index $t \leq \cO\left(\tfrac{\ell^2D_{\cY}\Delta_{\Psi_2}}{\epsilon^3}\right)$ such that $(\x^{t},\y^{t})$ is an $\epsilon$-GS of problem \cref{eq:prob}. 
\item (Optimization Stationarity) 
Let $r_{\y}:= \Theta\left(\tfrac{\epsilon^2}{\ell D_{\cY}^2}\right)$ and set $r_{\y}^t=r_{\y}$ for all $t \ge 0$,  then there exists an index $t\leq \cO\left(\tfrac{\ell^3D_{\cY}^2\Delta_{\Psi_2}}{ \epsilon^4}\right)$ such that $\x^{t}$ is an $\epsilon$-OS of problem \cref{eq:prob}.
\end{enumerate}
\end{theorem}

\begin{remark}
(i) When $r_{\y}^t=0$ for all $t$, \emph{Perturbed Smoothed GDA} reduces to \emph{Smoothed GDA}. 
For completeness, we state the iteration complexity of \emph{Smoothed GDA} in \cref{thm:sgda-epi}. 
For the OS guarantee, our proof differs from that of \citep{li2025nonsmooth}. 
Moreover, due to the additional quadratic perturbation, the analysis of \emph{Perturbed Smoothed GDA} is substantially different from that of \emph{Smoothed GDA}. 
In particular, by exploiting the condition $r_{\y}^t>0$, we avoid the two-stage analysis used in \citep{zhang2020single,li2025nonsmooth};
(ii) In terms of $\epsilon$-OS, \emph{Smoothed GDA} and \emph{Perturbed Smoothed GDA} 
achieve the same iteration complexity; 
(iii) In terms of $\epsilon$-GS, the two methods differ in the magnitude of the
step size~$\beta^t$.  
For \emph{Smoothed GDA}, one needs $\beta^t = \Theta\left(\tfrac{\epsilon^{2}}{\ell D_{\cY}^{2}}\right)$, whereas for \emph{Perturbed Smoothed GDA}, it suffices to take
$\beta^t = \Theta\left(\tfrac{\epsilon}{\ell D_{\cY}}\right)$.
This more relaxed requirement on~$\beta^t$ enables \emph{Perturbed Smoothed GDA} to
achieve a faster convergence to
game stationarity compared with \emph{Smoothed GDA}. 

\end{remark}

\begin{theorem}[Tightness analysis of  Perturbed Smoothed GDA] \label{thm:tight-psgda}
Consider the \emph{Perturbed Smoothed GDA}  with  step‐size conditions in \cref{thm:psgda-epi}. 
Let $(\x^{T},\y^{T},\z^{T})$ denote the output after $T$ iterations.
Then there exists an $\ell$-smooth function and an initialization $(\x^0,\y^0,\z^0)$ such that 
\begin{enumerate}[label=(\roman*)] 
\item (Game Stationarity) \emph{Perturbed Smoothed GDA}  
requires at least $T = \Omega\left(\tfrac{\ell^2D_{\cY}\Delta_{\Psi_2}}{\epsilon^3}\right)$
iterations to guarantee that the point $(\x^{T},\y^{T})$ is an $\epsilon$-GS of problem \cref{eq:prob}. 
\item (Optimization Stationarity) \emph{Perturbed Smoothed GDA} 
requires at least $T = \Omega\left(\tfrac{ \ell^3D_{\cY}^2\Delta_{\Psi_2}}{\epsilon^4}\right)$ iterations to guarantee that the point $\x^{T}$ is an $\epsilon$-OS of  problem \cref{eq:prob}.
\end{enumerate}
\end{theorem}
Using similar instances as in \cref{thm:tight-psgda}, we can also establish the following tight complexity result for \emph{Smoothed GDA}.

\begin{corollary}[Tightness analysis of Smoothed GDA]
    \label{thm:sgda-epi-tight}
Consider the \emph{Smoothed GDA} with the step‐size conditions from 
\cref{con-psgda}, and denote by $(\x^{T},\y^{T},\z^{T})$ the output after $T$ iterations.
Then, there exists an $\ell$-smooth function and an initialization $(\x^0,\y^0,\z^0)$ such that 
\begin{enumerate}[label=(\roman*)] 
\item (Game Stationarity)  \emph{Smoothed GDA} 
requires at least $T = \Omega\left(\tfrac{\ell^{3}D_{\cY}^{2}\Delta_{\Psi_2}}{\epsilon^{4}}\right)$ 
iterations to guarantee that the point $(\x^{T},\y^{T})$ is an $\epsilon$-GS of  problem \cref{eq:prob}.
\item (Optimization Stationarity) \emph{Smoothed GDA} 
requires at least $T = \Omega\left(\tfrac{\ell^3D_{\cY}^{2}\Delta_{\Psi_2}}{\epsilon^{4}}\right)$ 
iterations to guarantee that the point $\x^{T}$ is an $\epsilon$-OS of  problem \cref{eq:prob}. 
\end{enumerate}
\end{corollary}
These results clearly delineate the roles of the two acceleration techniques: Smoothing alone improves convergence to both optimization stationarity and game stationarity, whereas dual perturbation alone accelerates only convergence to game stationarity. 
By combining the two techniques, \emph{Perturbed Smoothed GDA} attains the best-known rate for game stationarity while preserving the state-of-the-art rate for optimization stationarity among single-loop first-order methods. In other words, smoothing suffices to achieve the state-of-the-art rate for optimization stationarity, while dual perturbation contributes only to further accelerating game stationarity when used alongside smoothing. 
Moreover, this cooperative speed-up for game stationarity suggests that smoothing and dual perturbation enhance convergence to $\epsilon$-GS through distinct, yet complementary, mechanisms.

While the above results focus on a constant (non-adaptive) choice of $r_{\y}^t$ for finite-time guarantees, 
we next consider a diminishing sequence $r_{\y}^t \searrow 0$ for both \emph{Perturbed GDA} and \emph{Perturbed Smoothed GDA}. 
We show that the vanishing perturbation further ensures asymptotic convergence to a stationary point of~\cref{eq:prob}.
\begin{theorem}[Asymptotic Convergence of  {\it \bf{Perturbed  GDA}}]
\label{thm:pgda-asm}
 Let $\{(\x^t,\y^t)\}_{t\ge 0}$ be generated by \emph{Perturbed GDA} with step sizes satisfying \cref{con-pgda}. Define 
 \[
\cI^\star(T)
= \bigcup_{t=1}^T 
\argmin_{k \in \{\lfloor t/2 \rfloor,\dots,t\}} 
\left\{
\frac{1}{4\alpha^k}\|\y^{k+1}-\y^k\|^2
+
\frac{1}{16c^k}\|\x^{k+1}-\x^k\|^2
\right\}.
\]
When we choose $r_{\y}^t \searrow  0$ and $t\cdot (r_{\y}^t)^2 \to \infty$, any limit point of $\{(\x^t,\y^t)\}_{t\in \cI^\star(\infty)}$ is a $0$-GS of \cref{eq:prob}. 
\end{theorem}
\begin{theorem}[Asymptotic Convergence of {\it \bf{Perturbed Smoothed GDA}}]
\label{thm:psgda-asm}
Let the sequence $\{(\x^t,\y^t,\z^t)\}_{t\ge 0}$ be generated by \emph{Perturbed Smoothed GDA}. 
Define 
\[
\cI^\star(T)
= \bigcup_{t=1}^T 
\argmin_{k \in \{\lfloor t/2 \rfloor,\dots,t\}} 
\!\left\{
\frac{1}{8c^k}\|\x^{k+1}-\x^{k}\|^{2}\!+\!\frac{1}{16\alpha^k}\! \|\y^{k}-{\y}^{k}_+(\z^{k})\|^2\! +\frac{r_{\x}\beta^k}{8}\|\z^{k}-\x^{k+1}\|^{2}
\right\},
\]
where ${\y}^{k}_+(\z):=\proj_{\mathcal{Y}}\left(\y^k+\alpha^k \nabla_{\y} {F_k}({\x}_{k}(\y^k, \z), \y^k,\z)\right)$ and ${\x}_{k}(\y, \z):=\argmin\limits_{\x \in \mathcal{X}}{F_k}(\x,\y,\z)$. When we choose $r_{\y}^t \searrow  0$ and $t\cdot r_{\y}^t \to \infty$, any limit point of $\{(\x^t,\y^t)\}_{t\in \cI^\star(\infty)}$ is a $0$-GS of \cref{eq:prob}. 
\end{theorem}
\begin{remark}
(i) An interesting point is that the original \emph{smoothed GDA} does not admit an asymptotic convergence guarantee under the current analysis. 
As can be seen from the proof in \cite{zhang2020single}, the two-stage argument introduces a fundamental difficulty: we cannot rule out the possibility that the iterates enter the unfavorable case infinitely often. Consequently, the Lyapunov function may fail to satisfy a sufficient descent property not only along the whole sequence, but also along any subsequence extracted from it. (ii) Although our non-asymptotic convergence results are stated for a non-adaptive choice of $r_{\y}^t$, the analysis in \cref{sec:framework} extends directly to adaptive choices of $r_{\y}^t$. 
In particular, the same finite-time convergence rates continue to hold, provided that the adaptive sequence satisfies the parameter conditions required in the proof.
\end{remark}

\section{A Unified Convergence Analysis}\label{sec:framework}
In this section, we establish iteration complexity guarantees for both \emph{Perturbed GDA} 
(see \cref{thm:pgda-nc-c-epi}) and \emph{Perturbed Smoothed GDA} 
(see \cref{thm:psgda-epi}) under two stationarity notions: optimization stationarity and game stationarity. 
To this end, we develop a unified analytical framework that encompasses both algorithms and serves as the foundation for the subsequent convergence analyses.

Our framework contains the following two key components:
\begin{enumerate}[label=(\roman*)] 
    \item The construction of a \emph{Lyapunov function} that captures both the primal descent and dual ascent dynamics. 
    \item  Derive a \emph{primal-dual balancing inequality} that characterizes the interaction between the primal and dual updates, which ensures a \emph{descent property} of the Lyapunov function.
\end{enumerate}
Together, these two components establish the iteration complexity guarantees for single-loop algorithms. Furthermore, this analytical framework can apply to the double-loop algorithm (see \cref{sec:double}) as well.
Before proceeding to the formal proof, we summarize in \cref{tab:notation} the main notations used in this section and throughout the paper.

\begin{table}[h]
  \caption{Notation}
 \centering
\begin{tikzpicture}
 \node[drop shadow,fill=white,inner sep=0pt] 
 {
 \resizebox{0.98\textwidth}{!}{ 
\begin{tabular}{|c|c|c|}
 \hline
 \textbf{Notation}  & \textbf{Definition} & \textbf{Notes} \\ \hline 
${d}_t(\y,\z)$           & $\min\limits_{\x \in \mathcal{X}}{F_t}(\x,\y,\z)$ &   dual function      \\ \hline
${p_t}(\z)$           & $\min\limits_{\x \in \mathcal{X}}\max\limits_{\y \in \mathcal{Y}}{F_t}(\x,\y,\z)$ &  proximal function         \\ \hline
${\x}_{t}(\y,\z)$  & $\mathop{\argmin}\limits_{\x \in \mathcal{X}}{F_t}(\x,\y,\z)$  & ---\\ \hline
${\x}^\star_{t}(\z)$   & $\mathop{\argmin}\limits_{\x \in \mathcal{X}}\max\limits_{\y \in \mathcal{Y}}{F_t}(\x,\y,\z)$ & ---               \\ \hline
$\y_{t}^\star(\x)$ & $ \mathop{\argmax}\limits_{\y \in \mathcal{Y}} {f_t}(\x,\y) $& ---\\ \hline
${\y_t}(\z)$     & $\mathop{\argmax}\limits_{\y\in \mathcal{Y}}{d_t}(\y,\z)$ &   ---  \\ \hline
$\y_{+}^t(\z)$    & $\proj_{\mathcal{Y}}\left(\y^t+\alpha^t\nabla_{\y} {F_t}({\x}_{t}(\y^t, \z), \y^t,\z)\right)$  & one-step projected gradient ascent on the dual function\\ \hline
\end{tabular}}
};
\end{tikzpicture}
\label{tab:notation}
\end{table}

\subsection{Convergence Analysis of Perturbed GDA}\label{subsec:Perturb-ub}
In this subsection, we prove \cref{thm:pgda-nc-c-epi,thm:pgda-asm}.
Although $r_{\y}^t$ is held constant in \cref{thm:pgda-nc-c-epi}, we still keep the superscript $t$. 
This choice of presentation is to provide the most general form of our analytical framework.

Our proof is based on the following Lyapunov function at each iteration $t$:
\vspace{1mm}
\begin{equation}
\label{eq:potential}
{\Psi_1^t}(\x,\y) := 2 \Phi_t(\x) -f_t(\x,\y) = \underbrace{\Phi_t(\x)}_{\textrm{Primal descent}} + \underbrace{\Phi_t(\x) -f_t(\x,\y)}_{\textrm{Dual ascent}},
\end{equation}
where $\Phi_t(\x):=\max_{\y\in\mathcal{Y}} f_t(\x,\y)+\iota_{\cX}(\x)$.

Note that ${\Psi_1^t}(\x,\y)$ consists of two parts: $\Phi_t(\x)$ measures how well $\x$ is minimizing the objective function and $\Phi_t(\x) -f_t(\x,\y) $ measures how far $\y$ is from the maximizer of the inner maximization problem of \eqref{eq:perturb} given $\x$. Therefore, if the algorithm proceeds properly, the value of the function ${\Psi_1^t}(\x,\y) $ should steadily decrease. 
The crux of our proof lies in establishing the descent of this Lyapunov function. In the following lemma, we give a basic descent estimate on the Lyapunov function in \cref{eq:potential}. For simplicity, we write $\Psi_{1}^t:=\Psi_1^t(\x^{t},\y^{t})$.


\begin{lemma}[Basic descent estimate]\label{lemma:perturb-potentialfunc-decrease}
Suppose \cref{con-pgda} holds, and let $\{(\x^{t},\y^{t})\}_{t \ge 0}$ be the sequence generated by \emph{Perturbed GDA}. Then, 
for any $t\ge 0$, we have 
\begin{align*}
  & \, {\Psi_{1}^{t+1}}-{\Psi_{1}^t} \notag\\
  \leq & \, 2\ell\|{\y}_t^\star(\x^{t})-\y^{t}\|\!\cdot\!\|\x^{t+1}-\x^{t}\|\!-\!\frac{1}{2\alpha^t}\|\y^{t+1}-\y^{t}\|^2\!-\!\frac{1}{2c^{t}}\|\x^{t+1}-\x^{t}\|^2\!+\left(r_{\y}^t-r_{\y}^{t+1}\right)D_{\cY}^2. 
\end{align*}
\end{lemma}
\begin{proof}[Proof of \cref{lemma:perturb-potentialfunc-decrease}.] 
      
    First of all, we quantify the basic descent of ${\Phi}_t$ as follows: 
    \begin{equation}\label{eq:ll1-eq1}
        {\Phi}_{t+1}(\x^{t+1})-{\Phi}_t(\x^{t})=\left({\Phi}_{t+1}(\x^{t+1})-{\Phi}_t(\x^{t+1})\right)+\left({\Phi}_{t}(\x^{t+1})-{\Phi}_t(\x^{t})\right).
    \end{equation}
    For the first term, observe that 
    \[{f}_{t+1}(\x,\y)={f}_{t}(\x,\y)+\frac{r_{\y}^t-r_{\y}^{t+1}}{2}\|\y\|^2. \]
    Hence, 
    \begin{align}
{\Phi}_{t+1}(\x^{t+1})
&=\max_{\y\in\cY} f_{t+1}(\x^{t+1},\y)\notag\\
&\le \max_{\y\in\cY} f_t(\x^{t+1},\y)
+ \max_{\y\in\cY} \frac{r_{\y}^t-r_{\y}^{t+1}}{2}\|\y\|^2 \notag\\
&= {\Phi}_t(\x^{t+1}) + \frac{r_{\y}^t-r_{\y}^{t+1}}{2} D_{\cY}^2.
\label{eq:phi-t}
\end{align}
For the second term in \cref{eq:ll1-eq1}, we have
    \begin{align}
         & \, {\Phi}_{t}(\x^{t+1})-{\Phi}_t(\x^{t})  \notag\\ 
         \leq  & \,  \left\langle \nabla {\Phi}_t              (\x^{t}), \x^{t+1}-\x^{t}\right\rangle + \left(\frac{\ell}{2}+\frac{\ell^2}{2r_{\y}^t}\right)\|\x^{t+1}-\x^{t}\|^2 \label{eq:phi2}\\
         \leq & \, \left\langle \nabla {\Phi}_t(\x^{t}) \!-\!\nabla_{\!\x} {f}_t
         (\x^{t},\y^{t}), \x^{t+1}\!-\!\x^{t}\right\rangle \!+\!\left\langle \nabla_{\!\x}{f}_t
         (\x^{t},\y^{t}), \x^{t+1}\!-\!\x^{t}\right\rangle\! +\! \left(\!\frac{\ell}{2}\!+\!\frac{\ell^2}{2r_{\y}^t}\!\right)\! \|\x^{t+1}\!-\!\x^{t}\|^2 \notag \\
        \leq & \, \ell \|\y_{t}^\star(\x^{t}) - \y^{t}\|\cdot\|\x^{t+1}-\x^{t}\| -\frac{1}{c^{t}}\|\x^{t+1} -\x^{t}\|^2+ \left(\frac{\ell}{2}+\frac{\ell^2}{2r_{\y}^t}\right) \|\x^{t+1}-\x^{t}\|^2 \notag  \\
         \leq & \,\ell \|\y_{t}^\star(\x^{t}) - \y^{t}\|\cdot\|\x^{t+1}-\x^{t}\| -\left(\frac{1}{c^{t}}-\frac{\ell}{2}-\frac{\ell^2}{2r_{\y}^t}\right)\|\x^{t+1}-\x^{t}\|^2 \label{eq:phi3}, 
    \end{align}
where the first inequality is due to the $(\ell+\frac{\ell^2}{r_{\y}^t})$-smoothness property of ${\Phi}_t$~\citep[Lemma 4.3]{lin2020gradient}, and the third inequality follows from the $\ell$-Lipschitz continuity of $\nabla_{\x}{f}_t(\cdot,\y)$, the Cauchy-Schwarz inequality, and the characterization of projections onto a closed convex set, i.e., 
$
\langle\x^{t} -\x^{t+1}, \x^{t}-c^{t}\nabla_{\x} {f_t}(\x^{t},\y^{t}) - \x^{t+1}\rangle \leq  0. 
$


Now we focus on the dual ascent part. First note that
\begin{align*}
    &\,\left({\Phi}_{t+1}(\x^{t+1})-{f}_{t+1}(\x^{t+1},\y^{t+1})\right)-\left({\Phi}_t(\x^{t})-{f}_t(\x^{t},\y^{t})\right)\\
    =&\,\left({\Phi}_{t+1}(\x^{t+1})\!-\!{\Phi}_t(\x^{t})\!\right)\!+\!\left({f}_{t}(\x^{t+1}\!,\y^{t+1})\!-\!{f}_{t+1}(\x^{t+1}\!,\y^{t+1})\!\right)\!+\!\left({f}_{t}(\x^{t+1}\!,\y^{t})\!-\!{f}_{t}(\x^{t+1}\!,\y^{t+1})\!\right)\!\\
    &\,+\left({f}_t(\x^{t},\y^{t})-{f}_t(\x^{t+1},\y^{t})\right).
\end{align*}
Because the decrease in ${\Phi}_t$ 
 has been quantified, we study the other three terms ${f}_{t}(\x^{t+1},\y^{t+1})-{f}_{t+1}(\x^{t+1},\y^{t+1})$, ${f}_{t}(\x^{t+1},\y^{t})-{f}_{t}(\x^{t+1},\y^{t+1})$ and ${f}_t(\x^{t},\y^{t})-{f}_t(\x^{t+1},\y^{t})$ in the following.
 
Since $r_{\y}^t$ is non-increasing, we have
\begin{align}
     {f}_{t}(\x^{t+1},\y^{t+1})-{f}_{t+1}(\x^{t+1},\y^{t+1})=\frac{r_{\y}^{t+1}-r_{\y}^t}{2}\|\y^{t+1}\|^2
     &\le0.\label{eq:def-per}
 \end{align}
Moreover, we have
\begin{align}
   & \, {f}_{t}(\x^{t+1},\y^{t})-{f}_{t}(\x^{t+1},\y^{t+1}) \notag\\ 
 \leq & \,   \left\langle  \nabla_{\y}{f}_t(\x^{t+1},\y^{t}), \y^{t}-\y^{t+1}\right\rangle + \frac{\ell+r_{\y}^t}{2}\|\y^{t+1}-\y^{t}\|^2 \notag\\ 
 \leq & \,  -\frac{1}{\alpha^t}\|\y^{t+1}-\y^{t}\|^2+\frac{\ell+r_{\y}^t}{2}\|\y^{t+1}-\y^{t}\|^2 \leq -\frac{1}{2\alpha^t}\|\y^{t+1}-\y^{t}\|^2, \label{eq:phi4}
\end{align}
where the first inequality stems from the $(\ell+r_{\y}^t)$-Lipschitz continuity of $\nabla_{\y}{f}_t(\x,\cdot)$, the second inequality is due to  the characterization of projections onto a closed convex set $\cY$, and the last one arises from \cref{con-pgda}, i.e., $0<\alpha^t< \frac{1}{4(\ell+r_{\y}^t)}$. 
Similarly, we obtain
\begin{align}
    & \, {f}_t(\x^{t},\y^{t})-{f}_t(\x^{t+1},\y^{t})\le \left\langle \nabla_{\x} {f}_t(\x^{t},\y^{t}), \x^{t}-\x^{t+1}\right\rangle+ \frac{\ell}{2}\|\x^{t+1}-\x^{t}\|^2 \label{eq:phi5}, 
\end{align}
where the inequality is by the $\ell$-Lipschitz continuity of $\nabla_{\x}{f}_t(\cdot,\y)$.

By summing \cref{eq:phi-t}, \cref{eq:phi2}, \cref{eq:def-per}, \cref{eq:phi4}, and \cref{eq:phi5}, we obtain a bound that controls the dual-ascent component in the Lyapunov function \cref{eq:potential}: 
\begin{align}
   & \,\left({\Phi}_{t+1}(\x^{t+1})-{f}_{t+1}(\x^{t+1},\y^{t+1})\right) -\left({\Phi}_t(\x^{t})-{f}_t(\x^{t},\y^{t})\right)   \notag \\
   \leq & \, \left\langle \nabla {\Phi}_t(\x^{t}) -\nabla_{\x} {f}_t
         (\x^{t},\y^{t}), \x^{t+1}-\x^{t}\right\rangle+ \left(\ell+\frac{\ell^2}{2r_{\y}^t}\right) \|\x^{t+1}-\x^{t}\|^2 
          -\frac{1}{2\alpha^t}\|\y^{t+1}-\y^{t}\|^2\notag\\
          &\,+\frac{r_{\y}^t-r_{\y}^{t+1}}{2}D_{\cY}^2\notag \\
   \leq & \, \ell\|\y^\star_t(\x^{t})\!-\!\y^{t}\|\!\cdot\!\|\x^{t+1}\!-\!\x^{t}\| \!+\! \left(\!\ell\!+\!\frac{\ell^2}{2r_{\y}^t}\!\right) \|\x^{t+1}\!-\!\x^{t}\|^2 
          \!-\!\frac{1}{2\alpha^t}\|\y^{t+1}\!-\!\y^{t}\|^2\!+\!\frac{r_{\y}^t\!-\!r_{\y}^{t+1}}{2}D_{\cY}^2, \label{eq:phi6}
\end{align}
where the second inequality follows from the Cauchy-Schwarz inequality and the $\ell$-Lipschitz continuity of $\nabla_{\x}{f}_t(\x,\cdot)$.

Putting \cref{eq:phi3}  and \cref{eq:phi6} together yields 
\begin{align*}
  & \, {\Psi_{1}^{t+1}}-{\Psi_{1}^{t}} \notag\\
  \leq & \, 2\ell\|\y^\star_t(\x^{t})-\y^{t}\|\cdot\|\x^{t+1}-\x^{t}\|-\frac{1}{2\alpha^t}\|\y^{t+1}-\y^{t}\|^2\notag-\left(\frac{1}{c^{t}}-\frac{3\ell}{2}-\frac{\ell^2}{r_{\y}^t}\right)\|\x^{t+1}-\x^{t}\|^2 \notag\\
  &\,+(r_{\y}^t-r_{\y}^{t+1})D_{\cY}^2\notag\\
  \leq & \, 2\ell\|\y^\star_t(\x^{t})-\y^{t}\|\cdot\|\x^{t+1}-\x^{t}\|-\frac{1}{2\alpha^t}\|\y^{t+1}-\y^{t}\|^2-\frac{1}{2c^{t}}\|\x^{t+1}-\x^{t}\|^2 +(r_{\y}^t-r_{\y}^{t+1})D_{\cY}^2,
\end{align*}
where the final inequality is due to \cref{con-pgda}, i.e., $0<c^t\leq\frac{r_{\y}^t}{\ell(3r_{\y}^t+2\ell)}$. 
This completes the proof.

\end{proof}

In \cref{lemma:perturb-potentialfunc-decrease}, the bound includes a positive term $2\ell\|\y^\star_t(\x^{t})-\y^{t}\|\cdot\|\x^{t+1}-\x^{t}\|$, which measures the error introduced by $\y^{t}$ not being optimal for a given $\x^{t}$. For the Lyapunov function to decrease, this term needs to be controlled.
We leverage the strong concavity of $ f_t(\x,\cdot)$ to apply a dual error bound in~\citep[Theorem~3.1]{pang1987posteriori}, which upper bounds the dual gap $\|\y_t^\star(\x^{t+1})-\y^t\|$ by the iterate difference $\|\y^{t+1}-\y^t\|$.
The resulting bound is stated below.

\begin{lemma}[Dual error bound; cf. \citep{pang1987posteriori} Theorem~3.1]\label{lemma:perturb-dualerror}
Suppose that the sequence $\{(\x^{t},\y^{t})\}_{t \ge 1}$ is generated by \emph{Perturbed GDA}. Then, for any $t \ge 0$, we have 
\begin{equation*}
\|\y_{t}^\star(\x^{t+1}) - \y^{t}\|\leq \frac{1+\alpha^t(\ell+r_{\y}^t)}{\alpha^t  r_{\y}^t } \|\y^{t+1}-\y^{t}\|.
\end{equation*}
\end{lemma}

With \cref{lemma:perturb-potentialfunc-decrease,lemma:perturb-dualerror},
we now outline the main idea of the proof of \cref{thm:pgda-nc-c-epi}.
We will use \cref{lemma:perturb-potentialfunc-decrease} to bound the decrease of the Lyapunov function and a key step is to plug the dual gap bound in
\cref{lemma:perturb-dualerror} to the descent inequality in
\cref{lemma:perturb-potentialfunc-decrease}.
It yields a refined descent
estimate for the Lyapunov function $\Psi_1^t$.
Together with the step-size conditions in \cref{con-pgda}, the claimed $\epsilon$-GS and $\epsilon$-OS guarantees can be directly derived from this descent estimate.

\begin{proof}[Proof of  \cref{thm:pgda-nc-c-epi}.]
First, by combining \cref{lemma:perturb-potentialfunc-decrease,lemma:perturb-dualerror}, we obtain a descent property for $\Psi_1$. For $t\ge0$,
\begin{align}
  & \, \Psi_1^{t+1}-\Psi_1^t\notag\\
  \leq & \, 2\ell\|\y_{t}^\star(\x_t)-\y^{t}\|\!\cdot\!\|\x^{t+1}-\x^{t}\|\!-\!\frac{1}{2\alpha^t}\|\y^{t+1}-\y^{t}\|^2\!-\!\frac{1}{2c^t}\|\x^{t+1}-\x^{t}\|^2\!+\!(r_{\y}^t-r_{\y}^{t+1})D_{\cY}^2 \notag\\
  \leq & \, 2\ell\|\y_{t}^\star(\x^{t+1})-\y^{t}\|\cdot\|\x^{t+1}-\x^{t}\|+2\ell\|\y_{t}^\star(\x^{t+1})-\y_{t}^\star(\x^{t})\|\cdot\|\x^{t+1}-\x^{t}\|\notag\\
  &\,-\frac{1}{2\alpha^t}\|\y^{t+1}-\y^{t}\|^2-\frac{1}{2c^t}\|\x^{t+1}-\x^{t}\|^2 +(r_{\y}^t-r_{\y}^{t+1})D_{\cY}^2\notag\\
  \leq & \, \frac{2\ell(1+\alpha^t(\ell+r_{\y}^t))}{\alpha^t r_{\y}^t}\|\y^{t+1}-\y^{t}\|\cdot\|\x^{t+1}-\x^{t}\|+\frac{2\ell^2}{r_{\y}^t}\|\x^{t+1}-\x^{t}\|^2\notag\\
  &\,-\frac{1}{2\alpha^t}\|\y^{t+1}-\y^{t}\|^2-\frac{1}{2c^t}\|\x^{t+1}-\x^{t}\|^2+(r_{\y}^t-r_{\y}^{t+1})D_{\cY}^2\notag\\
  \leq & \, \frac{1}{4\alpha^t}\|\y^{t+1}-\y^{t}\|^2+\frac{4\ell^2(1+\alpha^t(\ell+r_{\y}^t))^2}{\alpha^t(r_{\y}^t)^2}\|\x^{t+1}-\x^{t}\| ^2+\frac{2\ell^2}{r_{\y}^t}\|\x^{t+1}-\x^{t}\|^2\notag\\
   & \,  -\frac{1}{2\alpha^t}\|\y^{t+1}-\y^{t}\|^2-\frac{1}{2c^t}\|\x^{t+1}-\x^{t}\|^2 +(r_{\y}^t-r_{\y}^{t+1})D_{\cY}^2\notag\\
 \leq & \,-\frac{1}{4\alpha^t}\|\y^{t+1}-\y^{t}\|^2 -\frac{1}{16c^t}\|\x^{t+1}-\x^{t}\| ^2+(r_{\y}^t-r_{\y}^{t+1})D_{\cY}^2,\label{eq:t>0}
\end{align}
 where the first inequality is by \cref{lemma:perturb-potentialfunc-decrease} and
 the third inequality is due to \cref{lemma:perturb-dualerror} and  \cite[Lemma 4.3]{lin2020gradient},
 the fourth one follows because $ab\leq \frac{1}{2}(a^2+b^2)$ for any $a,b\in\R$, and the last one is due to \cref{con-pgda}, i.e., $1+\alpha(\ell+r_{\y}^t)\leq\frac{5}{4}$ and $c^t\leq \min\{\frac{3r_{\y}^t}{128\ell^2},\frac{(r_{\y}^t)^2\alpha^t}{16\ell^2}\}$. 
 Summing \cref{eq:t>0} for $t=0,\cdots,T-1$, we have
\[\Psi_1^T-\Psi_1^0 \leq -\sum_{t=0}^{T-1} \left(\frac{1}{4\alpha^t}\|\y^{t+1}-\y^{t}\|^2 +\frac{1}{16c^t}\|\x^{t+1}-\x^{t}\| ^2\right)+(r_{\y}^0-r_{\y}^T)D_{\cY}^2. \]
Thus, because we choose $r_{\y}^t=r_{\y}^0$ for all $t$, then there exists a $t\in[T]$ satisfying $t\ge\lfloor\frac{T}{2}\rfloor$ such that
	\begin{align}
\|\y^{t+1}-\y^{t}\| =  \mathcal{O}\left(\sqrt{\frac{\alpha^t\Delta_{\Psi_1}}{T}}\right) \quad \text{and}\quad \|\x^{t+1}-\x^{t}\|=  \mathcal{O}\left(\sqrt{\frac{c^t\Delta_{\Psi_1}}{T}}\right). \notag
	\end{align}
 This conclusion follows from the fact that  $\Psi_1^t(\x^{t},\y^{t})\geq\Phi_t(\x^{t})\geq\min_{\x\in\cX}\Phi_t(\x)$.\\ 
\noindent{(i)}
For the GS case,  we choose  $r_{\y}^t$ of order $\Theta\left(\sqrt[4]{\frac{\Delta_{\Psi_1}\ell^3}{D_{\cY}^2T}}\right)$.  
The  optimality condition of $\y$-update yields that
\begin{align}
\bz
\in &\,r_{\y}^t\y^{t}+\frac{1}{\alpha^t}(\y^{t+1}-\y^{t}) - \nabla_{\y} f(\x^{t+1}, \y^{t})+\partial \iota_{\mathcal{Y}}(\y^{t+1}) \notag \\
 \in &\,-\nabla_{\y} f(\x^{t+1}, \y^{t+1})+\partial \iota_{\mathcal{Y}}(\y^{t+1})+r_{\y}^t\y^{t}+ \frac{1}{\alpha^t}(\y^{t+1}-\y^{t})\notag\\
 &\,+ (\nabla_{\y} f(\x^{t+1}, \y^{t+1})- \nabla_{\y} f(\x^{t+1}, \y^{t})). \notag
 \end{align}
Then, we have 
\begin{align}
     &\,{\rm dist}(\bz,-\nabla_{\y} f(\x^{t+1}, \y^{t+1})+\partial\iota_{\mathcal{Y}}(\y^{t+1}))\notag\\
     \leq &\, \|r_{\y}^t\y^{t}\|+\frac{1}{\alpha^t}\|\y^{t+1}-\y^{t}\|+\|\nabla_{\y} f(\x^{t+1}, \y^{t+1})- \nabla_{\y} f(\x^{t+1}, \y^{t})\|\notag\\
     \leq &\, r_{\y}^tD_{\cY} + \left(\frac{1}{\alpha^t}+\ell\right)\|\y^{t+1}-\y^{t}\| \label{opt-y-up} \\
     = &\, \mathcal{O}\left(\max\left(r_{\y}^tD_{\cY},\sqrt{\frac{\ell\Delta_{\Psi_1}}{T}}\right)\right)\label{eq:Per-GS-y}\\ 
     = &\, \mathcal{O}\left(\sqrt[4]{\frac{ \ell^3D_{\cY}^2\Delta_{\Psi_1}}{T}}\right)\notag,
\end{align}
where the second inequality follows from the $\ell$-Lipschitz continuity of  $\nabla_{\y}f(\x^{t+1},\cdot)$ and the first equality is due to \cref{con-pgda}, i.e., $\alpha^t=\Theta(\tfrac{1}{\ell})$.

Next, we turn to the primal part. Similarly, the optimality condition of the $\x$-update yields that
\begin{align*}
\bz  & \in \frac{1}{c^t} (\x^{t+1}-\x^{t}) + \nabla_{\x} f(\x^{t},\y^{t}) +\partial\iota_{\mathcal{X}}(\x^{t+1}) \\
& \in  \nabla_{\x} f(\x^{t+1}, \y^{t+1}) +\partial\iota_{\mathcal{X}}(\x^{t+1})  +\frac{1}{c^t} (\x^{t+1}-\x^{t}) 
+(\nabla_{\x} f(\x^{t}, \y^{t})-\nabla_{\x} f(\x^{t+1}, \y^{t+1})).
\end{align*}
It follows that  
\begin{align}
&{\rm dist}(\bz,\nabla_{\x} f(\x^{t+1}, \y^{t+1}) +\partial\iota_{\mathcal{X}}(\x^{t+1})) \notag\\
\leq &\,\frac{1}{c^t} \|\x^{t+1}-\x^{t}\|+ \|\nabla_{\x} f(\x^{t+1}, \y^{t+1}) - \nabla_{\x} f(\x^{t}, \y^{t})\|\notag\\
\leq &\, \left(\frac{1}{c^t}+\ell\right) \|\x^{t+1}-\x^{t}\| + \ell  \|\y^{t+1}-\y^{t}\| \notag\\
= &\, \mathcal{O}\left(\max\left(\sqrt{\frac{\Delta_{\Psi_1}}{c^tT}},\sqrt{\frac{\ell\Delta_{\Psi_1}}{T}}\right)\right)\label{eq:Per-GS-x}\\
= &\, \mathcal{O}\left(\sqrt[4]{\frac{ \ell^3D_{\cY}^2\Delta_{\Psi_1}}{T}}\right)\notag,
\end{align}
where the second inequality follows from the $\ell$-smoothness of the function $f$, and the last equality is a consequence of  \cref{con-pgda}, i.e. $c^t=\Theta\left(\frac{(r_{\y}^t)^2}{\ell^3}\right)=\Theta\left(\sqrt{\frac{\Delta_{\Psi_1}}{\ell^3D_{\cY}^2T}}\right)$.

For the GS case, our analysis gives 
$\epsilon \leq \cO\left(\sqrt[4]{\tfrac{\ell^{3} D_{\cY}^{2} \Delta_{\Psi_1}}{T}}\right)$, 
which is equivalent to 
$T  \leq \cO\left(\tfrac{\ell^{3} D_{\cY}^{2} \Delta_{\Psi_1}}{\epsilon^{4}} \right)$. 
Thus $\cO\left(\tfrac{\ell^{3} D_{\cY}^{2} \Delta_{\Psi_1}}{\epsilon^{4}} \right)$ iterations suffice to reach an $\epsilon$-GS, and the corresponding choice of $r_{\y}^t$ is 
$r_{\y}^t = \Theta(\tfrac{\epsilon}{D_{\cY}})$.

\noindent(ii)
We now turn to the OS case. We choose  $r_{\y}^t$ as $\Theta\left({\sqrt[3]{\frac{\ell^2\Delta_{\Psi_1}}{D_{\cY}^2T}}}\right)$. 
To proceed, we rely on the following lemma, which establishes a connection between $\epsilon$-GS and $\epsilon$-OS for general smooth NC-C minimax problems. The proof of \cref{lemma2} is provided in \cref{sec:GS-OS}.

\begin{lemma} 
\label{lemma2}
Suppose that \cref{con-pgda} holds. We have 
\begin{align}
&\, \left\|\prox_{\frac{1}{2\ell} \Phi}(\x)-\x\right\|^2 \notag\\
\leq &\,   \frac{2 D_{\cY}}{\ell} {\rm dist}(\bz,-\nabla_{\y} f(\x, \y)+\partial\iota_{\mathcal{Y}}(\y)) +\frac{1}{\ell^2}  {\rm dist}^2(\bz,\nabla_{\x} f(\x, \y)+\partial\iota_{\mathcal{X}}(\x)). \notag
\end{align}
\end{lemma}
\noindent Now, we apply \cref{lemma2} to connect GS and OS:
\begin{align*}
&\, \left\|\prox_{\frac{1}{2\ell} \Phi}(\x^{t+1})-\x^{t+1}\right\|^2 \\
\leq &\,  \frac{2 D_{\cY} }{\ell}{\rm dist}(\bz,\!-\!\nabla_{\!\y} f(\x^{t+1}, \!\y^{t+1})\!+\!\partial\iota_{\!\mathcal{Y}}(\y^{t+1}))\!+\!\!\frac{1}{\ell^2}{\rm dist}^2(\bz,\!\nabla_{\!\x} f(\x^{t+1},\! \y^{t+1})\!+\!\partial\iota_{\!\mathcal{X}}(\x^{t+1})) \\
= &\, \cO\left(\frac{D_{\cY}}{\ell}\cdot\max\left(\sqrt{\frac{\ell\Delta_{\Psi_1}}{T}},r_{\y}^tD_{\cY}\right)+\frac{1}{\ell^2}\max\left(\sqrt{\frac{\Delta_{\Psi_1}}{c^tT}},\sqrt{\frac{\ell\Delta_{\Psi_1}}{T}}\right)^2\right)\\
= & \,\cO\left(\sqrt[3]{\frac{D_{\cY}^4 \Delta_{\Psi_1}}{\ell T}}\right),
\end{align*}
where the first equality comes from \cref{eq:Per-GS-y,eq:Per-GS-x}, and the last equality is due to \cref{con-pgda}, i.e., $c^t=\Theta\left(\frac{(r_{\y}^t)^2}{\ell^3}\right)=\Theta\left(\sqrt[3]{\frac{\Delta_{\Psi_1}^2 }{\ell^5 D_{\cY}^4T^2}}\right)$.

Putting everything together yields 
\[2\ell\left\|\prox_{\frac{1}{2\ell} \Phi}(\x^{t+1})-\x^{t+1}\right\|= \cO\left(\sqrt[6]{\frac{\ell^5D_{\cY}^4 \Delta_{\Psi_1}}{ T}}\right). \]

Following the similar argument, we need at least $T=\cO\left(\frac{\ell^5D_{\cY}^4 \Delta_{\Psi_1}}{\epsilon^6}\right)$ to reach an $\epsilon$-OS if we choose $r_{\y}^t=\Theta\left(\frac{\epsilon^2}{\ell D^2_{\cY}}\right)$. This completes the proof.
\end{proof}

 With \cref{thm:pgda-nc-c-epi} proved, we next prove \cref{thm:pgda-asm}. We show that the proof of \cref{thm:pgda-asm} follows directly from the preceding analysis, even though it uses a diminishing sequence of $r^t_{\y}$.
\begin{proof}[Proof of \cref{thm:pgda-asm}.]
It holds that $|\cI^\star(\infty)|=+\infty$, since for each $T$, the set
\[
\argmin_{k\in\{\lfloor T/2\rfloor,\dots,T\}}
\left\{
\frac{1}{4\alpha^k}\|\y^{k+1}-\y^k\|^2
+
\frac{1}{16c^k}\|\x^{k+1}-\x^k\|^2
\right\}
\]
is nonempty and contained in $\cI^\star(T)$. Hence,
$
\max_{k\in\cI^\star(T)} k \ge \Big\lfloor \frac{T}{2}\Big\rfloor
$
for every $T$, which implies that $\cI^\star(\infty)$ is infinite.

Let $(\x^\infty,\y^\infty)$ be any limit point of the sequence $\{(\x^t,\y^t)\}_{t\in \cI^\star(\infty)}$. Then there exists a subsequence $\{(\x^{i_k},\y^{i_k})\}_{k\ge1}$ such that
\begin{equation}
    \lim_{k\to\infty}  (\x^{i_k},\y^{i_k})  = (\x^\infty,\y^\infty). 
    \label{eq:conver-pgda}
\end{equation}

Note that \cref{eq:Per-GS-y,eq:Per-GS-x} holds for every $t=i_k$. In particular, \cref{eq:Per-GS-y} implies
\begin{align}
    &\,\lim_{k\to\infty}{\rm dist}(\bz,-\nabla_{\y} f(\x^{{i_k}+1}, \y^{{i_k}+1})+\partial\iota_{\mathcal{Y}}(\y^{{i_k}+1}))\notag\\
    =&\, \mathcal{O}\left(\lim_{k\to\infty}\max\left(\sqrt{\frac{\ell\left(\Delta_{\Psi_1}+r_{\y}^0D_{\cY}^2\right)}{i_k}},r_{\y}^{i_k}D_{\cY}\right)\right) =0,\label{eq:per-y-0}
\end{align}
where the last equality follows from $r_{\y}^{i_k}\searrow 0$. Moreover, since
$
\|\x^{i_k+1}-\x^{i_k}\|\to 0
$ and  $
\|\y^{i_k+1}-\y^{i_k}\|\to 0,
$
\cref{eq:conver-pgda} yields
$
(\x^{i_k+1},\y^{i_k+1})\to (\x^\infty,\y^\infty).
$ 
Because $\cY$ is nonempty, closed, and convex, the operator $\partial\iota_{\cY}$ is {\it maximal monotone}~\citep[Theorem~A]{rockafellar1970maximal}, and hence its graph is closed~\citep[Exercise~12.8]{rockafellar2009variational}. Therefore, combining the above with \cref{eq:per-y-0}, we obtain
\begin{equation}
    \bz\in -\nabla_{\y} f(\x^{\infty}, \y^{\infty})+\partial\iota_{\mathcal{Y}}(\y^{\infty}).
    \label{eq:p-11}
\end{equation}

Similarly, for the primal part, \cref{eq:Per-GS-x} implies 
\begin{align*}
    &\,\lim_{k\to\infty}{\rm dist}(\bz,\nabla_{\x} f(\x^{{i_k}+1}, \y^{{i_k}+1})+\partial\iota_{\mathcal{X}}(\x^{{i_k}+1}))\notag\\
    =&\, \mathcal{O}\left(\lim_{k\to\infty}\max\left(\sqrt{\frac{\Delta_{\Psi_1}+r_{\y}^0D_{\cY}^2}{c^{i_k}{i_k}}},\sqrt{\frac{\ell\left(\Delta_{\Psi_1}+r_{\y}^0D_{\cY}^2\right)}{{i_k}}}\right)\right) =0,
\end{align*}
where the last equality follows from $i_k(r_{\y}^{i_k})^2\to\infty$ and \cref{con-pgda}, i.e., $c^{i_k}=\Theta((r_{\y}^{i_k})^2)$. By the same argument as above, we obtain
\begin{equation}
    \bz\in \nabla_{\x} f(\x^{\infty}, \y^{\infty})+\partial\iota_{\mathcal{X}}(\x^{\infty}).\label{eq:p-2}
\end{equation}
 Therefore, combining \cref{eq:p-11,eq:p-2} yields that $(\x^\infty,\y^\infty)$ is a $0$-GS of problem \cref{eq:prob}. This completes the proof. 
\end{proof}

\subsection{Convergence Analysis of Perturbed Smoothed GDA}\label{subsec:PSGDA-ub}
In this subsection, we establish the iteration complexity of \emph{Perturbed Smoothed GDA}  stated in \cref{thm:psgda-epi}.
 Similar to \cref{subsec:ps-gda}, though \cref{thm:psgda-epi} is stated for constant choice of $r^t_{\y}$, we prove a stronger result where $r^t_{\y}$ is adaptive to provide the most general form of our analytical framework.

Following our unified analysis framework, we begin by introducing the Lyapunov function:
\begin{equation}
\Psi_2^t(\x,\y,\z)
:= 
\underbrace{F_t(\x,\y,\z)-d_t(\y,\z)}_{\text{Primal descent}}
+
\underbrace{\left(p_t(\z)-d_t(\y,\z)\right)}_{\text{Dual ascent}}
+
\underbrace{p_t(\z)}_{\text{Proximal descent}}. \label{Lya-fcn}
\end{equation}
Similar as before, we first establish a basic descent property of $\Psi_2^t(\x,\y,\z)$, which will be used to prove \cref{thm:psgda-epi}. The proof is essentially the same as that of \citep[Proposition~4.1]{zhang2020single}, except that an additional error term
$
(r_{\y}^t-r_{\y}^{t+1})D_{\cY}^2
$
appears.
To deal with the additional error term, the argument follows a similar technical route as \cref{lemma:perturb-potentialfunc-decrease}.
For simplicity, we write $\Psi_{2}^t:=\Psi_2^t(\x^{t},\y^{t},\z^{t})$. 


\begin{lemma}[Basic descent estimate]
\label{prop:decrease}
 Suppose  \cref{con-psgda} holds, and let $\{(\x^{t},\y^{t},\z^t)\}_{t \ge 0}$ be the sequence generated by  \emph{Perturbed Smoothed GDA}. 
 Then, for any $t\ge0$, we have
 \begin{equation*}
    \begin{aligned}
    \Psi_{2}^{t+1}-\Psi_{2}^{t}\le& - \frac{1}{8c^t}\|\x^{t+1}-\x^{t}\|^{2}-\frac{1}{8\alpha^t}\left\|\y^{t}-\y_{+}^t(\z^{t})\right\|^2 -\frac{r_{\x}\beta^t}{8}\|\z^{t}-\x^{t+1}\|^{2} \\
    & + 24r_{\x}\beta^t\left\|\x_{t}^\star(\z^{t})-\x_{t}(\y_{+}^t(\z^{t}), \z^{t})\right\|^{2}+\frac{3(r_{\y}^t-r_{\y}^{t+1})}{2}D_{\cY}^2.
    \end{aligned}
\end{equation*}
\end{lemma}

To establish the descent property of $\Psi_2^t$, it suffices to control the term $\|\x_{t}^\star(\z^{t})-\x_{t}(\y_{+}^t(\z^{t}),\z^{t})\|^{2}$. 
By construction, applying \emph{Perturbed Smoothed GDA}  on an NC-C minimax problem can be viewed as applying the vanilla \emph{Smoothed GDA} to an equivalent NC-SC minimax problem, where the strong concavity on the dual side is induced by the perturbation parameter $r_{\y}^t$; hence, using the \emph{homogeneous dual error bound} from \citep[Corollary~5.1]{li2025nonsmooth}, we bound it by the positive term $\|\y^{t}-\y_{+}^t(\z^{t})\|^{2}$ in the basic descent estimate. 
We formalize this homogeneous dual error bound in the following lemma and its proof is in \cref{sec:proof-dualerror}.
\begin{lemma}[Dual error bound for NC-SC]
\label{lemma:dualerror}
 For any $t \ge 0$, any $\y\in\mathcal{Y}$ and $\z \in \mathbb{R}^n$, we have
\vspace{1mm}
  \vspace{1mm}
    \[\left\|\x_{t}^\star(\z^{t})-\x_{t}(\y_{+}^t(\z^{t}), \z^{t})\right\|  \leq \omega^t\left\|\y^t-\y_{+}^t(\z^t)\right\|,\] 
      \vspace{2mm} 
     where $\omega^t:=\frac{1}{\sqrt{r_{\y}^t}} \cdot \frac{(r_{\x}-\ell-r^t_{\y})+\alpha^t(\ell+r_{\y}^t)(3r_{\x}-2\ell-2r_{\y}^t)}
           {\alpha^t (r_{\x}-\ell-r_{\y}^t)^{3/2}}$. 
\end{lemma}

With \cref{prop:decrease} and \cref{lemma:dualerror}, we are now ready to present the detailed proof of \cref{thm:psgda-epi}. Similar to the previous subsection, the key of the proof is still using \cref{prop:decrease,lemma:dualerror}, and \cref{con-psgda} to establish the decreasing property of the Lyapunov function.  

\begin{proof}[Proof of \cref{thm:psgda-epi}.]
First, under the step-size condition in \cref{con-psgda} that
$\beta^t < \tfrac{1}{384\, r_{\x}\alpha^t(\omega^t)^{2}}$, we obtain
\begin{equation*}
\begin{aligned}
&\,\Psi_{2}^t-\Psi_{2}^{t+1}\\
\ge\, &  \frac{1}{8c^t}\!\|\x^{t+1}\!-\!\x^{t}\|^{2}\!+\!\!\left(\!\!\frac{1}{8\alpha^t} \!-\!24r_{\x}\beta^t (\omega^t)^2\!\!\right)\! \|\y^{t}\!-\!\y_{+}^t(\z^{t})\|^2 \!+\!\frac{r_{\x}\beta^t}{8}\!\|\z^{t}\!-\!\x^{t+1}\|^{2}\!+\!\frac{3(r_{\y}^t\!\!-\!r_{\y}^{t+1})}{2}\!D_{\!\cY}^2 \\
 \ge \, & \frac{1}{8c^t}\|\x^{t+1}-\x^{t}\|^{2}+\frac{1}{16\alpha^t} \|\y^{t}-\y_{+}^t(\z^{t})\|^2 +\frac{r_{\x}\beta^t}{8}\|\z^{t}-\x^{t+1}\|^{2}+\frac{3(r_{\y}^t-r_{\y}^{t+1})}{2}D_{\cY}^2 .
\end{aligned}
\end{equation*}
By \citep[Lemma~B.1]{zhang2020single}, under the definition of $\Delta_{\Psi_2}$ and the parameter choice $r_{\y}^t = r_{\y}^0$ for all $t$, for any integer $T > 0$, there exists an index $t \in [T]$ such that $t\ge\lfloor\frac{T}{2}\rfloor$ such that
\[
\|\x^{t+1}-\x^{t}\| \leq \sqrt{\frac{8c^t\Delta_{\Psi_2}}{T}}, \, \|\y^{t}-\y_{+}^t(\z^{t})\| \leq  \sqrt{\frac{16\alpha^t\Delta_{\Psi_2}}{T}},\, \text{and}\, \|\z^{t}-\x^{t+1}\| \leq\sqrt{\frac{8\Delta_{\Psi_2}}{r_{\x}\beta^t T}}.
\]
Therefore, substituting the parameter choices from \cref{con-psgda} into the above estimate yields
  \begin{equation}
  \label{ps-xyz}
  \max\Bigl\{\|\x^{t+1}-\x^{t}\|,\ \|\y^{t}-\y_{+}^t(\z^{t})\|\Bigr\}
  = \cO\left(\sqrt{\frac{\Delta_{\Psi_2}}{\ell T}}\right),\quad 
  \|\z^{t}-\x^{t+1}\|
  = \cO\left(\sqrt{\frac{\Delta_{\Psi_2}}{r_{\y}^t T}}\right).
  \end{equation}

\noindent (i)
For the GS case, we choose $r_{\y}^t$ as $\Theta\left(\sqrt[3]{\frac{\ell^2\Delta_{\Psi_2}}{D_{\cY}^2T}}\right)$.\\
From \cref{opt-y-up}, we have
\begin{align}
     &\,{\rm dist}(\bz,-\nabla_{\y} f(\x^{t+1}, \y^{t+1})+\partial\iota_{\mathcal{Y}}(\y^{t+1}))\notag\\
     \leq &\, \left(\frac{1}{\alpha^t}+\ell\right)\|\y^{t+1}-\y^{t}\| +r_{\y}^tD_{\cY}\notag\\
      \leq &\, \left(\frac{1}{\alpha^t}+\ell\right) \left(\|\y^{t}-\y_{+}^t(\z^{t})\|+\alpha^t\ell \sigma_3^t\|\x^{t+1}-\x^{t}\|\right)+r_{\y}^tD_{\cY} \label{smooth-y}\\
       = &\, \mathcal{O}\left(\max\left(\sqrt{\frac{\ell\Delta_{\Psi_2}}{T}},r_{\y}^tD_{\cY}\right)\right)\label{ps-dual}\\
=&\, \mathcal{O}\left(\sqrt[3]{\frac{\ell^2 D_{\cY}\Delta_{\Psi_2}}{T}}\right),\notag
 \end{align}
where the second inequality follows from the triangle inequality and \citep[Lemma~B.9]{zhang2020single}, which states that $\|\y^{t+1}-\y_{+}^t(\z^{t})\|\le \alpha^t\ell\sigma_3^t\|\x^{t+1}-\x^{t}\|$, where $\sigma_3^t:=\frac{1+c^t(r_{\x}-\ell)}{c^t(r_{\x}-\ell)}=\Theta(1)$, and the first equality comes from \cref{ps-xyz}, the choice of $r^t_{\y}$ and \cref{con-psgda}, i.e., $\alpha^t=\Theta(\frac{1}{\ell})$.

Next, we examine the primal part. The optimality condition of the $\x$–update gives
\begin{align*}
\bz  \in &\, \nabla_{\x} f(\x^{t},\y^{t}) +\partial\iota_{\mathcal{X}}(\x^{t+1})+\frac{1}{c^t} (\x^{t+1}-\x^{t}) + r_{\x} (\x^{t+1}-\z^{t}) \\
\in &\, \nabla_{\x} f(\x^{t+1}, \y^{t+1}) +\partial\iota_{\mathcal{X}}(\x^{t+1}) +(\nabla_{\x} f(\x^{t}, \y^{t})-\nabla_{\x} f(\x^{t+1}, \y^{t+1})) \\
&\,+\frac{1}{c^t} (\x^{t+1}-\x^{t}) 
 + r_{\x} (\x^{t+1}-\z^{t}). 
\end{align*}
It follows that  
\begin{align}
&\,{\rm dist}(\bz,\nabla_{\x} f(\x^{t+1}, \y^{t+1}) +\partial\iota_{\mathcal{X}}(\x^{t+1})) \notag\\
 \leq& \, \|\nabla_{\x} f(\x^{t}, \y^{t})-\nabla_{\x} f(\x^{t+1}, \y^{t+1})\| +\frac{1}{c^t}\|\x^{t+1}-\x^{t}\| + r_{\x}\|\x^{t+1}-\z^{t}\|\notag\\ 
 \leq& \, \left(\frac{1}{c^t}+\ell\right) \|\x^{t+1}-\x^{t}\| + \ell  \|\y^{t+1}-\y^{t}\| + r_{\x}\|\x^{t+1}-\z^{t}\| \notag\\ 
 \leq& \, \left(\frac{1}{c^t}+\ell\right) \|\x^{t+1}-\x^{t}\| + \ell \left(\|\y^{t}-\y_{+}^t(\z^{t})\|+\alpha^t\ell \sigma_3^t\|\x^{t+1}-\x^{t}\|\right) + r_{\x}\|\x^{t+1}-\z^{t}\|\label{opt-x-up} \\
 = &\, \mathcal{O}\left(\max\left(\sqrt{\frac{\ell\Delta_{\Psi_2}}{T}},\sqrt{\frac{\ell^2\Delta_{\Psi_2}}{r_{\y}^tT}}\right)\right)\label{ps-primal}\\
=&\, \mathcal{O}\left(\sqrt[3]{\frac{\ell^2 D_{\cY}\Delta_{\Psi_2}}{T}}\right),\notag
\end{align}
where the second inequality follows from the $\ell$-Lipschitz continuity of $\nabla_{\x}f(\cdot,\cdot)$, the third inequality follow from \citep[Lemma B.9]{zhang2020single}, and the first equality is due to \cref{ps-xyz} and \cref{con-psgda}, i.e., $c^t, \alpha^t=\Theta(\frac{1}{\ell})$ and $r_{\x}=\Theta(\ell)$. 

Recall that for the GS case, we choose $r_{\y}^t=\Theta \left(\sqrt[3]{\frac{\ell^2\Delta_{\Psi_2}}{D_{\cY}^2T}}\right)$  so that the two terms in \cref{ps-dual,ps-primal} are of the same order. With this choice, the resulting bound scales as $\mathcal{O}\left(\sqrt[3]{\frac{\ell^2 D_{\cY}\Delta_{\Psi_2}}{T}}\right)$.
Thus, to reach an $\epsilon$-GS it suffices to take 
$T = \cO\left(\frac{\ell^{2} D_{\cY}\, \Delta_{\Psi_2}}{\epsilon^{3}}\right)$,
and substituting this $T$ back into the above choice yields the equivalent parameterization $r_{\y}^t = \Theta\left(\tfrac{\epsilon}{D_{\cY}}\right)$.

\noindent(ii)
For the OS case, we choose  $r_{\y}^t$ as $\Theta\left(\sqrt{\frac{\ell\Delta_{\Psi_2}}{D_{\cY}^2T}}\right)$. Then we obtain
\begin{align*}
&\, \left\|\prox_{\frac{1}{2\ell} \Phi}(\x^{t+1})-\x^{t+1}\right\|^2 \\
\leq &\,  \frac{2 D_{\cY} }{\ell}{\rm dist}(\bz,\!-\!\nabla_{\!\y} f(\x^{t+1}, \!\y^{t+1})\!+\!\partial\iota_{\!\mathcal{Y}}(\y^{t+1}))\!+\!\!\frac{1}{\ell^2}{\rm dist}^2(\bz,\!\nabla_{\!\x} f(\x^{t+1},\! \y^{t+1})\!+\!\partial\iota_{\!\mathcal{X}}(\x^{t+1})) \\
= &\,\cO\left(\frac{D_{\cY}}{\ell}\cdot\max\left(\sqrt{\frac{\ell\Delta_{\Psi_2}}{T}},r_{\y}^tD_{\cY}\right)+\frac{1}{\ell^2}\cdot\max\left(\sqrt{\frac{\ell\Delta_{\Psi_2}}{T}},\sqrt[3]{\frac{\ell^2\Delta_{\Psi_2} }{r_{\y}^tT}}\right)\right)\\
= &\,\cO\left(\sqrt{\frac{D_{\cY}^2\Delta_{\Psi_2}}{\ell T}}\right),
\end{align*}
where the first inequality arises from \cref{lemma2}, and the first equality follows from \cref{ps-primal,ps-dual}.
Putting everything together yields
\[2\ell\left\|\prox_{\frac{1}{2\ell} \Phi}(\x^{t+1})-\x^{t+1}\right\|= \cO\left(\sqrt[4]{\frac{\ell^3D_{\cY}^2\Delta_{\Psi_2}}{T}}\right). \]
Following the similar argument as the case of GS, we need at least $T=\cO\left(\frac{\ell^3D_{\cY}^2\Delta_{\Psi_2}}{\epsilon^4}\right)$ iterations to reach an $\epsilon$-OS when we choose $r_{\y}^t=\Theta\left(\frac{\epsilon^2}{ \ell D_{\cY}^2}\right)$.
This completes the proof.
\end{proof}
The proof of \cref{thm:psgda-asm} follows the same steps as the proof of \cref{thm:pgda-asm}, and is therefore omitted.
\begin{remark}
We prove \cref{thm:psgda-epi} (i) within our unified framework, illustrating its versatility and generality. 
We note, however, that the same iteration complexity bound can also be recovered through a simple reduction. 
Indeed, by \citep[Corollary 4.1]{yang2022faster}, the iteration complexity of \emph{Smoothed GDA} for the NC-SC minimax problem \eqref{eq:perturb} is 
$\cO(\kappa\ell\Delta\epsilon^{-2})$. 
If we choose $r_{\y}^t=r_{\y}=\Theta(\epsilon/D_{\cY})$ for all $t\ge 0$, then 
$\kappa=\ell/r_{\y}=\Theta(\ell D_{\cY}/\epsilon)$. 
Substituting this relation into the above bound gives
\[
    \cO(\kappa\ell\Delta\epsilon^{-2})
    =
    \cO\left(\frac{\ell^2 D_{\cY}\Delta_{\Psi_2}}{\epsilon^3}\right).
\]
A similar observation also appears in \citep[Remark 11]{Aybat2025ARM}. 
Nevertheless, such a reduction only recovers the same finite-time rate for a constant choice of $r_{\y}^t$. 
Since a constant perturbation parameter does not vanish, this reduction does not lead to an asymptotic convergence guarantee for the original NC-C problem. 
Our unified framework goes beyond this reduction by accommodating adaptive or diminishing perturbation sequences; this flexibility is essential for establishing the asymptotic convergence of \emph{Perturbed Smoothed GDA}; see \cref{thm:psgda-asm}.
\end{remark}

\section{Tight Analysis}\label{sec:analysis}
In \cref{sec:framework}, we established the iteration complexity of \emph{Perturbed GDA} and \emph{Perturbed Smoothed GDA}  for both game stationarity and optimization stationarity. In this section, we show that these bounds are tight. 
Specifically, we construct hard instances and derive lower bounds on the number of iterations required by each algorithm to reach a stationary point.
Comparing these lower bounds with the corresponding upper bounds from \cref{sec:framework} establishes the tightness claims.
For lower bound analysis in this section, we consider algorithms with constant choice of $r^t_{\y}$ for simplicity. Thus, we drop iteration superscripts and write $r_{y}^t = r_{y}$, $c^t = c$, $\alpha^t = \alpha$, and $\beta^t = \beta$, as these parameters are fixed across all iterations.

\subsection{Game Stationarity}\label{subsec:gs-tightness}
In this subsection, we prove \cref{thm:tight-alter}~(i) and \cref{thm:tight-psgda}~(i), establishing the tightness of the iteration complexities stated in \cref{thm:pgda-nc-c-epi}~(i) and \cref{thm:psgda-epi}~(i) for $\epsilon$-GS.
To this end, we construct a hard instance and show that the \emph{Perturbed GDA} and \emph{Perturbed Smoothed GDA}  require at least the corresponding number of iterations to reach an $\epsilon$-GS on this instance.
The hard instance is given below.
\begin{example}\label{eg:gs-alter-tight}
   Let the function $f:\mathbb{R}\times[0,D_{\cY}]\rightarrow\mathbb{R}$ and be defined by
    \begin{gather}\label{instance-gs}
		f(x,y)=
		\begin{cases}
            0,&  ~\mbox{if}~ x<0,\\
			-\frac{1}{2}\ell x^2+bxy, &  ~\mbox{if}~ 0\le x  \leq \frac{r_y D_{\cY}  }{b}, \\
			-\frac{\ell r_y^2 D_{\cY}^2}{2b^2}+r_yD_{\cY}y, &~\mbox{if}~  x >\frac{r_y D_{\cY} }{b},
		\end{cases}
	\end{gather}
where $b=\sqrt{3\ell r_{y}}$. It can be verified that
$f(x,y)$ is concave in $y$, and $\ell$-smooth.
\end{example}

\begin{remark}
The piecewise-smooth construction in \cref{eg:gs-alter-tight} is introduced to adapt the hard instance to the bounded-dual-domain setting. 
In particular, unlike standard constructions that rely on strong concavity and quadratic structure, our example only assumes concavity in the dual variable and requires a piecewise modification to ensure the normalization
\[
\min_{x\in\mathbb R}\max_{y\in[0,D_{\cY}]} f_0(x,y)
=
\min_{x\in\mathbb R}\max_{y\in[0,D_{\cY}]}
\left\{ f(x,y)-\frac{r_y}{2}y^2 \right\}
=0.
\]
This normalization ensures that the initial gaps $\Delta_{\Psi_1}$ and $\Delta_{\Psi_2}$ are well defined and explicitly computable, which is essential for matching the upper bounds in \cref{thm:pgda-nc-c-epi}~(i) and \cref{thm:psgda-epi}~(i). 
\end{remark}

    

\subsubsection{Perturbed GDA}\label{subsubsec:pgda-tight-gs}


In this subsubsection, we prove \cref{thm:tight-alter}~(i) by showing that \emph{Perturbed GDA} needs at least the number of iterations stated in \cref{thm:pgda-nc-c-epi}~(i) to reach an $\epsilon$-GS on \cref{eg:gs-alter-tight}.
Specifically, for any given $\epsilon$, we choose step sizes $\alpha$ and $c$ to satisfy \cref{con-pgda}, and set $r_y=\Theta\left(\frac{\epsilon}{D_{\cY}}\right)$. We show that if an iterate $(x^T,y^T)$ generated by \emph{Perturbed GDA} is an $\epsilon$-GS, then necessarily
$
T=\Omega\left(\frac{\ell^3 D_{\cY}^2 \Delta_{\Psi_1}}{\epsilon^4}\right).
$

To establish this lower bound, we analyze the dynamics of \emph{Perturbed GDA} on \cref{eg:gs-alter-tight}. 
In the regime $x^t\in\bigl[0,\frac{r_yD_{\cY}}{b}\bigr]$, the update rule can be written in the following linear vector form:
\begin{equation*}
    \left(
\begin{array}{c}
 x^{t+1}  \\
 y^{t+1}
 
\end{array}
\right)=(\bm{I}+\bm{M_1})\left(
\begin{array}{c}
 x^{t}  \\
 y^{t}
\end{array}
\right),~\mbox{where}~
    \bm{M_1} \coloneqq \left(
\begin{array}{cc}
  \ell c&-cb \\
 \alpha b(1+\ell c) & -\alpha r_{y}-cb^2\alpha
 
\end{array}
\right).
\end{equation*}
The convergence behavior of {\it Perturbed GDA} is determined by the spectral properties of $\bm{I}+\bm{M_1}$. 
To quantify its contraction factor, we study the eigenvalues of $\bm{M_1}$ (equivalently, those of $\bm{I}+\bm{M_1}$).
In particular, a direct calculation shows that $\bm{M_1}$ has an eigenvalue
\[
\lambda_1=\frac12\Big(A_1+\sqrt{A_1^2-4B_1}\Big),
\]
where
$
A_1:=\ell c-\alpha r_y-\alpha b^2c$ and 
$
B_1:=\alpha c\,(b^2-\ell r_y).
$
Using $b^2=3\ell r_y$, we have $A_1=\ell c-\alpha r_y(1+3\ell c)$ and $B_1=2\ell c\alpha r_y$.

Next, we show that $\lambda_1<0$ and characterize its magnitude. 
Because $\lambda_1<0$ and $|\lambda_1|$ is sufficiently small,  the recursion contracts with factor $(1+\lambda_1)$ along the corresponding eigendirection. 
This contraction factor determines the convergence speed along that direction. The proof of \cref{lemma:lamb-order-1} is deferred to \cref{sec:lem:lam}.
\begin{lemma}\label{lemma:lamb-order-1}
The eigenvalue $\lambda_1$ satisfies
$
\lambda_1 < 0$
and $
|\lambda_1| = \Theta\left(\frac{\epsilon^2}{\ell^2 D_{\mathcal Y}^2}\right). 
$ 
\end{lemma}

With the sign and magnitude of $\lambda_1$ known, we choose $(x^0,y^0)$ so that its direction is aligned with an eigenvector of $\bm{I}+\bm{M_1}$, which reduces the two-dimensional recursion to a one-dimensional multiplicative dynamics along the eigendirection and yields an explicit expression for the iterates.
\begin{lemma}\label{gs-alter-tight}
Let \emph{Perturbed GDA} be initialized at
$x^0=\frac{2\epsilon}{\ell}$ and $y^0=\frac{2b(1+\ell c)\alpha}{\sqrt{A_1^2-4B_1}+2\ell c-A_1}\cdot x^0$. Then, for all $t\in\mathbb N$, the iterates satisfy  $x^t\in[0,\frac{r_yD_{\cY}}{b}]$ and $y^t\in [0, D_{\cY}]$. Moreover, we have
    \begin{equation}\label{x_T-gs}
        x^t=(1+\lambda_1)^t x^0, \quad \mbox{and} \quad y^t= (1+\lambda_1)^t y^0.
    \end{equation}
\end{lemma}

The proof of \cref{gs-alter-tight} is in \cref{pf-of-eigen}.
With \cref{lemma:lamb-order-1,gs-alter-tight}, we are now ready to prove
\cref{thm:tight-alter}~(i).
\begin{proof}[Proof of \cref{thm:tight-alter}~(i).]
We initialize \emph{Perturbed GDA} at $x^0=\frac{2\epsilon}{\ell}\in\R$ and $y^0=\frac{2b(1+\ell c)\alpha}{\sqrt{A_1^2-4B_1}+2\ell c-A_1}\cdot\frac{2\epsilon}{\ell}\in[0,D_{\cY}]$. If $(x^T,y^T)$ is an $\epsilon$-GS, then by definition,
\begin{align*}
\epsilon
&\ge  |\nabla_x f(x^T,y^T)| = |b y^T-\ell x^T| \\
&= \left|
\left(
\frac{2b^2   (1+\ell c)\alpha}{\sqrt{A_1^2-4B_1}+2\ell c-A_1}-\ell
\right)x^T
\right| \\
&\ge \ell |x^T|
= \ell (1+\lambda_1)^T |x^0|
= 2\epsilon (1+\lambda_1)^T ,
\end{align*}
where the second equality uses~\eqref{x_T-gs} and the definition of $y^0$,
and the second inequality follows from
\[
\frac{2b^2(1+\ell c)\alpha}{\sqrt{A_1^2-4B_1}+2\ell c-A_1}
>
\frac{2b^2(1+\ell c)\alpha}{2(\ell c-A_1)}
=
\frac{3+3\ell c}{1+3\ell c}\cdot\ell
> 2\ell ,
\]
using the fact that $3\ell c=\Theta\left(\frac{\epsilon^2}{\ell^2 D_{\cY}^2}\right)<1$.

Therefore, we have 
\[
(1+\lambda_1)^T \le \frac12
\quad\Longrightarrow\quad
T=\Omega\left(\frac{1}{|\lambda_1|}\right)
=\Omega\left(\frac{\ell^2 D_{\cY}^2}{\epsilon^2}\right),
\]
where the last equality follows from \cref{lemma:lamb-order-1}. Finally, \cref{lemma:Delta_1} gives $\Delta_{\Psi_1}=\cO\left(\frac{\epsilon^2}{\ell}\right)$ under our initialization, and thus
\begin{align*}
    T&=\Omega\left(\frac{\ell^2D_{\cY}^2}{\epsilon^2}\right)=\Omega\left(\frac{\ell^2D_{\cY}^2}{\epsilon^2}\cdot\frac{\Delta_{\Psi_1}}{\epsilon^2/\ell}\right)=\Omega\left(\frac{\ell^3D_{\cY}^2\Delta_{\Psi_1}}{\epsilon^4}\right).
\end{align*}
This completes the proof.
\end{proof}
\subsubsection{Perturbed Smoothed GDA}\label{subsubsec:tight-psgda}
In this subsection, we prove \cref{thm:tight-psgda}~(i) with a similar approach as that for \emph{Perturbed GDA}.
We analyze the performance of \emph{Perturbed Smoothed GDA} on the hard instance \cref{eg:gs-alter-tight} and derive a lower bound on the number of iterations it needs to reach an $\epsilon$-GS.
The main technical difficulty here is that \emph{Perturbed Smoothed GDA}  introduces an additional auxiliary variable $z$ (via smoothing), so the middle-branch (when $0\le x\le \frac{r_yD_{\mathcal{Y}}}{b}$) updates of \cref{eg:gs-alter-tight} are governed by a three-dimensional linear recursion. 
Consequently, unlike \cref{subsubsec:pgda-tight-gs} where one only needs to control the spectrum of a $2\times2$ matrix, we must analyze the eigenvalues of a $3\times3$ matrix that couples $(x^t,y^t,z^t)$. 
Analyzing the contraction factor of this $3\times3$ matrix is the crux of the proof.

In the following,  we choose the step sizes $\alpha,c,\beta$ to satisfy \cref{con-psgda} and set $r_y = \Theta(\frac{\epsilon}{D_{\mathcal{Y}}})$.
We show that any iterate $(x^T,y^T,z^T)$ generated by \emph{Perturbed Smoothed GDA}  that is an $\epsilon$-GS must satisfy $
T = \Omega\left(\frac{\ell^2 D_{\mathcal{Y}} \Delta_{\Psi_2}}{\epsilon^3}\right).$ 

We analyze the dynamics of \emph{Perturbed Smoothed GDA}  on \cref{eg:gs-alter-tight} in the middle branch $x^t\in\bigl[0,\frac{r_yD_{\cY}}{b}\bigr]$, where the update takes the linear form:
\[
\left(\!
\begin{array}{c}
 x_{t+1}  \\
 y_{t+1}  \\
 z_{t+1}
\end{array}
\!\right)\!=\!(\bm{I}+\bm{M_2})\!\left(\!
\begin{array}{c}
 x_{t}  \\
 y_{t}  \\
 z_{t}
\end{array}
\!\right),~\mbox{where}~
    \bm{M_2}\!=\!\left(\!
\begin{array}{ccc}
 c\ell-cr_x & -{cb} & {c r_x} \\
 \alpha b(1+c\ell-cr_x) & -r_y \alpha -{c\alpha b}^2 & {c r_x \alpha b} \\
 \beta (1+c\ell-cr_x) & -cb\beta  & -\beta +\beta c r_x \\
\end{array}
\!\right).
\]
Similar to the previous subsection, we study the eigenvalues of $\bm{M_2}$ (equivalently, those of $\bm{I}+\bm{M_2}$) to quantify the contraction rate along a specific eigendirection. 
In particular, one eigenvalue admits the following closed form,\footnote{The expression and the subsequent algebraic simplifications were obtained using symbolic computations in \textsc{Mathematica}. The corresponding code is available at 
\url{https://github.com/Smoothing-Meets-Perturbation/A-Unified-and-Tight-Analysis-for-NC-C.git}.}
\begin{align}
\lambda_2=& \frac{1}{3} B_2+ \frac{\left( 1 + i \sqrt{3} \right) \left( B_2^2+3C_2\right)}{ 3 \cdot 2^{\frac{2}{3}} D_2^\frac{1}{3}}+\frac{1}{6\cdot2^\frac{1}{3}} (1 - i \sqrt{3})D_2^\frac{1}{3},\label{eq:lamb-2}
\end{align}
where
\begin{align*}
    A_2:= & -2 c^3 \ell^3 + 6 c^3 \ell^2 r_x - 6 c^3 \ell r_x^2 + 2 c^3 r_x^3 + 9 b^2 c^2 \ell \alpha \\
& \left. \left. + 6 b^2 c^3 \ell^2 \alpha - 9 b^2 c^2 r_x \alpha + 12 b^2 c^3 \ell r_x \alpha - 6 b^2 c^3 r_x^2 \alpha \right. \right. \\
& \left. \left. - 9 b^4 c^2 \alpha^2 - 6 b^4 c^3 \ell \alpha^2 + 6 b^4 c^3 r_x \alpha^2 + 2 b^6 c^3 \alpha^3 - 3 c^2 \ell^2 \beta \right. \right. \\
& \left. \left. - 3 c^2 \ell r_x \beta - 6 c^3 \ell^2 r_x \beta + 6 c^2 r_x^2 \beta + 12 c^3 \ell r_x^2 \beta - 6 c^3 r_x^3 \beta \right. \right. \\
& \left. \left. + 18 b^2 c \alpha \beta + 6 b^2 c^2 \ell \alpha \beta + 12 b^2 c^2 r_x \alpha \beta + 12 b^2 c^3 \ell r_x \alpha \beta \right. \right. \\
& \left. \left. - 12 b^2 c^3 r_x^2 \alpha \beta - 3 b^4 c^2 \alpha^2 \beta - 6 b^4 c^3 r_x \alpha^2 \beta + 3 c \ell \beta^2 \right. \right. \\
& \left. \left. + 6 c r_x \beta^2 + 3 c^2 \ell r_x \beta^2 - 12 c^2 r_x^2 \beta^2 - 6 c^3 \ell r_x^2 \beta^2 + 6 c^3 r_x^3 \beta^2 \right. \right.\\
& \left. \left. - 3 b^2 c \alpha \beta^2 - 3 b^2 c^2 r_x \alpha \beta^2 + 6 b^2 c^3 r_x^2 \alpha \beta^2 + 2 \beta^3 \right. \right. \\
& \left. \left. - 6 c r_x \beta^3 + 6 c^2 r_x^2 \beta^3 - 2 c^3 r_x^3 \beta^3 - 3 c^2 \ell^2 \alpha r_y \right. \right. \\
& \left. \left. + 6 c^2 \ell r_x \alpha r_y - 3 c^2 r_x^2 \alpha r_y - 9 b^2 c \alpha^2 r_y - 3 b^2 c^2 \ell \alpha^2 r_y \right. \right. \\
& \left. \left. + 3 b^2 c^2 r_x \alpha^2 r_y + 6 b^4 c^2 \alpha^3 r_y - 12 c \ell \alpha \beta r_y - 6 c r_x \alpha \beta r_y \right. \right. 
\end{align*}
\begin{align*}
& \left. \left. - 6 c^2 \ell r_x \alpha \beta r_y + 6 c^2 r_x^2 \alpha \beta r_y - 6 b^2 c \alpha^2 \beta r_y - 3 b^2 c^2 r_x \alpha^2 \beta r_y \right. \right. \\
& \left. \left. - 3 \alpha \beta^2 r_y + 6 c r_x \alpha \beta^2 r_y - 3 c^2 r_x^2 \alpha \beta^2 r_y\right. \right.\\
& \left. \left. + 3 c \ell \alpha^2 r_y^2 - 3 c r_x \alpha^2 r_y^2 + 6 b^2 c \alpha^3 r_y^2 - 3 \alpha^2 \beta r_y^2 \right. \right. \\
& \left. \left. + 3 c r_x \alpha^2 \beta r_y^2 + 2 \alpha^3 r_y^3, \right. \right.\\
B_2:=&\,c \ell - c r_x - b^2 c \alpha - \beta + c r_x \beta - \alpha r_y ,\\
C_2:=&-b^2 c \alpha + c \ell \beta - b^2 c \alpha \beta + c \ell \alpha r_y - c r_x \alpha r_y - \alpha \beta r_y + c r_x \alpha \beta r_y,~\mbox{and}\\
D_2:=&\,A_2+\sqrt{A_2^2-4(B_2^2+3C_2)^3}.
\end{align*}
Similar to \cref{subsubsec:pgda-tight-gs}, we summarize the sign and order of the eigenvalue $\lambda_2$ in the next lemma and its proof is deferred to \cref{sec:lam_2}.
\begin{lemma}\label{lemma:lamb-order-2}
    The eigenvalue $\lambda_2$ satisfies that $\lambda_2<0$ and $|\lambda_2|=\Theta\left(\frac{\epsilon}{\ell D_{\cY}}\right)$.
\end{lemma}
\begin{remark}
In the analysis for \emph{Perturbed GDA}, we have $|\lambda_1|=\Theta\left(\frac{\epsilon^2}{\ell^2D_{\cY}^2}\right)$ (cf. \cref{lemma:lamb-order-1}), whereas for \emph{Perturbed Smoothed GDA}, we obtain $|\lambda_2|=\Theta\left(\frac{\epsilon}{\ell D_{\cY}}\right)$ (cf. \cref{lemma:lamb-order-2}). Note that the contraction factor is $1+\lambda_i=1-\Theta(|\lambda_i|)$ ($i=1,2$), and a larger $|\lambda_2|$ leads to faster contraction.
Thus these two algorithms have different dependence on $\epsilon$.
\end{remark}

 Next, to explicitly characterize the iterates, we choose $(x^0, y^0, z^0)$ so that its direction aligns with an eigenvector of $\bm{I}+\bm{M_2}$ corresponding to the eigenvalue $\lambda_2$.
\begin{lemma}\label{gs-psgda-tight}
Suppose that \emph{Perturbed Smoothed GDA}  is initialized at
\begin{equation}\label{eq:initial:psgda}
    x^0=\frac{2(\beta+\lambda_2)\epsilon}{r_x\lambda_2},~y^0=\frac{c \ell \beta + c \ell \lambda_2 - c r_x \lambda_2 - \beta \lambda_2 + c r_x \beta \lambda_2 - \lambda_2^2}{bc(\beta+\lambda_2)}\cdot x^0,~z^0=\frac{\beta(1+\lambda_2)}{\beta+\lambda_2}\cdot x^0.
\end{equation}
Then, for all $t\in\mathbb N$, the iterates satisfy  $x^t\in[0,\frac{r_yD_{\cY}}{b}]$ and $y^t\in [0, D_{\cY}]$. Moreover, we have
    \begin{equation*}
        x^t=(1+\lambda_2)^tx^0,\, y^t= (1+\lambda_2)^t y^0,~\mbox{and } z^t= (1+\lambda_2)^t z^0.
    \end{equation*}
\end{lemma}
The proof of \cref{gs-psgda-tight} is given in \cref{sec:ite}. 
With \cref{lemma:lamb-order-2,gs-psgda-tight}, we can prove
\cref{thm:tight-psgda}~(i). Because the argument is almost the same as the proof of \cref{thm:tight-alter}~(i), we defer the details to \cref{app:pf-tight-psgda-gs}. 

\subsection{Optimization Stationarity}\label{sec:analysis-OS}
With the tightness results for game stationarity established, we now turn to optimization stationarity.
In this subsection, we prove \cref{thm:tight-alter}~(ii) and \cref{thm:tight-psgda}~(ii), showing that the iteration complexities for finding an $\epsilon$-OS in \cref{thm:pgda-nc-c-epi}~(ii) and \cref{thm:psgda-epi}~(ii) are tight.
We first present a hard instance, and then show that \emph{Perturbed GDA} and \emph{Perturbed Smoothed GDA}  require at least the stated numbers of iterations to reach an $\epsilon$-OS on this instance.
\begin{example}\label{eg:os-gda}
Let $\cX=\R$ and $\cY=[0,D_{\cY}]$. Define $f:\cX\times\cY\to\R$ by
\begin{equation*}
    f(x,y)=h(x)\,y,
\end{equation*}
where
\begin{gather}\label{hfunc:altergda-tight-os}
    h(x)=
    \begin{cases}
        \frac{\ell}{2(D_{\cY}+1)}x^2, & ~\mbox{if}~\lvert x \rvert \le 1, \\
        \frac{\ell}{D_{\cY}+1}-\frac{\ell}{2(D_{\cY}+1)}(\lvert x \rvert-2)^2, & ~\mbox{if}~1 < \lvert x \rvert < 2, \\
        \frac{\ell}{D_{\cY}+1}, & ~\mbox{if}~  \lvert x \rvert \ge 2.
    \end{cases}
\end{gather}
\end{example}
\begin{remark}

The construction in \cref{eg:os-gda} is adapted from the hard instance in 
\citep[Appendix B.3.1]{mahdavinia2022tight}. 
We rescale the coefficients so that the resulting function satisfies the $\ell$-smoothness condition. 
The smoothness verification is provided in \cref{lem:smooth}.
\end{remark}
    \subsubsection{Perturbed GDA}\label{subsubsec:pgda-os-tight}
We prove \cref{thm:tight-alter}~(ii) in this subsection.
Based on \cref{eg:os-gda}, we show that \emph{Perturbed GDA} requires at least the stated number of iterations in \cref{thm:pgda-nc-c-epi}~(ii) to find an $\epsilon$-OS.
Our analysis relies on the fact established in \cref{lemma:x-lessthan-1}: Under a suitable initialization and step-size choice, the iterates satisfy $|x^t|\le 1$ and $y^t\ge 0$ for all $t$. 
This is useful because, by \cref{hfunc:altergda-tight-os}, $h$ is convex on $|x|\le 1$, and therefore $f(x,y^t)=y^t h(x)$ is convex in $x$ when $y^t\ge 0$.

\begin{lemma}\label{lemma:x-lessthan-1}
Let the \emph{Perturbed GDA} be initialized at $(x^0,y^0)$ with $|x^0|\le 1$ and $y^0\ge 0$.
We choose the step sizes to satisfy \cref{con-pgda}. 
Then, when \emph{the Perturbed GDA} is applied on \cref{eg:os-gda}, for all $t\ge 0$,
\begin{enumerate}[label=(\roman*)]
\item $\lvert x^t\rvert \leq 1$ and $y^t \geq 0$;
\item $\lvert x^t\rvert \geq \left(1-\frac{\ell c D_{\cY}}{D_{\cY}+1}\right)^t\lvert x^0\rvert$.
\end{enumerate}
\end{lemma}
\begin{proof}[Proof of \cref{lemma:x-lessthan-1}.]
We prove the two claims by induction using the update rule of \emph{Perturbed GDA} on  \cref{eg:os-gda}.  The claims hold for $t=0$ by the hypothesis.

Suppose for $t\ge0$, we have $|x^t|\le 1$. Then by \cref{hfunc:altergda-tight-os} we have
$h'(x^t)=\frac{\ell}{D_{\cY}+1}x^t$, and therefore
\begin{equation}\label{x-ite}
    x^{t+1}
    =x^t-c\nabla_{x}f_0(x^t,y^t)
    =x^t-c\,h'(x^t)\,y^t
    =\left(1-\frac{\ell c}{D_{\cY}+1}y^t\right)x^t.
\end{equation}
(i) First of all, we have $y^t\ge 0$  for all $t$ since $\cY=[0,D_{\cY}]$.
We next prove $|x^t|\le 1$ for all $t$ by induction. 
Assume $|x^t|\le 1$ for some $t\ge 0$. Using $0\le y^t\le D_{\cY}$ and the step-size condition
$0<c\le \frac{r_{y}^2\alpha}{16\ell^2}$ with $\alpha\le \frac{1}{4(\ell+r_y)}$, we have
\[
0\le \frac{\ell c}{D_{\cY}+1}y^t\le \frac{\ell c}{D_{\cY}+1}D_{\cY}
\le \ell c 
\le \frac{r_{y}^2\alpha}{16\ell}
\le \frac{r_{y}^2}{64\ell(\ell+r_y)}<1,\]
where the last inequality follows from the choice of $r_{y}=\Theta(\frac{\epsilon^2}{\ell D_{\cY}^2})$. 
Then, it implies $0\le 1-\frac{\ell c}{D_{\cY}+1}y_t\le 1$.
Combining this with \cref{x-ite} yields $|x_{t+1}|\le |x_t|\le 1$. This completes the induction. 

(ii) 
Since $0\le y^t\le D_{\cY}$, from \cref{x-ite} we have
\[
|x^{t+1}|
=\left(1-\frac{\ell c}{D_{\cY}+1}y^t\right)|x^t|
\ge \left(1-\frac{\ell c D_{\cY}}{D_{\cY}+1}\right)|x^t|.
\]
Iterating the above inequality for $t$ steps gives (ii). This completes the proof.
\end{proof}

With \cref{lemma:x-lessthan-1}, we are ready to prove \cref{thm:tight-alter}~(ii).

\begin{proof}[Proof of \cref{thm:tight-alter}~(ii).]
By \cref{defi:primal-dual} (i) , we have that $2\ell |\prox_{\frac{1}{2\ell} \Phi}(x)-x|\leq \epsilon$ when $x$ is an $\epsilon$-OS.
Moreover, by \citep[Lemma 2.2]{davis2019stochastic}, we have
\[
2\ell\left|\prox_{\frac{1}{2\ell}\Phi}(x)-x\right|=\left|\nabla \Phi_{\frac{1}{2\ell}}(x)\right|.
\]
When $\lvert x \rvert \leq 1$, a direct computation yields that 
\begin{equation}\label{eq:phix-ME-1}
    \Phi_\frac{1}{2\ell}(x)=\min_{z\in\R} \Phi(z)+\ell (z-x)^2 =\min_{z\in\R}\frac{\ell D_{\cY}}{2(D_{\cY}+1)}z^2+\ell (z-x)^2= \frac{\ell D_{\mathcal{Y}}}{3 D_{\mathcal{Y}}+ 2} x^2.
\end{equation}
We fix $\epsilon>0$ small enough so that $x^0:=\frac{3 D_{\mathcal{Y}}+ 2}{\ell D_{\mathcal{Y}}}\epsilon \le 1$, and
run \emph{Perturbed GDA} on \cref{eg:os-gda} initialized at $(x^0,y^0)$ with $y^0=D_{\cY}$. If $x^T$ is an $\epsilon$-OS, then
\begin{align*}
    \epsilon
    \ge &\, \left| \nabla \Phi_{\frac{1}{2\ell}}(x^T)\right|
    \overset{\cref{eq:phix-ME-1}}{=}\frac{2\ell D_{\mathcal{Y}}}{3 D_{\mathcal{Y}}+ 2} |x^T|\ge \frac{2\ell D_{\mathcal{Y}}}{3 D_{\mathcal{Y}}+ 2} \left(1-\frac{\ell c D_{\cY}}{D_{\cY}+1}\right)^T |x^0|
    = 2\epsilon  \left(1-\frac{\ell c D_{\cY}}{D_{\cY}+1}\right)^T,
\end{align*}
where the second inequality uses \cref{lemma:x-lessthan-1}   (ii).
Therefore, we have 
\begin{equation}\label{eq:T-lb-os-gda-14}
\left(1-\frac{\ell c D_{\cY}}{D_{\cY}+1}\right)^T \le \frac12
\quad\Longrightarrow\quad
T=\Omega\left(\frac{D_{\cY}+1}{\ell c D_{\cY}}\right),
\end{equation}
    
In \cref{lem:delta1}, we show that under the initialization
$(x^0,y^0)=(\tfrac{3 D_{\mathcal{Y}}+ 2}{\ell D_{\mathcal{Y}}}\epsilon,D_{\cY})$, the initial gap satisfies
$\Delta_{\Psi_1}=\cO\left(\frac{D_{\cY}+1}{\ell D_{\cY}}\epsilon^2\right)$.
Substituting this bound into \cref{eq:T-lb-os-gda-14} yields
    \[
    T=\Omega\left(\frac{D_{\cY}+1}{\ell c D_{\cY}}\right)=\Omega\left(\frac{D_{\cY}+1}{\ell c D_{\cY}}\cdot\frac{ \ell D_{\cY}{\Delta}_{\Psi_1}}{(D_{\cY}+1)\epsilon^2}\right)=\Omega\left(\frac{\ell^5D_{\cY}^4\Delta_{\Psi_1}}{\epsilon^6}\right),
    \]
where the last equality follows from $c=\Theta\left(\frac{r_{y}^2}{\ell^3}\right)=\Theta\left(\frac{\epsilon^4}{\ell^5 D_{\cY}^4}\right)$.
This completes the proof.
\end{proof}
 
\subsubsection{Perturbed Smoothed GDA} \label{subsubsec:psgda-tight}
In this subsection, we use \cref{eg:os-gda} to prove \cref{thm:tight-psgda} (ii), showing that the iteration complexity bound for \emph{Perturbed Smoothed GDA} in \cref{thm:psgda-epi} (ii) is tight.  

The key idea here differs from that in \cref{subsubsec:tight-psgda}. 
A direct analysis of the coupled three-variable dynamics $(x^t,y^t,z^t)$ and their spectral properties is technically cumbersome.
To address this issue, we restrict attention to trajectories for which $x^t$ remains in the region $|x^t|\le 1$. Moreover, we introduce an auxiliary (comparison) sequence in which the dual variable is frozen on the boundary, i.e., $y^t \equiv D_{\cY}$ for all $t$.  
This modification yields a simplified update rule that is no slower than \emph{Perturbed Smoothed GDA} in terms of reaching an $\epsilon$-OS solution. 
Consequently, a lower bound on the iteration complexity of this auxiliary process also applies to \emph{Perturbed Smoothed GDA}.

When $|x^t|\leq 1$, substituting $h'(x^t)=\frac{\ell}{D_{\cY}+1}x^t$ into the \emph{Perturbed Smoothed GDA} update yields the equivalent recursion:
    \begin{align*}
	x^{t+1}&=\left(1-cr_x-\frac{\ell c}{D_{\cY}+1}y^t\right)x^t+cr_xz^t,\\
	y^{t+1}&=\proj_{[0,D_{\cY}]}\left(y^t+\alpha (h(x^{t+1})-r_yy^t)\right),~\mbox{and}\\
	z^{t+1}&=\beta\left(1-cr_x-\frac{\ell c}{D_{\cY}+1}y^t\right)x^t+(1-\beta+\beta cr_x)z^t.
    \end{align*} 

To formalize the comparison argument, we introduce an auxiliary sequence
$\{\bar{x}^t,\bar{y}^t,\bar{z}^t)\}_{t\ge 0}$ that follows the same recursion as
\emph{Perturbed Smoothed GDA} except that the dual iterate is frozen on the boundary.
Specifically, we initialize $(\bar{x}^0, \bar{y}^0, \bar{z}^0)=(x^0, y^0, z^0)$ and for all $t\ge 0$, we define
\begin{equation}\label{eq:new_update_rule}
\begin{aligned}
\bar{x}^{t+1} &= \left(1-cr_x-\frac{\ell c}{D_{\cY}+1}D_{\cY}\right)\bar{x}^t + cr_x \bar{z}^t,\\
\bar{y}^{t+1} &= D_{\cY},\\
\bar{z}^{t+1} &= \beta\left(1-cr_x-\frac{\ell c}{D_{\cY}+1}D_{\cY}\right)\bar{x}^t
+ \left(1-\beta+\beta cr_x\right)\bar{z}^t.
\end{aligned}
\end{equation}
The following lemma shows the relationship between the auxiliary sequence $\{(\bar{x}^t,\bar{y}^t,\bar{z}^t)\}_{t\ge 0}$ and the original iterates $\{(x^t,y^t,z^t)\}_{t\ge 0}$.
\begin{lemma}\label{lemma:xyz-fix-y}
 Suppose that the initialization $(x^0,y^0,z^0)$ satisfies $0<x^0,z^0<1$ and $y^0=D_{\cY}$. 
Then, for all $t\ge 0$, it holds that
	 \begin{enumerate}[label=(\roman*)] 
	 	\item $0<\bar{x}^t\leq x^t\le1$,
	 	\item $0\le y^t\leq \bar{y}^{t}=D_{\cY}$,
            \item $0<\bar{z}^{t}\leq z^t\le1$.
	 \end{enumerate}
	\end{lemma}
	\begin{proof}[Proof of \cref{lemma:xyz-fix-y}.]
We prove the claim by induction. 
For the base case $t=0$, the claim holds by construction (the auxiliary sequence shares the same initialization as \emph{Perturbed Smoothed GDA}). 
Assume that the claim holds for some $t\ge 0$. We next show that it also holds for $t+1$.

        (i) We first show that $\bar{x}^{t+1}>0$. By the update rule of $\bar{x}^{t+1}$, it suffices to verify that $(1-cr_x-\tfrac{\ell c}{D_{\cY}+1}D_{\cY})>0$. Recall from \cref{con-psgda} that $c<\frac{1}{r_x+\ell+r_y}$. Therefore, 
        \[
1-cr_x-\frac{\ell c}{D_{\cY}+1}D_{\cY}
\;\ge\;
1-c(r_x+\ell)
\;>\;
0,
\]
which proves $\bar{x}^{t+1}>0$. Next, we prove that $\bar{x}^{t+1}\le x^{t+1}$. Using the update rule of $x^{t+1}$, we have 
\begin{align*}
x^{t+1}
&=\left(1-cr_x-\frac{\ell c}{D_{\cY}+1}y^t\right)x^t+cr_x z^t \\
&\ge \left(1-cr_x-\frac{\ell c}{D_{\cY}+1}D_{\cY}\right)\bar{x}^t+cr_x \bar{z}^{t} = \bar{x}^{t+1},
\end{align*}
where the inequality follows from $y^t\le D_{\cY}$, and the induction hypothesis $\bar{x}^t\le x^t$ and $\bar{z}^{t}\le z^t$.
We now show that $x^{t+1}\le 1$. Since $0<x^t,z^t\le1$ and $y^t\ge0$, it holds that
\begin{align*}
x^{t+1}
&=\left(1-cr_x-\frac{\ell c}{D_{\cY}+1}y^t\right)x^t+cr_x z^t \\
&\le \left(1-cr_x-\frac{\ell c}{D_{\cY}+1}y^t\right)\cdot 1 + cr_x \cdot 1 = 1-\frac{\ell c}{D_{\cY}+1}y^t \le 1. 
\end{align*}

(ii) Since $\bar{y}^{t}=D_{\cY}$ for all $t$, the bound $0\le y^t\le \bar{y}^{t}=D_{\cY}$ holds by the projection in the $y$-update rule. 
        
   (iii) The proof of $0<\bar{z}^{t}\le z^t\le 1$ follows the same argument as in part~(i), and is therefore omitted.

     This completes the proof.
	\end{proof}

Next, we analyze the auxiliary update rule and prove a lower bound on its iteration complexity, which in turn implies a lower bound on the iteration complexity of \emph{Perturbed Smoothed GDA} on \cref{eg:os-gda}. Note that $\bar{y}^{t}$ remains constant under the auxiliary update rule \cref{eq:new_update_rule}. 
Hence, it suffices to study the evolution of the pair $(\bar{x}^t,\bar{z}^{t})$. Define $\gamma:=\frac{D_{\cY}}{D_{\cY}+1}$ and introduce the matrix 
 \[
\bm{M_3}:=\begin{pmatrix}
-\ell c\gamma-cr_x & cr_x \\
\beta(1-\ell c\gamma-cr_x) & -\beta+\beta cr_x
\end{pmatrix}.
\]
Then the recursion in \cref{eq:new_update_rule} can be written as
         \[
\begin{pmatrix}
\bar{x}^{t+1}\\
\bar{z}^{t+1}
\end{pmatrix}
=
(\bm{I}+\bm{M_3})
\begin{pmatrix}
\bar{x}^t\\
\bar{z}^{t}
\end{pmatrix}.
\] 
	     
To lower bound the iteration complexity of the updates rule in \cref{eq:new_update_rule}, we study 
the eigenvalues of $\bm{M_3}$, which governs the growth rate of the iterates. One can verify that $\bm{M_3}$ admits an eigenvalue 
\begin{equation}\label{lambda_OS}
	\lambda_3=-\frac{A_3+B_3\beta}{2}+ \frac{1}{2}\sqrt{(A_3+B_3\beta)^2-4\ell c\beta\gamma},
	\end{equation}
    where
    \[A_3:=\ell c \gamma+cr_x,
\quad \text{and} \quad B_3:=1-cr_x.\]
The following \cref{lemma:eigen-value-M,x_0withz_0} summarize the magnitude of $\lambda_3$ and the structure of its associated eigenvector. Their proofs are deferred to \cref{sec:lem:eig,sec:lem:x0}, respectively.
\begin{lemma}\label{lemma:eigen-value-M}
The eigenvalue $\lambda_3$ satisfies that $\lambda_3<0$ and $|\lambda_3|=\Theta\left(\frac{\gamma\epsilon^2}{\ell^2 D_{\cY}^2}\right)$.
	\end{lemma}
\begin{lemma}\label{x_0withz_0}
Let $\bm{v}=(v_1,1)^\top$ be an eigenvector of $\bm{M_3}$ associated with the eigenvalue $\lambda_3$. 
Then
\begin{equation}\label{eq:v1-expression}
v_1=\frac{r_x}{r_x+\gamma \ell}+\cO(\beta).
\end{equation}
\end{lemma} 

Equipped with these lemmas, we are now ready to prove \cref{thm:tight-psgda} (ii).

\begin{proof}[Proof of \cref{thm:tight-psgda} (ii).]
We consider \emph{Perturbed Smoothed GDA} initialized at $(x^0,y^0,z^0)$, where the pair $(x^0,z^0)$ is chosen as an eigenvector associated with the eigenvalue $\lambda_3$, and $y^0 = D_{\cY}$. In the following analysis, only the $x$-component of this eigenvector is needed. Specifically,  $x^0=\frac{3 D_{\mathcal{Y}}+ 2}{\ell D_{\mathcal{Y}}}\epsilon$. 
Let $T$ denote the first iteration such that $x^T$ is an $\epsilon$-OS point. 
By definition of an $\epsilon$-OS point, we have $\epsilon \geq \lvert \nabla\Phi_{\frac{1}{2\ell}}(x^T)\rvert$. Therefore, 
\begin{align*}
\epsilon 
&\ge \left\lvert \nabla \Phi_{\frac{1}{2\ell}}(x^T)\right\rvert \\
&= \frac{2\ell D_{\cY}}{3D_{\cY}+2}\,|x^T| \\
&\ge \frac{2\ell D_{\cY}}{3D_{\cY}+2}\,|\bar{x}^T| \ge \frac{2\ell D_{\cY}}{3D_{\cY}+2}\,(1+\lambda_3)^T |\bar{x}^0| = 2\epsilon\,(1+\lambda_3)^T, 
\end{align*} 
where the first equality follows from \eqref{eq:phix-ME-1}, 
the second inequality follows from \cref{lemma:xyz-fix-y}, 
and the third inequality follows from the fact that $(\bar{x}_0,\bar{z}_0)$ is an eigenvector corresponding to $\lambda_3$.
It implies that 
\[
(1+\lambda_3)^T \le \frac12
\quad\Longrightarrow\quad
T=\Omega\left(\frac{1}{|\lambda_3|}\right),
\]
Moreover, by \cref{lemma:Delta-Psi2}, under the initialization $y^0=D_{\cY}$ and $(x^0,z^0)$ chosen as an eigenvector associated with $\lambda_3$ with $x^0=\frac{3 D_{\mathcal{Y}}+ 2}{\ell D_{\mathcal{Y}}}\epsilon$, we have $\Delta_{\Psi_2}=\cO\left(\frac{\epsilon^2}{\ell\gamma}\right).$ Combining this estimate with \cref{lemma:eigen-value-M}, we further obtain  
	\[T=\Omega\left(\frac{1}{|\lambda_3|}\right)=\Omega\left(\frac{\ell^2D_{\cY}^2}{\gamma\epsilon^2}\cdot\frac{\ell D_{\cY}}{(D_{\cY}+1)\epsilon^2}\cdot \Delta_{\Psi_2}\right)=\Omega\left(\frac{\ell^3D_{\cY}^2\Delta_{\Psi_2}}{\epsilon^4}\right).\]
 This completes the proof.
    \end{proof}	

\section{Double-loop Algorithm: Perturbed Smoothed FOAM}
\label{sec:double}
Our analysis for the iteration complexity of \emph{Perturbed Smoothed GDA} on $\epsilon$-GS shows the synergistic effect between smoothing and perturbation.
In this section, we show that this synergy between smoothing and perturbation is actually structural to the NC-C minimax problem and can extend naturally to a double-loop algorithmic framework as well. 
Specifically, we develop a double-loop \emph{Perturbed Smoothed FOAM} algorithm that uses both smoothing and perturbation techniques. 
We prove that its iteration complexity of finding an $\epsilon$-GS and an $\epsilon$-OS point is
$\cO\left(\epsilon^{-2.5}\log(\tfrac{1}{\epsilon})\right)$ and $\cO\left(\epsilon^{-3}\log(\tfrac{1}{\epsilon})\right)$, respectively. 
To the best of our knowledge, these rates are state-of-the-art among all first-order methods  for NC-C minimax problems.


The idea of \emph{Perturbed Smoothed FOAM} is as follows. 
We employ the same surrogate function $F_t$ defined in \cref{eq:tilde-F} as in \emph{Perturbed Smoothed GDA}. 
Importantly, the combination of smoothing and dual perturbation transforms the original NC-C problem into a sequence of strongly convex-strongly concave (SC-SC) subproblems. 
In particular, for any fixed $\z$, the function $F_t(\cdot,\cdot,\z)$ is strongly convex in $\x$ and strongly concave in $\y$. Unlike \emph{Perturbed Smoothed GDA}, which performs gradient updates on $(\x^t,\y^t,\z^t)$ in each iteration, 
\emph{Perturbed Smoothed FOAM} adopts a double-loop scheme. 
Specifically, at each iteration $t$, we first fix $\z^t$ and compute $(\x^{t+1},\y^{t+1})$ as an approximate solution to the SC-SC subproblem
\begin{equation}\label{eq:tilde-p}
    p_t(\z^t)=\min_{\bx\in\cX}\max_{\by\in\cY} F_t(\bx,\by,\z^t),
\end{equation}
and then perform a gradient step on $\z^t$.
To solve the inner SC-SC problem on $\x$ and $\y$, we use the single-loop algorithm \emph{FOAM} proposed by \citep{kovalev2022first}. 
Specifically, we let $\mbox{FOAM}(h,\x^0,\y^0,\mu_{\x},\mu_{\y},\delta)$ denote the output of \emph{FOAM} algorithm applied to a SC-SC function $h$. Here,  $(\x^0,\y^0)$ is the initial point, and $\mu_{\x}$ and $\mu_{\y}$ denote the strong convexity modulus of $h(\x,\y)$ in $\x$ and the strong concavity modulus of $h(\x,\y)$ in $\y$, respectively.
\emph{FOAM} returns a $\delta$-accurate solution (cf.~\citep[Definition~1]{kovalev2022first}) in $\cO\left(\sqrt{\kappa_{\x}\kappa_{\y}}\log(\delta^{-1})\right)$ iterations, where $\kappa_{\x}$ and $\kappa_{\y}$ are the condition numbers of $h$ with respect to $\x$ and $\y$, respectively.
We formally present \emph{Perturbed Smoothed FOAM} in \cref{alg:Smoothed-FOAM}.
\begin{algorithm}[hbt!]
		\caption{Perturbed Smoothed FOAM}\label{alg:Smoothed-FOAM}
		\KwData{Initial point $(\x^0,\y^0,\z^0)$, and parameters $r_{\x}>\ell,r_{\y}^t>0$, $\beta^t\in(0,1)$, and $\delta^t>0$}
		\For{$t=0,\ldots,T$}{
        define $g_t(\x,\y):=F_t(\x,\y,\z^{t})+\iota_\cX(\x)-\iota_\cY(\y)$;\\
	    $(\x^{t+1},\y^{t+1}):=\FOAM\left(g_t,\x^{t},\y^{t},r_{\x}-\ell,r_{\y}^t,\delta^t\right)$;\\
        $\z^{t+1}=\z^{t}+\beta^t(\x^{t+1}-\z^{t})$;
		}
\end{algorithm}

Next, we establish the iteration complexity of \emph{Perturbed Smoothed FOAM}. We begin by specifying the step-size conditions for the \emph{Perturbed Smoothed FOAM} algorithm in \cref{con-FOAM}.

\begin{condition}[Parameter conditions for Perturbed Smoothed FOAM]\label{con-FOAM}
Let $r_{\y}^t > 0$ for all $t \ge 0$. The step sizes are chosen to satisfy
\[
\delta^t = \cO\left(\frac{(r_{\y}^t)^{4} D_{\mathcal{Y}}^{2}}{\ell^{4}}\right),
\quad
r_{\x} = \Theta(\ell) > 3\ell,
\quad
\beta^t = \Theta(1), \quad \text{and} \quad 1-\beta^t = \Theta(1). 
\]

\end{condition}
We prove the iteration complexity of \cref{alg:Smoothed-FOAM} using the unified analysis framework developed in \cref{sec:framework}. 
In this setting, we use $p_t(\z)$ defined in \cref{eq:tilde-p} as the Lyapunov function and establish its basic descent property in the following lemma.



\begin{lemma}[Basic descent estimate]\label{lem:p-est}
 Suppose  \cref{con-FOAM} holds, and let $\{(\x^{t},\y^{t},\z^t)\}_{t \ge 0}$ be the sequence generated by  \emph{Perturbed Smoothed FOAM}. 
 Then, for any $t\ge0$, we have
    \begin{align*}
        p_{t+1}(\z^{t+1})-p_t(\z^{t}) \le&\,-\frac{\beta^t(1-\beta^t)}{r_{\x}}\|\nabla p_t(\z^{t})\|^2+\beta^t\|\nabla p_t(\z^{t})\|\cdot\left\|\x^{t+1}-\x_{t}^\star(\z^{t})\right\|\\
        &\,+r_{\x}(\beta^t)^2\left\|\x^{t+1}-\x_{t}^\star(\z^{t})\right\|^2+\frac{r_{\y}^t-r_{\y}^{t+1}}{2}D_{\cY}^2.
    \end{align*}
\end{lemma}
\begin{proof}[Proof of \cref{lem:p-est}.]
We bound the descent of $p_t$ as follows:
\begin{equation}
    p_{t+1}(\z^{t+1})-p_t(\z^{t})=(p_{t+1}(\z^{t+1})-p_t(\z^{t+1}))+(p_{t}(\z^{t+1})-p_t(\z^{t})).\label{eq:des:p}
\end{equation}
We first bound the first term on the right-hand side (RHS). 
Since
\[
    F_{t+1}(\x,\y,\z)
    =
    F_t(\x,\y,\z)
    +
    \frac{r_{\y}^t-r_{\y}^{t+1}}{2}\|\y\|^2,
\]
and \(r_{\y}^t\ge r_{\y}^{t+1}\), for any fixed \(\z\) we have
\[
\begin{aligned}
p_{t+1}(\z)
&=
\min_{\x\in\cX}\max_{\y\in\cY}
\left\{
F_t(\x,\y,\z)
+
\frac{r_{\y}^t-r_{\y}^{t+1}}{2}\|\y\|^2
\right\} \\
&\le
\min_{\x\in\cX}
\left\{
\max_{\y\in\cY} F_t(\x,\y,\z)
+
\frac{r_{\y}^t-r_{\y}^{t+1}}{2}D_{\cY}^2
\right\}\\
&=
p_t(\z)
+
\frac{r_{\y}^t-r_{\y}^{t+1}}{2}D_{\cY}^2 .
\end{aligned}
\]
Applying this bound with \(\z=\z^{t+1}\) gives
\[
    p_{t+1}(\z^{t+1})-p_t(\z^{t+1})
    \le
    \frac{r_{\y}^t-r_{\y}^{t+1}}{2}D_{\cY}^2.
\]
For the second term on the RHS of \cref{eq:des:p}, by \citep[Proposition 12.29]{bauschke2011convex}, the function $p_t$ is $r_{\x}$-smooth. Consequently, 
 \begin{align*}
        p_t(\z^{t+1})&\leq p_t(\z^{t})+\underbrace{\langle\nabla p_t(\z^{t}),\z^{t+1}-\z^{t}\rangle}_{\text{\O1}}+\underbrace{\frac{r_{\x}}{2}\|\z^{t+1}-\z^{t}\|^2}_{\text{\O2}}.
\end{align*} 
We next bound the two terms $\text{\O1}$ and $\text{\O2}$ separately. For term $\text{\O1}$, we use the update rule of $\z^{t+1}$, 
 \begin{align*}
        \,\langle\nabla p_t(\z^{t}),\z^{t+1}-\z^{t}\rangle =&\,\beta^t\langle\nabla p_t(\z^{t}),\x^{t+1}-\z^{t}\rangle\notag\\
        =&\,\beta^t\left\langle\nabla p_t(\z^{t}),\x_{t}^\star(\z^{t})-\z^{t}\right\rangle+\beta^t\left\langle\nabla p_t(\z^{t}),\x^{t+1}-\x_{t}^\star(\z^{t})\right\rangle\notag\\
        \leq&\,-\frac{\beta^t}{r_{\x}}\|\nabla p_t(\z^{t})\|^2+\beta^t\|\nabla p_t(\z^{t})\|\cdot\left\|\x^{t+1}-{\x}_{t}^\star(\z^{t})\right\|,
    \end{align*}
    where the inequality follows the fact that  $ \nabla p_t(\z)=r_{\x}(\z-{\x}_{t}^\star(\z))$ and the Cauchy-Schwarz inequality.

    We now bound term $\text{\O2}$, 
    \begin{align*}
        \frac{r_{\x}}{2}\|\z^{t+1}-\z^{t}\|^2&=\frac{r_{\x}(\beta^t)^2}{2}\|\z^{t}-\x^{t+1}\|^2\notag\\
        &\le r_{\x}(\beta^t)^2\left(\left\|\x^{t+1}-{\x}_{t}^\star(\z^{t})\right\|^2+\left\|{\x}_{t}^\star(\z^{t})-\z^{t}\right\|^2\right)\notag\\
        &\le r_{\x}(\beta^t)^2\left\|\x^{t+1}-{\x}_{t}^\star(\z^{t})\right\|^2+\frac{(\beta^t)^2}{r_{\x}}\|\nabla p_t(\z^{t})\|^2, 
    \end{align*}
where the first equality follows from the $\z$-update rule, the first inequality follows from the elementary bound $(a+b)^2 \le 2(a^2+b^2)$ for any $a,b \in \mathbb{R}$, and the last inequality follows from the identity $ \nabla p_t(\z)=r_{\x}(\z-{\x}_{t}^\star(\z))$.

Putting everything together yields 
     \begin{align*}
        p_{t+1}(\z^{t+1})-p_t(\z^{t}) \le&\,-\frac{\beta^t(1-\beta^t)}{r_{\x}}\|\nabla p_t(\z^{t})\|^2+\beta^t\|\nabla p_t(\z^{t})\|\cdot\left\|\x^{t+1}-\x_{t}^\star(\z^{t})\right\|\\
        &\,+r_{\x}(\beta^t)^2\left\|\x^{t+1}-\x_{t}^\star(\z^{t})\right\|^2+\frac{r_{\y}^t-r_{\y}^{t+1}}{2}D_{\cY}^2.
    \end{align*}
    This completes the proof.
    
\end{proof}
Following  the analysis framework in \cref{sec:framework}, after establishing the basic descent estimate in \cref{lem:p-est}, it remains to control the positive term involving $\|\x^{t+1}-{\x}_{t}^\star(\z^{t})\|$ in \cref{lem:p-est}.
In single-loop algorithms, such error terms are typically controlled indirectly via primal–dual error bounds, e.g., see \cref{lemma:perturb-dualerror} and \cref{lemma:dualerror}.
In contrast, for the double-loop scheme, the quantity $\|\x^{t+1}-{\x}_{t}^\star(\z^{t})\|$ directly corresponds to the inaccuracy in solving the inner SC-SC subproblem. 
In other words, by the definition of $\delta$-accurate solution, we directly have $\|\x^{t+1}-{\x}_{t}^\star(\z^{t})\|\le \sqrt{\delta^t}$.

In the next theorem, we derive the iteration complexity of \emph{Perturbed Smoothed FOAM}. 
\begin{theorem}[Iteration complexity of Perturbed Smoothed FOAM]\label{thm:Smoothed-FOAM-epi} 
Let $\{(\x^{t},\y^{t},\z^{t})\}_{t\ge 0}$ be the sequence generated by \emph{Perturbed Smoothed FOAM}. For any $\epsilon>0$, if the parameters satisfy \cref{con-FOAM} and $\Delta_{p_0}:=p_0(\z^0)-\min_{\x\in\cX}\max_{\y\in\cY}f_0(\x,\y)$ denote the initial gap. Then
\begin{enumerate}[label=(\roman*)] 
\item (Game Stationarity) Let $r_{\y}= \Theta\left(\tfrac{\epsilon}{D_{\cY}}\right)$ and set $r_{\y}^t=r_{\y}$ for all $t \ge 0$, 
then we can obtain an $\epsilon$-GS of problem \cref{eq:prob} after $\cO\left(\frac{\ell^{1.5}D_{\cY}^{0.5}\Delta_{p_0}}{\epsilon^{2.5}}\log\frac{1}{\epsilon}\right)$ gradient steps.
\item (Optimization Stationarity) 
Let $r_{\y}= \Theta\left(\frac{\epsilon^2}{\ell D_{\cY}^2}\right)$ and set $r_{\y}^t=r_{\y}$ for all $t \ge 0$, 
then we can obtain an $\epsilon$-OS of problem \cref{eq:prob} after $\cO\left(\frac{\ell^2 D_{ \cY}\Delta_{p_0}}{\epsilon^{3}}\log\frac{1}{\epsilon}\right)$ gradient steps.
\end{enumerate}
\end{theorem}

\begin{proof}[Proof of \cref{thm:Smoothed-FOAM-epi}.]

By \cref{alg:Smoothed-FOAM}, for each iteration $t \ge 0$, the pair $(\x^{t+1}, \y^{t+1})$ is a $\delta^t$-accurate solution of $F_t(\cdot,\cdot,\z^{t})$. By definition, this implies that
    \begin{equation}
        \left\|\x^{t+1}-{\x}_{t}^\star(\z^{t})\right\|^2+\|\y^{t+1}-{\y}_t(\z^{t})\|^2\leq\delta^t.\label{FOAM}
    \end{equation}
Together with \cref{lem:p-est}, it yields that 
\begin{align}
        &\, p_{t+1}(\z^{t+1})-p_t(\z^{t})\notag\\
        \le&\,-\frac{\beta^t(1-\beta^t)}{r_{\x}}\|\nabla p_t(\z^{t})\|^2+\beta^t\sqrt{\delta^t}\|\nabla p_t(\z^{t})\|+r_{\x}(\beta^t)^2\delta^t+\frac{r_{\y}^t-r_{\y}^{t+1}}{2}D_{\cY}^2\notag\\
        \le &\,-\frac{\beta^t(1-\beta^t)}{2r_{\x}}\|\nabla p_t(\z^{t})\|^2
    +\left(\frac{\beta^t}{2(1-\beta^t)}+(\beta^t)^2\right) r_{\x}\delta^t+\frac{r_{\y}^t-r_{\y}^{t+1}}{2}D_{\cY}^2,
    \label{eq:p-1}
\end{align}
where the second inequality follows from $ab\le\frac{1}{2} (a^2+b^2)$ for all $a,b\in \R$. Under \cref{con-FOAM}, we choose $\beta^t=\Theta(1)$ and $1-\beta^t=\Theta(1)$, and set
$\delta^t=\cO\left(\frac{(r_{\y}^t)^4 D_{\cY}^2}{\ell^4}\right)$.

If we further choose $r_{\y}^t=\cO(\tfrac{1}{\sqrt{T}})$, then the second term in \cref{eq:p-1} is of order 
$\cO(r_{\x}\delta^t)=\cO(\tfrac{1}{T^2})$, 
and hence its cumulative contribution over $T$ iterations is $\cO(\tfrac{1}{T})$.

Using the facts that $r_{\y}^t=r_{\y}^0$ for all $t\ge0$ and $p_0(\z^{t})=\min_{\x\in\cX}\max_{\y\in\cY}F_0(\bx,\by,\z^{t})\ge\min_{\x\in\cX}\max_{\y\in\cY}f_0(\x,\y)$, we sum \cref{eq:p-1} over $t=0,\ldots,T-1$ to obtain
    \begin{equation}\label{eq:p} 
        \min_{t\in \{\lfloor\frac{T}{2}\rfloor,\cdots,T-1\}}\|\nabla p_t(\z^{t})\|=\cO\left(\sqrt{\frac{r_{\x}\Delta_{p_0} }{\beta^t(1-\beta^t)T}}\right)=\cO\left(\sqrt{\frac{\ell\Delta_{p_0}}{T}}\right).  
    \end{equation}

    Let $t^\star \in \arg\min_{t\in\{\lfloor\frac{T}{2}\rfloor,\ldots,T-1\}} \|\nabla p_t(\z^{t})\|$.
Then \cref{eq:p} implies that 
\begin{align}
     \|\x^{t^\star+1}-\z^{t^\star}\|  & \leq   \left\|\x^{t^\star+1}-{\x}_{t^\star}^\star(\z^{t^\star})\right\|+\left\|{\x}_{{t^\star}}^\star(\z^{t^\star})-\z^{t^\star}\right\|\notag\\
    &  \leq \sqrt{\delta^{t^\star}} +\frac{\|\nabla p_{t^\star}(\z^{t^\star}) \|}{r_{\x}} = \cO\left(\sqrt{\frac{\Delta_{p_0}}{\ell T}}\right) \label{norm:x-z},
\end{align}
where the second inequality follows from \eqref{FOAM}, and the last equality is due to \cref{eq:p,con-FOAM}. 

Then, we consider one projected gradient update
starting from $(\x^{t^\star+1}, \y^{t^\star+1},\z^{t^\star})$, defined by
\begin{equation} \label{eq:x-step}
\hat{\x} := \proj_{\cX}\!\left(\x^{t^\star+1} - \alpha \nabla_{\x} F_{t^\star}(\x^{t^\star+1}, \y^{t^\star+1}, \z^{t^\star})\right),
\end{equation}
\begin{equation}\label{eq:y-step}
  \,\hat{\y} := \proj_{\cY}\!\left(\y^{t^\star+1} + c \nabla_{\y} F_{t^\star}(\x^{t^\star+1}, \y^{t^\star+1}, \z^{t^\star})\right),  
\end{equation}
where $\alpha=\Theta(\frac{1}{\ell}),~c=\Theta(\frac{1}{\ell})$ with $\,0<\alpha \leq \frac{1}{r_{\x}+\ell},\,0<c \leq \frac{1}{r_{\y}^{t^\star}+\ell}$. 
Using \cref{lem:closeness}, we have that $(\hat \x, \hat \y)$ is close to
$(\x^{t^\star+1}, \y^{t^\star+1})$, namely, 
\begin{equation}\label{eq:eb_proxy}
\|\x^{t^\star+1}-\hat{\x}\|=\cO\left(\sqrt{\delta^{t^\star}}\right),\quad \text{and} \quad  \|\y^{t^\star+1}-\hat{\y}\|=\cO\left(\frac{\ell}{r_{\y}^{t^\star}}\sqrt{\delta^{t^\star}}\right).    
\end{equation}
We therefore analyze GS and OS stationarity at $(\hat \x,\hat \y)$.

(i) For the GS case, we choose $r_{\y}^t=\Theta\left(\sqrt{\frac{\ell\Delta_{p_0}}{D_{\cY}^2T}}\right)$.
For the primal stationarity measure, we have 
\begin{align}
    &\,\dist(\bz,\nabla_{\x} f(\hat{\x},\hat{\y})+\partial\iota_\cX(\hat{\x}))\notag\\
    =&\,\dist\left(\bz,\nabla_{\x} F_{t^\star}(\hat{\x},\hat{\y},\z^{t^\star})-r_{\x}(\hat{\x}-\z^{t^\star})+\partial\iota_\cX(\hat{\x})\right)\notag\\
    \leq&\,\dist\left(\bz,\nabla_{\x} F_{t^\star}(\hat{\x},\hat{\y},\z^{t^\star})+\partial\iota_\cX(\hat{\x})\right)+r_{\x}\|\hat{\x}-\z^{t^\star}\|\notag\\
    \leq&\,\dist\left(\bz,\nabla_{\x} F_{t^\star}(\x^{t^\star+1},\y^{t^\star+1},\z^{t^\star})+\partial\iota_\cX(\hat{\x})\right)+r_{\x}\|\hat{\x}-\z^{t^\star}\|\notag\\
    &\,+\left\|\nabla_{\x} F_{t^\star}(\hat{\x},\hat{\y},\z^{t^\star})-\nabla_{\x} F_{t^\star}(\x^{t^\star+1},\y^{t^\star+1},\z^{t^\star})\right\|\notag\\
    \leq&\,\!\frac{1}{\alpha}\|\x^{t^\star+1}\!-\!\hat{\x}\|\! +\!r_{\x}\|\x^{t^\star+1}\!-\!\hat{\x}\|\!+ \!r_{\x}\|\x^{t^\star+1}\!-\!\z^{t^\star}\|\!+\!(r_{\x}+\ell)(\|\x^{t^\star+1}-\hat{\x}\|\!+\!\|\y^{t^\star+1}-\hat{\y}\|)\notag\\
    =&\,\cO\left(\max\left(\sqrt{\frac{\ell\Delta_{p_0}}{T}},\frac{\ell^2}{r_{\y}^{t^\star}}\sqrt{\delta^{t^\star}}\right)\right)=\cO\left(\sqrt{\frac{\ell\Delta_{p_0}}{T}}\right),\label{gradient:x}
\end{align}
where the third inequality is due to \eqref{eq:x-step} and $(r_{\x}+\ell)$-Lipschitz continuity of $\nabla_{\x}F_{t^\star}(\cdot,\cdot,\z)$, the second equality follows from \cref{eq:eb_proxy,norm:x-z}, and the last equality is owing to the choice of $r_{\y}^{t^\star}$ and $\delta^{t^\star}=\cO\left(\frac{(r_{\y}^{t^\star})^4D^2_{\cY}}{\ell^4}\right)$.

We now turn to the dual part. By a similar argument, we have 
\begin{align}
    &\,\dist(\bz,-\nabla_{\y} f(\hat{\x},\hat{\y})+\partial\iota_\cY(\hat{\y}))\notag\\
    =&\,\dist\left(\bz,-\nabla_{\y} F_{t^\star}(\hat{\x},\hat{\y},\z^{t^\star})-r_{\y}^{t^\star}\hat{\y}+\partial\iota_\cY(\hat{\y})\right)\notag\\
    &+\left\|\nabla_{\y} F_{t^\star}(\hat{\x},\hat{\y},\z^{t^\star})-\nabla_{\y} F_{t^\star}(\x^{t^\star+1},\y^{t^\star+1},\z^{t^\star})\right\|\notag\\
    \leq&\,\frac{1}{c}\|\y^{t^\star+1}-\hat{\y}\|+r_{\y}^{t^\star}\|\hat{\y}\|+(r_{\y}^{t^\star}+\ell)(\|\x^{t^\star+1}-\hat{\x}\|+\|\y^{t^\star+1}-\hat{\y}\|)\notag\\
    =&\,\cO\left(\max\left(\frac{\ell^2}{r_{\y}^{t^\star}}\sqrt{\delta^{t^\star}},r_{\y}^{t^\star}D_{\cY}\right)\right) =\cO\left(\sqrt{\frac{\ell\Delta_{p_0}}{T}}\right).\label{gradient:y}
\end{align}


Thus, to reach an $\epsilon$-GS solution, it suffices to take
$T=\cO\left(\frac{\ell\Delta_{p_0}}{\epsilon^2}\right)$
outer iterations.
Substituting this choice of $T$ into the parameter selection in
\cref{con-FOAM} yields the equivalent scaling
$r_{\y}^{t}=\Theta\left(\frac{\epsilon}{D_{\cY}}\right)$.
Moreover, the number of inner iterations required is
$
\cO\left(\sqrt{\frac{\ell^2}{(r_{\x}-\ell)r_{\y}^{t^\star}}}\log\frac{1}{\delta^{t^\star}}\right)=\cO\left(\sqrt{\frac{\ell D_{\cY}}{\epsilon}}\log\frac{1}{\epsilon}\right).
$ 
Consequently, the overall oracle complexity is
\[T=\cO\left(\frac{\ell\Delta_{p_0}}{\epsilon^2}\cdot\sqrt{\frac{\ell D_{\cY}}{\epsilon}}\log\frac{1}{\epsilon}\right)=\cO\left(\frac{\ell^{1.5}D_{\cY}^{0.5}\Delta_{p_0}}{\epsilon^{2.5}}\log\frac{1}{\epsilon}\right).\]


\noindent(ii) For the OS case, we choose $r_{\y}^t=\cO\left(\frac{\Delta_{p_0}}{D_{\cY}^2T}\right)$. Due to \cref{lemma2}, \cref{gradient:x,gradient:y}, we have 
\begin{align*}
&\, \|\prox_{\frac{1}{2\ell} \Phi}(\hat{\x})-\hat{\x}\|^2 \notag\\
\leq &\,   \frac{2D_{\cY}}{\ell}  {\rm dist}(\bz,-\nabla_{\y} f(\hat{\x}, \hat{\y})+\partial\iota_{\mathcal{Y}}(\hat{\y})) +\frac{1}{\ell^2}  {\rm dist}^2(\bz,\nabla_{\x} f(\hat{\x}, \hat{\y})+\partial\iota_{\mathcal{X}}(\hat{\x})) \notag\\
=&\, \cO\left(\frac{D_{\cY}}{\ell}\cdot r_{\y}^{t^\star}D_{\cY}+\frac{\ell D_{\cY}}{r_{\y}^{t^\star}}\sqrt{\delta^{t^\star}}+\frac{1}{\ell^2}\frac{\ell\Delta_{p_0}}{T}\right) 
=\cO\left(\frac{\Delta_{p_0}}{\ell T}\right).
\end{align*}
Therefore, to achieve an $\epsilon$-OS solution, it suffices to take
$T=\cO\left(\frac{\ell\Delta_{p_0}}{\epsilon^2}\right)$
outer iterations.
Substituting this choice of $T$ into the parameter selection in
\cref{con-FOAM} yields the equivalent scaling
$r_{\y}^t=\Theta\left(\frac{\epsilon^2}{\ell D_{\cY}^2}\right)$.
Additionally, the number of inner iterations required is
$
\cO\left(\sqrt{\frac{\ell^2}{(r_{\x}-\ell) r_{\y}^{t^\star}}}\log\frac{1}{\delta^{t^\star}}\right)=\cO\left(\frac{\ell D_{\cY}}{\epsilon}\log\frac{1}{\epsilon}\right). 
$ 
Consequently, the overall oracle complexity is
\[T=\cO\left(\frac{\ell\Delta_{p_0}}{\epsilon^2}\cdot\frac{\ell D_{\cY}}{\epsilon}\log\frac{1}{\epsilon}\right)=\cO\left(\frac{\ell^2 D_{\cY}\Delta_{p_0}}{\epsilon^3}\log\frac{1}{\epsilon}\right).\]
This completes the proof.

\end{proof}

\section{Closing Remarks}
\label{sec:close}
This paper provides a unified framework to analyze smoothing and perturbation in NC-C minimax optimization. Rather than viewing these techniques as interchangeable acceleration tools, our framework clarifies their distinct algorithmic roles and explains when and how they interact to yield sharper convergence guarantees. This separation leads to a more transparent analysis of existing first-order methods and 
provides deeper understanding to these acceleration tools, which enables us to unify
several seemingly disparate convergence results in the literature.
An interesting direction for future research is to investigate whether the structural insights developed here—particularly the decoupling of smoothing and perturbation and the associated primal–dual viewpoint—can be exploited to obtain stronger or more general convergence guarantees. Recent works such as \citep{zheng2024universal,zheng2025doubly} and \citep{lu2025first} establish convergence results under broad global or local Kurdyka–Łojasiewicz–type assumptions. It would be of interest to explore whether our framework can be used to streamline these analyses or to clarify the algorithmic roles of smoothing and perturbation under such general conditions.

\bibliography{ref}
\bibliographystyle{apalike}
\begin{appendices}
\section{Useful Lemmas}\label{app:aux}
To keep the presentation streamlined, we defer the proof details of technical lemmas to the appendix.
In particular, the appendix collects proofs of the lemmas stated in the paper, as well as auxiliary lemmas invoked in the main body.

\subsection{Proof of \cref{lemma2}}\label{sec:GS-OS}

\begin{proof}[Proof of \cref{lemma2}.]
Our proof builds upon the argument in \citep[Proposition~2.1]{yang2022faster}. 
In particular, we extend their analysis from the unconstrained setting to the constrained case.
We first compute 
\begin{align}
 &\, \frac{\ell}{2}\left\|\prox_{\frac{1}{2\ell} \Phi}(\x)-\x\right\|^2 \notag\\  \leq & \, \Phi(\x) - \Phi\left(\prox_{\frac{1}{2\ell} \Phi}(\x)\right)-\ell\left\|\prox_{\frac{1}{2\ell} \Phi}(\x)-\x\right\|^2 \notag \\
 = & \, \underbrace{\Phi(\x) -f(\x,\y)}_\text{\O1}+ f\left(\prox_{\frac{1}{2\ell} \Phi}(\x),\y\right)- \Phi\left(\prox_{\frac{1}{2\ell} \Phi}(\x)\right) \notag\\
 &\, +\underbrace{f(\x,\y) - f\left(\prox_{\frac{1}{2\ell} \Phi}(\x),\y\right)-\ell\left\|\prox_{\frac{1}{2\ell} \Phi}(\x)-\x\right\|^2 }_\text{\O2},\label{eq:equ1}
\end{align}
where the first inequality follows from the $\ell$-strong convexity of the function $\Phi(\cdot)+\ell\|\cdot-\x\|^2$. Moreover, because $f(\prox_{\frac{1}{2\ell} \Phi}(\x),\y)- \Phi(\prox_{\frac{1}{2\ell} \Phi}(\x))\leq 0$, we just focus on the two parts  \text{\O1} and \text{\O2}. We begin with term \text{\O1}. Note that
\begin{equation}
\label{eq:equ2}
   \Phi(\x) -f(\x,\y) = \max_{y'\in\mathcal{Y}} f(\x,\y') -f(\x,\y) 
\leq {\rm dist}(\bz,-\nabla_{\y} f(\x, \y)+\partial\iota_{\mathcal{Y}}(\y)) \cdot D_{\cY}, 
\end{equation}
where the inequality follows from the Cauchy–Schwarz inequality and the concavity of the function $f(\x,\cdot)+\iota_{\mathcal{Y}}(\cdot)$ for any $\x \in \mathcal{X}$. 
Then, we proceed to consider term \text{\O2} in \cref{eq:equ1},  
\begin{align}
   & \, f(\x,\y) - f\left(\prox_{\frac{1}{2\ell} \Phi}(\x),\y\right)-\ell\left\|\prox_{\frac{1}{2\ell} \Phi}(\x)-\x\right\|^2\notag\\
   \leq  &\, \left\|\x-\prox_{\frac{1}{2\ell} \Phi}(\x) \right\| \cdot {\rm dist}(\bz,\nabla_{\x} f(\x, \y)+\partial\iota_{\mathcal{X}}(\x)) -\frac{\ell}{2}\left\|\prox_{\frac{1}{2\ell} \Phi}(\x)-\x\right\|^2 \notag\\
   \leq & \,  \frac{1}{2\ell} {\rm dist}^2(\bz,\nabla_{\x} f(\x, \y)+\partial\iota_{\mathcal{X}}(\x))\label{eq:equ3},  
\end{align}
where the first inequality follows from the $\ell$-strong convexity of $f(\cdot,\y)+\ell\|\cdot-\x\|^2+\iota_{\mathcal{X}}(\cdot)$, and the second one is owing to $ab\leq \frac{1}{2}(a^2+b^2)$ for any $a,b\in\R$. 

Combining \cref{eq:equ1}, \cref{eq:equ2} and \cref{eq:equ3} yields
\begin{align}
&\, \left\|\prox_{\frac{1}{2\ell} \Phi}(\x)-\x\right\|^2 \notag\\
\leq &\,   \frac{2 D_{\cY}}{\ell}  {\rm dist}(\bz,-\nabla_{\y} f(\x, \y)+\partial\iota_{\mathcal{Y}}(\y)) +\frac{1}{\ell^2}  {\rm dist}^2(\bz,\nabla_{\x} f(\x, \y)+\partial\iota_{\mathcal{X}}(\x)). \notag
\end{align}
This completes the proof.
\end{proof}

\subsection{Proof of \cref{lemma:dualerror}}\label{sec:proof-dualerror} 
\begin{proof}[Proof of \cref{lemma:dualerror}.]
    The result follows directly from \citep[Corollary~5.1]{li2025nonsmooth}. 
    Since any $r_{\y}^t$-strongly concave function is also a K{\L}-function with 
    exponent $\theta=\tfrac{1}{2}$ and parameter $\mu=\sqrt{2r_{\y}^t}$ \citep[Assumption~3.2]{li2025nonsmooth}, we may
    substitute these values into the expression for $\omega^t=\omega_2$ in  
    \citep[Corollary~5.1]{li2025nonsmooth}.  
    This yields
    \[
    \omega^t
    = \frac{1}{\sqrt{r_{\y}^t}}
      \cdot
      \frac{1 + \alpha^t(\ell+r_{\y}^t)(1+\sigma_2^t)}
           {\alpha^t\sqrt{r_{\x}-(\ell+r_{\y}^t)}}
    = \frac{1}{\sqrt{r_{\y}^t}}
      \cdot
      \frac{(r_{\x}-\ell-r_{\y}^t)+\alpha^t(\ell+r_{\y}^t)(3r_{\x}-2\ell-2r_{\y}^t)}
           {\alpha^t (r_{\x}-\ell-r_{\y}^t)^{3/2}},
    \]
    where $\sigma_2^t = \frac{2r_{\x}-\ell-r_{\y}^t}{r_{\x}-\ell-r_{\y}^t}$ is given in 
    \citep[Lemma~A.1]{li2025nonsmooth}. This completes the proof.
\end{proof}
\subsection{Proof of \cref{lemma:lamb-order-1}}\label{sec:lem:lam}
\begin{proof}[Proof of \cref{lemma:lamb-order-1}.]
Since we choose $r_y=\Theta\left(\frac{\epsilon}{D_{\mathcal{Y}}}\right),\, ~\mbox{and}~c=\Theta\left(\frac{\epsilon^2}{\ell^3D_{\cY}^2}\right)$, it follows
that $A_1<0$ and
\[ |A_1| =\Theta\left( \frac{\epsilon}{\ell D_{\cY}}\right),\,B_1=\Theta\left( \frac{\epsilon^3}{\ell^3 D_{\cY}^3}\right),\,\mbox{and}~\sqrt{A_1^2-4B_1} =\Theta\left( \frac{\epsilon}{\ell D_{\cY}}\right).\]
First, we establish that $\lambda_1$ is negative. Indeed,
\begin{equation*}
    \lambda_1=\frac{1}{2}\left(A_1+\sqrt{A_1^2-4B_1}\right)=-\frac{2B_1}{\sqrt{A_1^2-4B_1}-A_1}<0.
\end{equation*}
Then, direct computation yields $|\lambda_1| =\Theta\left( \frac{\epsilon^2}{\ell^2 D_{\cY}^2}\right)$. 
This completes the proof. 
\end{proof}

\subsection{Proof of \cref{gs-alter-tight}}\label{pf-of-eigen}


\begin{proof}[Proof of \cref{gs-alter-tight}.]
We prove by induction that for all $t\in\mathbb N$,
\begin{equation}\label{eq:ind-claim}
x^t\in \left[0,\frac{r_yD_{\cY}}{b}\right],\qquad
y^t\in[0,D_{\mathcal Y}],\quad \mbox{and} \quad
(x^t,y^t)^\top=(1+\lambda_1)^t (x^0,y^0)^\top.
\end{equation}
\textbf{Base case ($t=0$):} By the choice of the initial point $x^0=\frac{2\epsilon}{\ell}$ and noting that
$
\tfrac{r_y D_{\mathcal Y}}{b}
= \Theta\left(\sqrt{\frac{\epsilon D_{\mathcal Y}}{\ell}}\right),
$
we have
$
0 < x^0 < \frac{r_y D_{\mathcal Y}}{b}.
$
Similarly, since $b=\Theta\left(\sqrt{\frac{\epsilon \ell}{D_{\mathcal Y}}}\right)$, it follows that
\begin{equation}
0 < y^0 = \frac{2b(1+\ell c)\alpha}{\sqrt{A_1^2-4B_1}+2\ell c-A_1} \cdot \frac{2\epsilon}{\ell}
= \Theta\left(\sqrt{\frac{\epsilon D_{\mathcal Y}}{\ell}}\right), \label{eq:y_0-GS-pgda}
\end{equation}
which implies $y^0 \in [0, D_{\mathcal Y}]$. Therefore, \cref{eq:ind-claim} holds for $t=0$.

\textbf{Inductive step:} Assume that \cref{eq:ind-claim} holds for some $t\ge 0$. In particular, $x^t\in\bigl[0,\frac{r_yD_{\cY}}{b}\bigr]$, so by the definition of $f$ in \cref{eg:gs-alter-tight} (the middle branch), the \emph{Perturbed GDA} update admits the linear form
\[
(x^{t+1},y^{t+1})^\top=(\bm{I}+\bm{M_1})(x^t,y^t)^\top.
\]
By construction, the initialization $(x^0,y^0)^\top$ is an eigenvector of $\bm{I}+\bm{M_1}$
associated with the eigenvalue $1+\lambda_1$, which implies that $(\bm{I}+\bm{M_1})(x^{0},y^{0})^\top=(1+\lambda_1)(x^{0},y^{0})^\top$. Hence 
\begin{equation}
\label{eq:multi}
(x^{t+1},y^{t+1})^\top\!=\!(\bm{I}+\bm{M_1})(x^{t},y^{t})^\top\!=\!(\bm{I}+\bm{M_1})\left((1+\lambda_1)^t(x^{0},y^{0})^\top\right)\!=\!(1+\lambda_1)^{t+1}(x^{0},y^{0})^\top.
\end{equation}
where the second equality uses the induction hypothesis $(x^t,y^t)^\top=(1+\lambda_1)^t (x^0,y^0)^\top$. Since $0<1+\lambda_1<1$ by \cref{lemma:lamb-order-1}, we have $(1+\lambda_1)^{t+1}\in(0,1)$, and thus
\begin{equation}
    \label{eq:bound}
0\le x^{t+1}\le x^0\le \frac{r_y D_{\cY}}{b},
\qquad
0\le y^{t+1}\le y^0\le D_{\cY}.
\end{equation}
Combining \cref{eq:multi,eq:bound}, we conclude that \cref{eq:ind-claim} holds for $t+1$, which completes the induction.

Consequently, \eqref{eq:ind-claim} holds for all $t\in\mathbb N$. This completes the proof.
\end{proof}

\subsection{Initial Gap Bound for \emph{Perturbed GDA} on \cref{eg:gs-alter-tight}}\label{sec:lem:Delta}
\begin{lemma}\label{lemma:Delta_1}
Suppose that  \emph{Perturbed GDA} is applied to \cref{eg:gs-alter-tight} and initialized as in \cref{gs-alter-tight}. Then the initial gap satisfies
\[
\Delta_{\Psi_1}=\cO\left(\frac{\epsilon^2}{\ell}\right). 
\]
\end{lemma}
\begin{proof}[Proof of \cref{lemma:Delta_1}.]
Recall in \cref{def:initial_gap}, we have that
\[
\begin{aligned}
&\Delta_{\Psi_1}
= {\Psi_1^0}(x^0,y^0)-
\min_{x\in\R}\max_{y\in[0,D_{\cY}]}f_0(x,y),\\
&{\Psi_1^0}(x,y)=2\max_{y'\in[0,D_{\cY}]}f_0(x,y')-f_0(x,y).
\end{aligned}
\]
By \cref{gs-alter-tight}, we have $x^0\in\bigl[0,\frac{r_yD_{\cY}}{b}\bigr]$ and $y^0\in[0,D_{\cY}]$.

Since $x^0\in\bigl[0,\frac{r_yD_{\cY}}{b}\bigr]$, by the definition of \cref{eg:gs-alter-tight} and $f_0(\cdot)$, we have
\[
f_0(x^0,y)= -\frac{1}{2}\ell (x^0)^2 + b x^0 y-\frac{r_y}{2}y^2,
\qquad \forall~ y\in[0,D_{\cY}].
\]
Direct computation yields that
\[
\max_{y\in[0,D_{\cY}]}f_0(x^0,y)
= -\frac{1}{2}\ell (x_0)^2+\frac{b^2}{2r_y}(x_0)^2
=\ell (x_0)^2,
\]
where the last equality uses $b^2=3\ell r_y$ in \cref{instance-gs}. Therefore,
\begin{align*}
{\Psi_1^0}(x^0,y^0)
&=2\ell (x^0)^2-\Bigl(-\frac{1}{2}\ell (x^0)^2+b x^0y^0-\frac{r_y}{2}(y^0)^2\Bigr)\\
&=\frac{5}{2}\ell (x^0)^2-bx^0y^0+\frac{r_y}{2}(y^0)^2
\le \frac{5}{2}\ell (x^0)^2+\frac{r_y}{2}(y^0)^2,
\end{align*}
where the inequality uses $x^0\ge 0$ and $y^0\ge 0$.
Using $x^0=\frac{2\epsilon}{\ell}$, $r_y=\Theta\left(\frac{\epsilon}{D_{\cY}}\right)$, and $y^0=\Theta\left(\sqrt{\frac{\epsilon D_{\cY}}{\ell}}\right)$ from \cref{eq:y_0-GS-pgda}, we obtain
\[
{\Psi_1^0}(x^0,y^0)=\cO\left(\ell (x^0)^2+r_y (y^0)^2\right)
=\cO\left(\frac{\epsilon^2}{\ell}\right).
\]

To proceed, one can verify that
\[
\Phi_0(x)=\max_{y\in[0,D_{\cY}]}f_0(x,y)=
\begin{cases}
0, & x<0,\\
\ell x^2, & 0\le x\le \frac{r_yD_{\cY}}{b},\\
\frac{r_yD_{\cY}^2}{3}, & x> \frac{r_yD_{\cY}}{b},
\end{cases}
\]
and hence
\begin{equation}\label{eq:minmax}
    \min_{x\in\R}\max_{y\in[0,D_{\cY}]}f_0(x,y)=0.
\end{equation}
Putting everything together, we obtained that $\Delta_{\Psi_1}=\cO\left(\frac{\epsilon^2}{\ell}\right)$. This completes the proof.
\end{proof}

\subsection{Proof of \cref{lemma:lamb-order-2}}\label{sec:lam_2}
\begin{proof}[Proof of \cref{lemma:lamb-order-2}.]
(i) First, we show that $\lambda_2$ is real.  To this end, we verify that $4(B_2^2+3C_2)^3-A_2^2\geq 0$. A direct expansion yields
\begin{align*}
& 4(B_2^2+3C_2)^3-A_2^2\\
=&27 c^4 (r_x - \ell)^2 \underbrace{\left(
\ell^2 \beta^2 - 
10 r_x \ell \alpha \beta r_y + 
4 \ell^2 \alpha \beta r_y + 
r_x^2 \alpha^2 r_y^2 + 
4 r_x \ell \alpha^2 r_y^2 + 
4 \ell^2 \alpha^2 r_y^2
\right)}_{=:E_2}+\cO(r_y^3).
\end{align*}
Since $c^4(r_x-\ell)^2>0$, it suffices to show that $E_2 > 0$. A straightforward calculation yields the lower bound
\begin{align}
E_2 &=(r_x \alpha r_y - 5\ell \beta)^2 - 25 \ell^2 \beta^2+4 \ell^2 \alpha \beta r_y + 
r_x^2 \alpha^2 r_y^2 + 
4 r_x \ell \alpha^2 r_y^2 + 
4 \ell^2 \alpha^2 r_y^2 \notag\\
&\geq (r_x \alpha r_y - 5\ell \beta)^2 - 25 \ell^2 \beta^2+ 4\ell^2 \alpha^2 r_y^2.
\label{eq:E-bound-2} 
\end{align}

Given our choice of $\beta$ by \cref{con-psgda}, we see that 
\begin{align*}
    \beta&\leq\frac{1}{384r_x\alpha}\cdot\frac{r_y\alpha^2(r_x-\ell-r_y)^3}{((r_{x}-\ell-r_y)+\alpha \ell(3r_{x}-2\ell-2r_y))^2}\\
    &\le\frac{1}{384r_x\alpha}\cdot\frac{r_y\alpha^2(r_x-\ell-r_y)^3}{(r_{x}-\ell-r_y)^2}
    =\frac{\alpha r_y}{384}\cdot\frac{r_x-\ell-r_y}{r_x}
    \leq\frac{\alpha r_y}{384},
\end{align*}
which implies
$
25\ell^2\beta^2\le \frac{25}{384^2}\ell^2\alpha^2 r_y^2.
$
Plugging this into \eqref{eq:E-bound-2} yields
\[
E_2\ge (r_x \alpha r_y - 5\ell \beta)^2+\left(4-\frac{25}{384^2}\right)\ell^2\alpha^2 r_y^2>0,
\]
and hence $4(B_2^2+3C_2)^3-A_2^2\geq0$.

Next, we examine the closed-form expression \cref{eq:lamb-2}, which involves the complex quantity
\[
D_2:=A_2+\sqrt{4(B_2^2+3C_2)^3-A_2^2}\,i.
\]
We show that the two cube-root terms in \cref{eq:lamb-2} are complex conjugates, which implies that $\lambda_2\in\R$. 
To this end, we rewrite \cref{eq:lamb-2} as
$
\lambda_2=\frac{1}{3}B_2+\Gamma_1+\Gamma_2,
$ 
where
\[
\Gamma_1\coloneqq \frac{\left( 1 + i \sqrt{3} \right) \left( B_2^2+3C_2\right)}{ 3 \cdot 2^{\frac{2}{3}} D_2^\frac{1}{3}}
\quad\text{and}\quad
\Gamma_2\coloneqq\frac{1}{6\cdot2^\frac{1}{3}} (1 - i \sqrt{3})D_2^\frac{1}{3}.
\]
To verify the conjugacy, we compute their product and magnitudes. First,
$
\Gamma_1\Gamma_2=\frac{B_2^2+3C_2}{9}>0.
$
Moreover, as $4(B_2^2+3C_2)^3-A_2^2\ge 0$, we have
\[
|D_2|=\sqrt{A_2^2+\left(4(B_2^2+3C_2)^3-A_2^2\right)}=2\left(\sqrt{B_2^2+3C_2}\right)^3.
\]
Therefore,
\[
|\Gamma_1|=\frac{2(B_2^2+3C_2)}{6\sqrt{B_2^2+3C_2}}=\frac{\sqrt{B_2^2+3C_2}}{3},~\mbox{and}~|\Gamma_2|=\frac{2\cdot2^\frac{1}{3}}{6\cdot2^{\frac{1}{3}}}\sqrt{B_2^2+3C_2}=\frac{\sqrt{B_2^2+3C_2}}{3}=|\Gamma_1|.
\]
Therefore, $\Gamma_1$ and $\Gamma_2$ are complex conjugates, and hence $\lambda_2\in\R$.



\vspace{1em}
(ii) Next, we establish that $\lambda_2<0$ and characterize its magnitude. Since $\Gamma_1$ and $\Gamma_2$ are complex conjugates, we can rewrite the eigenvalue as $\lambda_2=\frac{1}{3}B_2+2\textrm{Re}(\Gamma_2)$. To characterize the magnitude of $\textrm{Re}(\Gamma_2)$, we first analyze the order of $D_2$.
Recall $D_2=A_2+\sqrt{4(B_2^2+3C_2)^3-A_2^2}\,i$. Using the expansion of the discriminant term and the expression of $A_2$, we have
\[
\sqrt{4(B_2^2+3C_2)^3-A_2^2}=3\sqrt{3}\,c^2(r_x-\ell)\sqrt{E_2}+\cO(r_y^2).
\]
Therefore, we have 
\begin{align}
    D_2=&\,A_2+\sqrt{4(B^2+3C_2)^3-A_2^2}\,i \notag\\
    =&\,
    \underbrace{-3 c^2 (r_x \!-\! \ell) \Bigl(\!-2 r_x \beta \!+\! 2 c r_x^2 \beta \!-\! \ell \beta \!-\! 2 c r_x \ell \beta \!+\! r_x \alpha r_y \!+\! 8 \ell \alpha r_y \!-\! 6 c r_x \ell \alpha r_y \!+\! 6 c \ell^2 \alpha r_y \!\Bigr)}_{=:D_{2,1}} \notag\\
    &\,+i\underbrace{3\sqrt{3}\,c^2(r_x-\ell)\sqrt{E_2}}_{=:D_{2,2}}+2(c r_x-\ell c)^3+\cO(r_y^2). \label{eq:D-decom}
\end{align}
Now, with \cref{eq:D-decom}, we compute the order of $2\textrm{Re}(\Gamma_2)$. The direct computation yields 
\begin{align}
    2\textrm{Re}\left(\Gamma_2\right)&\overset{\cref{eq:D-decom}}{=}\textrm{Re}\left(\frac{(1-\sqrt{3}i)\left(2(c r_x-\ell c)^3+D_{2,1}+D_{2,2}i+\cO(r_y^2)\right)^\frac{1}{3}}{3*2^\frac{1}{3}}\right)\notag\\
    &=\frac{1}{3}(c r_x-\ell c)\cdot \textrm{Re}\left((1-\sqrt{3}i)\left(1+\frac{D_{2,1}+D_{2,2}i}{2(c r_x-\ell c)^3}+\cO(r_y^2)\right)^\frac{1}{3}\right)\notag\\
    &=\frac{1}{3}(c r_x-\ell c)\cdot \textrm{Re}\left((1-\sqrt{3}i)\left(1+\frac{D_{2,1}+D_{2,2}i}{6(c r_x-\ell c)^3}+\cO(r_y^2)\right)\right)\notag\\
    &=\frac{1}{3}(c r_x-\ell c)\left(1+\frac{D_{2,1}+\sqrt{3}D_{2,2}}{6(c r_x-\ell c)^3}+\cO(r_y^2)\right)\notag\\
    &=\frac{1}{3}(c r_x-\ell c)+\frac{D_{2,1}+\sqrt{3}D_{2,2}}{18(c r_x-\ell c)^2}+\cO(r_y^2),\label{eq:ReGamma}
\end{align}
where the third equality uses the Maclaurin expansion for $\left(1+\cdot\right)^\frac{1}{3}$.
Thus, $\lambda_2$ can be represented in the form of
\begin{align}
    \lambda_2&=\frac{1}{3}B_2+\frac{1}{3}(c r_x-\ell c)+\frac{D_{2,1}+\sqrt{3}D_{2,2}}{18(c r_x-\ell c)^2}+\cO(r_y^2)\notag\\
    &=\frac{-3\ell c\alpha r_y-\beta+c r_x\beta-\alpha r_y}{3}+\frac{1}{6(r_x-\ell)}\left(3\sqrt{E_2}-(-2 r_x \beta + 2 c r_x^2 \beta - \ell \beta \right. \notag\\
    &\left. \quad- 2 c r_x \ell \beta + r_x \alpha r_y + 8 \ell \alpha r_y - 6 c r_x \ell \alpha r_y + 6 c \ell^2 \alpha r_y)\right)+\cO(r_y^2)\notag\\
    &=\frac{\ell \beta - r_x \alpha r_y - 2 \ell \alpha r_y}{2(r_x-\ell)}+\frac{1}{2(r_x-\ell)}\sqrt{E_2}+\cO(r_y^2)\notag\\
    &=\frac{1}{2(r_x-\ell)}\left(\frac{E_2-(\ell \beta - r_x \alpha r_y - 2 \ell \alpha r_y)^2}{\sqrt{E_2}+ r_x \alpha r_y + 2 \ell \alpha r_y-\ell \beta }\right)+\cO(r_y^2)\notag\\
    &=-\frac{8 (r_x - \ell) \ell \alpha \beta r_y}{2(r_x-\ell)}\cdot\frac{1}{\sqrt{E_2}+ r_x \alpha r_y + 2 \ell \alpha r_y-\ell \beta}+\cO(r_y^2)\notag\\
    &=-\frac{4\ell\alpha\beta r_y}{\sqrt{E_2}+ r_x \alpha r_y + 2 \ell \alpha r_y-\ell \beta}+\cO(r_y^2),\label{result}
\end{align}
 where the first equality follows from \cref{eq:ReGamma}, the second equality follows from \cref{eq:D-decom}, the fourth equality follows from the definition of $E_2$.
In fact, \cref{result} implies that $\lambda_2<0$. Moreover, note that $\sqrt{E_2}=\Theta(r_y)$ and $r_x \alpha r_y + 2 \ell \alpha r_y-\ell \beta=\Theta(r_y)$. Combining \cref{con-psgda} with \eqref{result} yields $|\lambda_2|=\Theta\left(\frac{\epsilon}{\ell D_{\cY}}\right)$. This completes the proof.
\end{proof}

\subsection{Proof of \cref{gs-psgda-tight}}\label{sec:ite}
\begin{proof}[Proof of \cref{gs-psgda-tight}.]
We prove by induction that for all $t\in\mathbb N$,
\begin{equation}\label{eq:ind-claim-psgda}
x^t\in \left[0,\frac{r_yD_{\cY}}{b}\right],\qquad
y^t\in[0,D_{\mathcal Y}],\quad \mbox{and} \quad
(x^t,y^t,z^t)^\top=(1+\lambda_2)^t (x^0,y^0,z^0)^\top.
\end{equation}
\textbf{Base case ($t=0$):} By the initialization in \cref{eq:initial:psgda} and noting that
$
\tfrac{r_y D_{\mathcal Y}}{b}
= \Theta\left(\sqrt{\frac{\epsilon D_{\mathcal Y}}{\ell}}\right),
$
it follows that
$
0 < x^0 < \frac{r_y D_{\mathcal Y}}{b}.$ 
Similarly, using \cref{eq:initial:psgda} and $b=\Theta\left(\sqrt{\frac{\epsilon \ell}{D_{\mathcal Y}}}\right)$, one can verify that
\begin{equation*}
    0<y^0=\frac{c \ell \beta + c \ell \lambda_2 - c r_x 
    \lambda_2 - \beta \lambda_2 + bc r_x \beta \lambda_2 - 
    \lambda_2^2}{c(\beta+\lambda_2)}\cdot\frac{2(\beta
    +\lambda_2)\epsilon}{r_x \lambda_2}=\Theta\left(\sqrt{\frac{\epsilon D_{\cY}}{\ell}}\right),
\end{equation*}
which implies that $y^0\in[0,D_{\mathcal Y}]$. Therefore, \cref{eq:ind-claim-psgda} holds for $t=0$.

\textbf{Inductive step:} Assume that \cref{eq:ind-claim-psgda} holds for some $t\ge 0$. In particular,
$x^t\in[0,\frac{r_yD_{\cY}}{b}]$. By the definition of $f$ in \cref{eg:gs-alter-tight}, when $x^t$ lies in the middle branch of $f(\cdot,y)$, the \emph{Perturbed Smoothed GDA}  update admits the linear form
\[
(x^{t+1},y^{t+1},z^{t+1})^\top=(\bm{I}+\bm{M_2})(x^t,y^t,z^t)^\top.
\]
By construction, the initialization $(x^0,y^0,z^0)^\top$ is an eigenvector of $\bm{I}+\bm{M_2}$
associated with the eigenvalue $1+\lambda_2$, which implies that $(\bm{I}+\bm{M_2})(x^{0},y^{0},z^0)^\top=(1+\lambda_2)(x^{0},y^{0},z^0)^\top$. Hence 
\begin{align}
(x^{t+1},y^{t+1},z^{t+1})^\top=(\bm{I}+\bm{M_2})(x^{t},y^{t},z^{t})^\top&=(\bm{I}+\bm{M_2})\left((1+\lambda_2)^t(x^{0},y^{0},z^{0})^\top\right)\notag\\
&=(1+\lambda_2)^{t+1}(x^{0},y^{0},z^{0})^\top, 
\label{eq:multi-2}
\end{align}
where the second equality follows from the induction hypothesis $(x^t,y^t,z^t)^\top=(1+\lambda_2)^t (x^0,y^0,z^0)^\top$. Since $0<1+\lambda_2<1$ by \cref{lemma:lamb-order-2}, we have $(1+\lambda_2)^{t+1}\in(0,1)$, and thus
\begin{equation}
    \label{eq:bound-2}
0\le x^{t+1}\le x^0\le \frac{r_y D_{\mathcal Y}}{b},
\qquad
0\le y^{t+1}\le y^0\le D_{\mathcal Y}.
\end{equation}
Combining \cref{eq:multi-2,eq:bound-2}, we conclude that \cref{eq:ind-claim-psgda} holds for $t+1$, which completes the induction.

Consequently, \eqref{eq:ind-claim-psgda} holds for all $t\in\mathbb N$.  This completes the proof.
\end{proof}

\subsection{Initial Gap of \cref{eg:gs-alter-tight} under Perturbed Smoothed GDA}\label{sec:del22}
\begin{lemma}\label{lemma:Delta_22}Suppose that \emph{Perturbed Smoothed GDA}  on \cref{eg:gs-alter-tight} is initialized in the same way as  \cref{gs-psgda-tight}. Then the initial gap satisfies
    \[\Delta_{\Psi_2}=\cO\left(\frac{\epsilon^2}{\ell}\right).\]
\end{lemma} 
\begin{proof}[Proof of \cref{lemma:Delta_22}.]
Recall from \cref{def:initial_gap} that
\[
\Delta_{\Psi_2}
:={\Psi_2^0}(x^0,y^0,z^0)-\min_{x\in\R}\max_{y\in[0,D_{\cY}]}f_0(x,y).
\]

Next, by \cref{eq:initial:psgda} we have $x^0=\Theta(\frac{\epsilon}{\ell})$, $z^0=\Theta(\frac{\epsilon}{\ell})$, and $y^0=\Theta(\frac{\epsilon}{b})$, and moreover $x^0\in[0,\frac{r_yD_{\cY}}{b}]$. Therefore,
 \begin{align*}
        &F_0(x^0,y^0,z^0)=-\frac{1}{2}\ell (x^0)^2+bx^0 y^0-\frac{1}{2}r_y (y^0)^2+\frac{r_x}{2}(x^0-z^0)^2,\\
        &d_0(y^0,z^0)=-\frac{(b y^0-r_x z^0)^2}{2(r_x-\ell)}+\frac{r_x}{2}(z^0)^2-\frac{r_y}{2}(y^0)^2,~\mbox{and}\\
        &p_0(z^0)=\frac{\ell r_x (z^0)^2}{r_x+2\ell}.
\end{align*}
Therefore, we can bound ${\Psi_2^0}(x^0,y^0,z^0)$ with
\begin{align*}
   & \, {\Psi_2^0}(x^0,y^0,z^0) \notag\\
    =
    &\, F_0(x^0,y^0,z^0)-2d_0(y^0,z^0)+2p_0(z^0)\notag\\
    =&\,-\frac{1}{2}\ell (x^0)^2+bx^0 y^0-\frac{1}{2}r_y (y^0)^2+\frac{r_x}{2}(x^0-z^0)^2\notag\\
    &\,-2\left(-\frac{(b y^0-r_x z^0)^2}{2(r_x-\ell)}+\frac{r_x}{2}(z^0)^2-\frac{r_y}{2}(y^0)^2\right)+\frac{2\ell r_x (z^0)^2}{r_x+2\ell}\notag\\
    \le&\,bx^0 y^0+\frac{r_x}{2}(x^0-z^0)^2+\frac{(b y^0-r_x z^0)^2}{r_x-\ell}+\frac{2\ell r_x (z^0)^2}{r_x+2\ell}=\cO\left(\frac{\epsilon^2}{\ell}\right),
\end{align*}
where the last equality follows from \cref{eq:initial:psgda,con-psgda}. Combining with \eqref{eq:minmax} yields the desired result. 
\end{proof}

\subsection{Proof of \cref{thm:tight-psgda}~(i)}\label{app:pf-tight-psgda-gs}
\begin{proof}[Proof of \cref{thm:tight-psgda}~(i).]
We initialize \emph{Perturbed Smoothed GDA}  as in \cref{gs-psgda-tight}, i.e., using \cref{eq:initial:psgda}.
If $(x^T,y^T)$ is an $\epsilon$-GS, then by definition,
\begin{align*}
\epsilon&\ge|\nabla_{x}f(x^T,y^T)|=|by^T-\ell x^T|\\
&=\left|\left(\frac{-c \ell \beta - c \ell \lambda_2 + c r_x \lambda_2 + \beta \lambda_2 - c r_x \beta \lambda_2 + \lambda_2^2}{c(-\beta-\lambda_2)}-\ell\right)x^T\right|\notag\\
&=\left|\frac{(c r_x (1 - \beta) + \beta + \lambda_2) \lambda_2}{c (\beta + \lambda_2)}x^T\right|\notag\\
&\ge\frac{r_x \lambda_2}{\beta+\lambda_2}|x^T|=\frac{r_x \lambda_2}{\beta+\lambda_2}(1+\lambda_2)^T|x^0|=2\epsilon(1+\lambda_2)^T.
\end{align*}
Therefore,
\[
(1+\lambda_2)^T \le \frac12
\quad\Longrightarrow\quad
T=\Omega\left(\frac{1}{|\lambda_2|}\right)
=\Omega\left(\frac{\ell D_{\cY}}{\epsilon}\right),
\]
where the last equality follows from \cref{lemma:lamb-order-2}. Finally, \cref{lemma:Delta_22} gives $\Delta_{\Psi_2}=\cO\left(\frac{\epsilon^2}{\ell}\right)$ under our initialization, and thus
\begin{align*}
    T&=\Omega\left(\frac{\ell D_{\cY}}{\epsilon}\right)=\Omega\left(\frac{\ell D_{\cY}}{\epsilon}\cdot\frac{\Delta_{\Psi_2}}{\epsilon^2/\ell}\right)=\Omega\left(\frac{\ell^2D_{\cY}\Delta_{\Psi_2}}{\epsilon^3}\right).
\end{align*}
This completes the proof.
\end{proof}

\subsection{$\ell$-Smoothness Property in \cref{eg:os-gda}}\label{sec:lem:smooth}
\begin{lemma}\label{lem:smooth}
   The function $f(x,y)$ in \cref{eg:os-gda} is $\ell$-smooth.
\end{lemma}
Before establishing the smoothness of $f$, we first introduce a technical lemma that will be used in the proof of \cref{lem:smooth}. 
\begin{lemma}[Bounded derivative implies Lipschitz continuity]
    \label{lemma-Lip}
    Let $\phi:\mathbb{R}\to\mathbb{R}$ be continuous and piecewise $C^1$.
    If $|\phi'(t)|\le \ell$ wherever $\phi'$ exists, then $\phi$ is $\ell$-Lipschitz.
    \end{lemma}
    \begin{proof}[Proof of \cref{lemma-Lip}.]
    Fix $t<t'$ and let $\{t_1,\dots,t_n\}$ be the (finite) set of points in $(t,t')$
    where $\phi$ is not differentiable. Define $t_0:=t$ and $t_{n+1}:=t'$.
    Then $\phi$ is differentiable on each interval $(t_i,t_{i+1})$.
    By the mean value theorem, for every $i=0,\dots,n$, there exists
    $\xi_i\in(t_i,t_{i+1})$ such that
    \[
    |\phi(t_{i+1})-\phi(t_i)| = |\phi'(\xi_i)|\,|t_{i+1}-t_i|
    \le \ell\,|t_{i+1}-t_i|.
    \]
    Summing over $i$ yields
    \[
    |\phi(t')-\phi(t)|
    \le \sum_{i=0}^n |\phi(t_{i+1})-\phi(t_i)|
    \le \ell \sum_{i=0}^n |t_{i+1}-t_i|
    = \ell|t'-t|,
    \]
    which proves that $\phi$ is $\ell$-Lipschitz. This completes the proof.
    \end{proof}
    With \cref{lemma-Lip} in hand, we are ready to prove \cref{lem:smooth}.

\begin{proof}[Proof of \cref{lem:smooth}.]
By \cref{lemma-Lip}, to prove \cref{lem:smooth}, it suffices to show that for all $x\in\R$ with $|x|\neq 1$ and $|x|\neq 2$,
    \begin{equation}
        \label{partialwithell}
         \max\{|\nabla_{xx}f(x,y)|,\,|\nabla_{xy}f(x,y)|,\,|\nabla_{yy}f(x,y)|\}\leq\ell.
    \end{equation}
   A direct calculation shows that the second-order derivatives are given by
   \begin{align*}
   |\nabla_{xx}f(x,y)|&=\frac{\ell}{(D_{\cY}+1)}|y|\leq\frac{\ell D_{\cY}}{D_{\cY}+1}\leq\ell, \,\,\,\,\,\forall \,x\in\R~\mbox{and}~|x|\neq 1\mbox{ or }2\\
   |\nabla_{xy}f(x,y)|&=
		\begin{cases}
			\frac{\ell}{(D_{\cY}+1)}|x|\leq\frac{\ell}{(D_{\cY}+1)} \leq \ell,&  \lvert x \rvert \le 1, \\
			\frac{\ell}{(D_{\cY}+1)}|2\textrm{sign}(x)-x| \leq\frac{\ell}{(D_{\cY}+1)} \leq \ell,& 1 < \lvert x \rvert \le 2, \\
			0\leq\ell, &  \lvert x \rvert >2,
		\end{cases}\\
    |\nabla_{yy}f(x,y)|&=0\leq\ell, \,\,\,\,\, \forall x\in\R. 
    \end{align*}
   Therefore, \cref{partialwithell} holds and we derive our desired property of $f$. This completes the proof.
   \end{proof}
\subsection{Initial Gap of \cref{eg:os-gda} under Perturbed GDA}\label{sec:lem:delta_11}
\begin{lemma}\label{lem:delta1}
Suppose that the initialization is chosen as $(x_0,y_0)=(\tfrac{3 D_{\mathcal{Y}}+ 2}{\ell D_{\mathcal{Y}}}\epsilon,D_{\cY})$. Then, it holds that 
\[\Delta_{\Psi_1}=\cO\left(\frac{D_{\cY}+1}{\ell D_{\cY}}\epsilon^2\right).\]
\end{lemma}  
\begin{proof}[Proof of \cref{lem:delta1}.]
Recall that ${\Delta}_{{\Psi_1}}={\Psi_1^0}(x^0,y^0)-\min_{x\in\R}\max_{y\in[0,D_{\cY}]}f_0(x,y)$, where $\Psi_1^0(x^0,y^0)=2\Phi_0(x^0)-f_0(x^0,y^0)$. 
    To bound $\Delta_{{\Psi_1}}$, we first compute $\Phi_0(x^0)$, which is given by
    \begin{align}
        \Phi_0(x^0)=&\,\max_{y\in[0,D_{\cY}]}f_0(x^0,y)\notag\\
        =&\,\max_{y\in[0,D_{\cY}]}\left(h(x^0)y-\frac{r_y}{2}y^2\right)\notag\\
        \leq&\,\max_{y\in[0,D_{\cY}]}h(x^0)y\notag\\
        =&\,\frac{\ell}{2(D_{\cY}+1)}(x^0)^2D_{\cY}.\label{eq:tilde-Phi-ub}
    \end{align}
    Thus, we bound the initial gap as follows:
    \begin{align*}
        {\Delta}_{{\Psi_1}}=&\,{\Psi_1^0}(x^0,y^0)-\min_{x\in\R}\max_{y\in[0,D_{\cY}]}f_0(x,y)\\
        =&\,2\Phi_0(x^0)-f_0(x^0,y^0)-\min_{x\in\R}\max_{y\in[0,D_{\cY}]}f_0(x,y)\\
         \le&\,\frac{\ell}{D_{\cY}+1}(x^0)^2D_{\cY}-\left(\frac{\ell}{2(D_{\cY}+1)}(x^0)^2y^0-\frac{r_y}{2}(y^0)^2\right)-\min_{x\in\R}\max_{y\in[0,D_{\cY}]}\left(h(x)y-\frac{r_y}{2}y^2\right)\\
        =&\,\left(\frac{\ell}{2(D_{\cY}+1)}(x^0)^2D_{\cY}+\frac{r_y}{2}D_{\cY}^2\right)-\min_{x\in\R}\max_{y\in[0,D_{\cY}]}\left(h(x)y-\frac{r_y}{2}y^2\right)\\
        \le&\,\left(\frac{\ell}{2(D_{\cY}+1)}(x^0)^2D_{\cY}+\frac{r_y}{2}D_{\cY}^2\right)-\min_{x\in\R}\max_{y\in[0,D_{\cY}]}h(x)y+\frac{r_y}{2}D_{\cY}^2\\
        =&\,\frac{\ell}{2(D_{\cY}+1)}(x^0)^2D_{\cY}+r_yD_{\cY}^2\\
        =&\,\Theta\left(\frac{2(3D_{\cY}+2)^2}{\ell(D_{\cY}+1)D_{\cY}}\epsilon^2+\frac{1}{\ell}\epsilon^2\right)\\
        =&\,\Theta\left(\frac{(D_{\cY}+1)}{\ell D_{\cY}}\epsilon^2\right).
    \end{align*}
    Here, the first inequality follows from \cref{eq:tilde-Phi-ub} and the equality $y^0=D_\cY$. Moreover, for any $x \in \R$, we have $\max_{y\in[0,D_{\cY}]}(h(x)y-\frac{r_y}{2}y^2)\ge \max_{y\in[0,D_{\cY}]}h(x)y-\frac{r_y}{2}D_\cY^2$.
    Therefore, taking the minimum over $x\in\R$ on both sides gives, 
    \[
        \min_{x\in\R}\max_{y\in[0,D_{\cY}]}\left(h(x)y-\frac{r_y}{2}y^2\right)\ge \min_{x\in\R}\max_{y\in[0,D_{\cY}]}h(x)y-\frac{r_y}{2}D_\cY^2.
    \] 
    Next, note that $h(x)\ge 0$ for all $x \in \R$, for each fixed $x$, the maximizer of the mapping $y\mapsto h(x)y$ over $[0,D_{\cY}]$ is $y^\star=D_\cY$, and hence $\min_{x\in\R}\max_{y\in[0,D_{\cY}]}h(x)y=0$, which leads to the fourth equality.  
    Finally, the fifth equality uses the expression of $x^0$ and $r_y$. This completes the proof. 
\end{proof}
\subsection{Proof of \cref{lemma:eigen-value-M}}\label{sec:lem:eig}
    \begin{proof}[Proof of \cref{lemma:eigen-value-M}.]
    First, we establish that $\lambda_3<0$. Starting from \eqref{lambda_OS}, we have   
   \begin{align}
       \lambda_3&=-\frac{A_3+B_3\beta}{2}+ \frac{1}{2}\sqrt{(A_3+B_3\beta)^2-4\ell c\beta\gamma}\notag\\
       &=-\frac{1}{2}\cdot\frac{4\ell c\beta\gamma}{A_3+B_3\beta+\sqrt{(A_3+B_3\beta)^2-4\ell c\beta\gamma}}\notag\\
       &=-\frac{2\ell c\gamma}{A_3+B_3\beta+\sqrt{(A_3+B_3\beta)^2-4\ell c\beta\gamma}}\beta.\label{eq:lamb-expression}
   \end{align}
To show that $\lambda_3<0$, because $A_3>0$ and $B_3>0$, it suffices to verify that 
\[(A_3+B_3\beta)^2-4\ell c\beta\gamma \ge 0.\] 
Indeed, we have
$
(A_3+B_3\beta)^2-4\ell c\beta\gamma
= A_3^2 + \cO(\beta),
$
and thus the above inequality holds for $\beta=\Theta\left(\frac{\epsilon^2}{\ell^2D_y^2}\right)$ sufficiently small.

Moreover, we estimate the magnitude of $\lambda_3$. By \eqref{eq:lamb-expression}, the numerator is of order $\Theta(\gamma)$, while the denominator is of order $\Theta(1)$ by \cref{con-psgda}. Therefore,
\[
|\lambda_3| = \cO(\gamma\beta)
= \cO\left(\frac{\gamma\epsilon^2}{\ell^2D_y^2}\right).
\]
This completes the proof. 
\end{proof}

    \subsection{Proof of \cref{x_0withz_0}}\label{sec:lem:x0}
    \begin{proof}[Proof of \cref{x_0withz_0}.]

One can verify that an eigenvector associated with $\lambda_3$ is given by 
\begin{align*}
\bm{v} 
  &= \left(
    \frac{1}{2(1 - \ell c\gamma - cr_x)\beta}
    \Bigl[
      -\ell c\gamma - cr_x + \beta - cr_x\beta 
    \Bigr. \right. \\
  &\quad \left. \left.
    + \sqrt{
        (\ell c\gamma + cr_x)^2
        - 2(\ell c\gamma + c^2\gamma\ell r_x + c^2r_x^2 - cr_x)\beta 
        + (cr_x - 1)^2\beta^2
      }
    \right],
    1
  \right) \\
  &= \left(
    \frac{1}{2(1 - A_3)\beta}
    \left[
      -A_3 + B_3\beta 
      + \sqrt{(A_3 - B_3\beta)^2 + 4cr_x\beta(1 - A_3)}
    \right],
    1
  \right).
\end{align*}
Let $v_1$ denote the first entry of $\bm{v}$ in the above expression. We next show that $v_1$ admits the representation in \cref{eq:v1-expression}. In particular, we have
	\begin{align*}
		v_1=\,&\frac{1}{2(1-A_3)\beta}\left(-A_3+B_3\beta+\sqrt{(A_3-B_3\beta)^2+4cr_x\beta(1-A_3)}\right)\\
        =\,&\frac{1}{2(1-A_3)\beta}\cdot\frac{4cr_x\beta(1-A_3)}{A_3-B_3\beta+\sqrt{(A_3-B_3\beta)^2+4cr_x\beta(1-A_3)}}\\
		=\,&\frac{2cr_x}{A_3-B_3\beta+\sqrt{(A_3-B_3\beta)^2+4cr_x\beta(1-A_3)}}\\
        =\,&\frac{2cr_x}{A_3+\cO(\beta)+A_3+\cO(\beta)}\\
        =\,&\frac{r_x}{r_x+\gamma\ell}+\cO(\beta),
	\end{align*}
where the fourth equality follows from $\beta$ is sufficiently small and the Maclaurin expansion of $(1+\cdot)^{-\frac{1}{2}}$.
The last equality uses the fact that $\beta$ is sufficiently small and the Maclaurin expansion of $(1+\cdot)^{-1}$. This completes the proof.
    \end{proof}
 \subsection{Initial Gap of \cref{eg:os-gda} under Perturbed Smoothed GDA}\label{sec:lem:del}
 \begin{lemma}\label{lemma:Delta-Psi2}
Suppose that \emph{Perturbed Smoothed GDA} applied to \cref{eg:os-gda} is initialized with $x^0=\frac{3 D_{\mathcal{Y}}+ 2}{\ell D_{\mathcal{Y}}}\epsilon, y^0=D_{\cY}~\mbox{and}~z^0=\frac{x^0}{v_1}$. Then the initial gap satisfies
    \[\Delta_{\Psi_2}=\cO\left(\frac{\epsilon^2}{\ell \gamma}\right).\]
\end{lemma} 
\begin{proof}[Proof of \cref{lemma:Delta-Psi2}.]
Recall that 
\begin{align*}
   &  \Delta_{\Psi_2}={\Psi_2^0}(x^0,y^0,z^0)-\min_{x\in\R}\max_{y\in[0,D_{\cY}]}f_0(x,y) \\
   & {\Psi_2^0}(x,y,z) = F_0(x,y,z) - 2\min_{x\in\R}F_0(x,y,z) + 2\min_{x\in\R}\max_{y\in[0,D_{\cY}]}F_0(x,y,z). 
\end{align*}
We first compute the explicit form of ${\Psi_2^0}(x,y,z) $. 
Recall that  $d_0(y^0,z^0)= \min_{x\in\R}F_0(x,y^0,z^0)$ and $p_0(z^0)= \min_{x\in\R}\max_{y\in[0,D_{\cY}]}F_0(x,y,z^0)$. 
Then, we have
    \begin{align*}
        &\,F_0(x^0,y^0,z^0)=\frac{\ell}{2(D_{\cY}+1)}(x^0)^2y^0+\frac{r_x}{2}(x^0-z^0)^2-\frac{r_y}{2}(y^0)^2,\\
        &\,d_0(y^0,z^0)=\min_{x\in\R}\left(h(x)y^0+\frac{r_x}{2}(x-z^0)^2-\frac{r_y}{2}(y^0)^2\right)\\
        &~~~~~~~~~~~~~~~=\min_{x\in\R}\left(h(x)D_{\cY}+\frac{r_x}{2}(x-z^0)^2\right)-\frac{r_y}{2}D_{\cY}^2,\\
        &\,p_0(z^0)=\min_{x\in\R}\left(\max_{y\in[0,D_{\cY}]}\left(h(x)y+\frac{r_x}{2}(x-z^0)^2-\frac{r_y}{2}y^2\right)\right)\\
        &~~~~~~~~~\,\leq\min_{x\in\R}\left(\max_{y\in[0,D_{\cY}]}\left(h(x)y+\frac{r_x}{2}(x-z^0)^2\right)\right)\\
        &~~~~~~~~~\,=\min_{x\in\R}\left(h(x)D_{\cY}+\frac{r_x}{2}(x-z^0)^2\right).
    \end{align*}
Consequently, we obtain
    \begin{align*}
    {\Psi_2^0}(x^0,y^0,z^0)&=F_0(x^0,y^0,z^0)-2d_0(y^0,z^0)+2p_0(z^0)\\
    &\le\frac{\ell}{2(D_{\cY}+1)}(x^0)^2y^0+\frac{r_x}{2}(x^0-z^0)^2+\frac{r_y}{2}(y^0)^2. 
    \end{align*}
Moreover, the initial gap $\Delta_{\Psi_2}$ satisfies
    \begin{align}
        \Delta_{\Psi_2}&={\Psi_2^0}(x^0,y^0,z^0)-\min_{x\in\R}\max_{y\in[0,D_{\cY}]}\left(f(x,y)-\frac{r_y}{2}y^2\right)\notag\\
        &\leq\frac{\ell}{2(D_{\cY}+1)}(x^0)^2y^0+\frac{r_x}{2}(x^0-z^0)^2+\frac{r_y}{2}(y^0)^2-\min_{x\in\R}\max_{y\in[0,D_{\cY}]}\left(f(x,y)-\frac{r_y}{2}D_{\cY}^2\right)\notag\\
        &=\frac{\ell}{2(D_{\cY}+1)}(x^0)^2D_{\cY}+\frac{r_x}{2}(x^0-z^0)^2+r_yD_{\cY}^2, \label{eq:Del_Psi_2}
    \end{align}
where the second equality uses that $\min_{x\in\R}\max_{y\in[0,D_{\cY}]}f(x,y)=0$.

It remains to bound the order of the terms above. By \cref{x_0withz_0}, we have $z^0-x^0=\Theta\left(\left(\frac{r_x+\gamma\ell}{r_x}-1\right)x^0\right)=\Theta\left(\gamma x^0\right)=\Theta\left(\frac{\epsilon}{\ell}\right)$. 
Substituting these estimates yields $\Delta_{\Psi_2}=\cO\left(\frac{D_{\cY}+1}{\ell D_{\cY}}\epsilon^2+\frac{\epsilon^2}{\ell}+\frac{\epsilon^2}{\ell}\right)=\cO\left(\frac{\epsilon^2}{\ell \gamma}\right)$.
This completes the proof. 
\end{proof}

\subsection{A Technical Lemma for \cref{thm:Smoothed-FOAM-epi}}
\begin{lemma} 
  \label{lem:closeness}
 For $\hat{\x}$ and $\hat{\y}$ defined in \cref{eq:x-step,eq:y-step},
  we have 
    \begin{equation*}
   \|\x^{t^\star+1}-\hat{\x}\|=\cO\left(\sqrt{\delta^{t^\star}}\right),\quad \text{and} \quad  \|\y^{t^\star+1}-\hat{\y}\|=\cO\left(\frac{\ell}{r_{\y}^{t^\star}}\sqrt{\delta^{t^\star}}\right). 
    \end{equation*}
\end{lemma}
\begin{proof}[Proof of \cref{lem:closeness}.]
First of all, we have 
            \begin{align*}
            &\,\|\x^{t^\star+1}-\hat{\x}\|\\
            \le&\,\|\x^{t^\star+1}-{\x}_{t^\star}(\y^{t^\star+1},\z^{t^\star})\|+\|\hat{\x}-{\x}_{t^\star}(\y^{t^\star+1},\z^{t^\star})\|\\
            \leq&\, (2+\alpha(r_{\x}+\ell))\|\x^{t^\star+1}-{\x}_{t^\star}(\y^{t^\star+1},\z^{t^\star})\|\\
            \le&\,(2+\alpha(r_{\x}+\ell))\left(\|\x^{t^\star+1}-{\x}_{t^\star}({\y}_{t^\star}(\z^{t^\star}),\z^{t^\star})\|+\|{\x}_{t^\star}({\y}_{t^\star}(\z^{t^\star}),\z^{t^\star})-{\x}_{t^\star}(\y^{t^\star+1},\z^{t^\star})\|\right)\notag\\
            \le&\,(2+\alpha(r_{\x}+\ell))\!\left(\!\left\|\x^{t^\star+1}-\!{\x}_{t^\star }^\star(\z^{t^\star})\right\|+\frac{2(r_{\x}+\ell+r_{\y}^{t^\star})}{r_{\x}-\ell-r_{\y}^{t^\star}}\|\y^{t^\star+1}-{\y}_{t^\star}(\z^{t^\star})\|\right)
            =\cO\left(\sqrt{\delta^{t^\star}}\!\right),
            \end{align*}
            where the second inequality follows from the nonexpansiveness of $\proj_{\cX}(\cdot)$, the fact that ${\x}_{t^\star}(\y^{t^\star+1},\z^{t^\star}) = \proj_{\cX}({\x}_{t^\star}(\y^{t^\star+1},\z^{t^\star}) -\alpha\nabla_{\x} F_{t^\star}({\x}_{t^\star}(\y^{t^\star+1},\z^{t^\star}),\y^{t^\star+1},\z^{t^\star}) )$, and $(r_{\x}+\ell)$-Lipschitz continuity of $\nabla_{\x} F_{t^\star}(\cdot,\y^{t^\star+1},\z^{t^\star})$, the fourth inequality is from \citep[Lemma B.2]{zhang2020single}, and the last equality follows from \cref{FOAM,con-FOAM}, i.e., $r_{\x}>3\ell$, and $\alpha<\frac{1}{r_{\x}+\ell}$.
            
           We now turn to the dual part. By a similar argument, we have 
            \begin{align*}
            &\,\|\y^{t^\star+1}-\hat{\y}\|\\
            \leq&\,(2+c(r_{\y}^{t^\star}+\ell))\left(\left\|\y^{t^\star+1}-{\y}_{t^\star}^\star\!\left({\x}^\star_{t^\star}(\z^{t^\star})\right)\right\|+\left\|{\y}^\star_{t^\star}(\x^{t^\star+1})-{\y}^\star_{t^\star}\!\left({\x}_{t^\star}^\star(\z^{t^\star})\right)\right\|\right)\\
            \le&\,(2+c(r_{\y}^{t^\star}+\ell))\left(\|\y^{t^\star+1}\!-\!{\y}_{t^\star}(\z^{t^\star})\|\!+\!\frac{\ell}{r_{\y}^{t^\star}}\left\|\x^{t^\star+1}\!-\!{\x}_{t^\star}^\star(\z^{t^\star})\right\|\right)=\cO\left(\frac{\ell}{r_{\y}^{t^\star}}\sqrt{\delta^{t^\star}}\right),
            \end{align*}
            where the second inequality follows from that $\y_{t^\star}(\z)=\y^\star_{t^\star}(\x^\star_{t^\star}(\z))$ for all $\z$ and ${\y}^\star_{t^\star}(\cdot)$ is $\frac{\ell}{ r_{\y}^{t^\star}}$-smooth \citep[Lemma 4.3]{lin2020gradient}, and the last equality comes from \cref{FOAM} and $c<\frac{1}{r_{\y}^{t^\star}+\ell}$. 
            This completes the proof.
\end{proof}

\section{Convergence Analysis of Smoothed GDA}\label{app:smgda}
In this section, we refine the iteration–complexity result of \citep{zhang2020single} by making explicit its dependence on the smoothness constant $\ell$ and the diameter $D_\cY$, which  reveals  how these structural problem parameters dictate the convergence rate. Moreover, unlike \citep{zhang2020single}, which only establishes guarantees for $\epsilon$-GS, we further characterize how to attain an $\epsilon$-OS solution. The mechanism for achieving $\epsilon$-OS differs from the argument in \citep{li2025nonsmooth}.

Following our unified analysis framework, we adopt the same Lyapunov function as in \cref{Lya-fcn} under the restriction $r_{\y}^t=0$, i.e., $f_t=f$.  To proceed, we recall several parameter definitions originally introduced in \citep{zhang2020single}, which will repeatedly appear in our convergence analysis.


\begin{definition}[Parameters conditions for  Smoothed GDA]\label{def:sgda}
Under \cref{con-psgda},  we set the parameters as follows:
\begin{align}
&\sigma_2:=\frac{2(r_{\x}+\ell)}{r_{\x}-\ell}=\Theta(1),\notag\\
&\rho_1:=384r_{\x}D_{\cY}\frac{1+\alpha^t\ell+\alpha^t\ell\sigma_2}{r_{\x}-\ell}=\Theta(D_{\cY}),\notag\\
        &\rho_2:=384D_{\cY}\rho_1\frac{1+\alpha^t\ell+\alpha^t\ell\sigma_2}{\alpha^t(r_{\x}-\ell)}=\Theta\left(D_{\cY}^2\right),\notag\\
        &\rho_3:=384r_{\x}c^tD_{\cY}\rho_1\frac{1+\alpha^t\ell+\alpha^t\ell\sigma_2}{\alpha^t(r_{\x}-\ell)}=\Theta\left(D_{\cY}^2\right),\notag
\end{align}
\end{definition}
We now present the iteration–complexity result for \emph{Smoothed GDA}, together with its proof.
\begin{theorem}[Iteration complexity of {\it Smoothed GDA}]
\label{thm:sgda-epi}
Let 
$\{(\x^{t},\y^{t},\z^{t})\}_{t\ge 0}$ be generated by {\it Smoothed GDA}.
For any $\epsilon>0$, if the step sizes are chosen according to \cref{con-psgda},
then after at most 
$t=\cO\left(\tfrac{\ell^3 D_{\cY}^2 \Delta_{\Psi_2}}{\epsilon^4}\right)$ 
iterations, $(\x^{t},\y^{t})$ is an $\epsilon$-GS and $\x^{t}$ is an $\epsilon$-OS for 
problem \cref{eq:prob}.
\end{theorem}
\begin{proof}[Proof of \cref{thm:sgda-epi}.]
In the analysis below, we set 
$\beta^t = \Theta\left(\sqrt{\tfrac{\Delta_{\Psi_2}}{\ell D_{\cY}^2 T}}\right)$.
Recall from \citep[Proposition~4.1]{zhang2020single} that
$\Psi_2^t$ satisfies the basic descent bound
\begin{align}
\Psi_{2}^t-\Psi_{2}^{t+1}\ge\ &  \frac{1}{8c^t}\|\x^{t+1}-\x^{t}\|^{2}+\frac{1}{8\alpha^t}\|\y^{t}-\y_{+}^t(\z^{t})\|^2 +\frac{r_{\x}\beta^t}{8}\|\z^{t}-\x^{t+1}\|^{2} \notag\\
& - 24r_{\x}\beta^t\left\|\x_{t}^\star(\z^{t})-\x_{t}(\y_{+}^t(\z^{t}), \z^{t})\right\|^{2}.\label{bas-dec-psi}
\end{align}
This descent inequality motivates a case-based argument for bounding the
progress of each component.
There are two complementary situations, which we treat separately.

(i): There exists some $t\in\{0,1,\dots,T-1\}$ such that
\begin{align*}
    &\,\frac{1}{2}\max\left\{\frac{1}{8c^t}\|\x^{t+1}-\x^{t}\|^2,\frac{1}{8\alpha^t}\|\y^{t}-\y_{+}^t(\z^{t})\|^2,\frac{r_{\x}\beta^t}{8}\|\z^{t}-\x^{t+1}\|^2\right\}\\
    \leq&\,24r_{\x}\beta^t\left\|\x^\star_t(\z^{t})-\x_{t}(\y_{+}^t(\z^{t}),\z^{t})\right\|^2.
\end{align*}
Thanks to \citep[Theorem 3.4 (B.63)]{zhang2020single}, we have 
\begin{equation}
\label{smooth-xyz-1}
    \,\|\x^{t+1}-\x^{t}\|\leq\sqrt{\rho_3}\beta^t,\,\|\y^{t}-\y_{+}^t(\z^{t})\|\leq\rho_1\beta^t,\,\text{and}\,\|\z^{t}-\x^{t+1}\|\leq\sqrt{\rho_2\beta^t}. 
\end{equation}
(ii): For any $t\in\{0,1,\cdots,T-1\}$, we have
\begin{align}
    &\,\frac{1}{2}\max\left\{\frac{1}{8c^t}\|\x^{t+1}-\x^{t}\|^2,\frac{1}{8\alpha^t}\|\y^{t}-\y_{+}^t(\z^{t})\|^2,\frac{r_{\x}\beta^t}{8}\|\z^{t}-\x^{t+1}\|^2\right\}\notag\\
    \geq&\,24r_{\x}\beta^t\left\|\x^\star_t(\z^{t})-\x_{t}(\y_{+}^t(\z^{t}),\z^{t})\right\|^2.\label{case2}
\end{align}
By combining \cref{bas-dec-psi,case2}, we obtain that for any $t\in\{0,1,\cdots,T-1\}$ it holds
\begin{align*}
    {\Psi}_{2}^t-{\Psi}_{2}^{t+1}
    &\ge \frac{1}{16c^t}\|\x^{t+1}-\x^{t}\|^2+\frac{1}{16\alpha^t}\|\y^{t}-\y_{+}^t(\z^{t})\|^2+\frac{r_{\x}\beta^t}{16}\|\z^{t}-\x^{t+1}\|^2.
\end{align*}
 Thus, by \citep[Lemma~B.1]{zhang2020single}, for any integer $T>0$, there exists an index $t\in T$ such that
\begin{equation}
\label{smooth-xyz-2}
    \!\!\|\x^{t+1}\!-\!\x^{t}\| \!=\! \mathcal{O}\!\left(\!\sqrt{\frac{c^t\Delta_{\Psi_2}}{T}}\right)\!,
        \|\y^{t}\!-\!\y^{t}_+(\z^{t})\| \!=\!  \mathcal{O}\!\left(\!\sqrt{\frac{\alpha^t\Delta_{\Psi_2}}{T}}\right)\!, \|\z^{t}\!-\!\x^{t+1}\| \!=\!
    \mathcal{O}\!\left(\!\sqrt{\frac{\Delta_{\Psi_2}}{r_{\x}\beta^t T}}\right)\!.
    \end{equation}
   Therefore, combining the two cases and substituting the parameter choices from \cref{def:sgda,con-psgda} into \cref{smooth-xyz-1,smooth-xyz-2}, we obtain that there exists some $t\in[T]$ such that
\begin{equation}
\label{smooth-xyz}
    \|\x^{t+1}\!-\!\x^{t}\| \!=\! \mathcal{O}\!\left(\!\sqrt{\frac{\Delta_{\Psi_2}}{\ell T}}\right)\!,
        \|\y^{t}\!-\!\y^{t}_+(\z^{t})\| \!=\!  \mathcal{O}\!\left(\!\sqrt{\frac{\Delta_{\Psi_2}}{\ell T}}\right)\!, \|\z^{t}\!-\!\x^{t+1}\| \!=\!
    \mathcal{O}\!\left(\!\sqrt[4]{\frac{D_{\cY}^2\Delta_{\Psi_2}}{\ell T}}\right)\!.
    \end{equation}
Then, the necessary optimality condition of $\y$-update yields
\begin{align}
     &\,{\rm dist}(\bz,-\nabla_{\y} f(\x^{t+1}, \y^{t+1})+\partial\iota_{\mathcal{Y}}(\y^{t+1}))\notag\\
      \leq &\, \left(\frac{1}{\alpha^t}+\ell\right) \left(\|\y^{t}-\y_{+}^t(\z^{t})\|+\alpha^t\ell \frac{1+c^t(r_{\x}-\ell)}{c^t(r_{\x}-\ell)}\|\x^{t+1}-\x^{t}\|\right)\label{y-gap}\\
      =  &\,\cO\left(\sqrt{\frac{\ell\Delta_{\Psi_2}}{T}}\right),\label{sm-gs-1}
 \end{align}
where the first inequality is obtained by applying \cref{smooth-y} with \(r_{\y}^t=0\), 
and the last inequality follows from \cref{smooth-xyz,con-psgda}, namely 
$\alpha^t = \Theta(\frac{1}{\ell})$. Moreover, from \cref{opt-x-up}, we have 
\begin{align}
&{\rm dist}(\bz,\nabla_{\x} f(\x^{t+1}, \y^{t+1}) +\partial\iota_{\mathcal{X}}(\x^{t+1})) \notag\\
 \leq& \, \left(\frac{1}{c^t}\!+\!\ell\right) \|\x^{t+1}\!-\!\x^{t}\| \!+\! \ell\!  \left(\|\y^{t}\!-\!\y_{+}^t(\z^{t})\|\!+\!\alpha^t\ell \frac{1\!+\!c^t(r_{\x}\!-\!\ell)}{c^t(r_{\x}\!-\!\ell)}\|\x^{t+1}\!-\!\x^{t}\|\right) \!+\! r_{\x}\|\x^{t+1}\!-\!\z^{t}\| \notag\\
 =&\,\cO\left(\sqrt[4]{\frac{\ell^3 D_{\cY}^2 \Delta_{\Psi_2}}{T}}\right),\label{sm-gs-2}
\end{align}
where the last inequality comes from \cref{smooth-xyz,con-psgda}, i.e, $c^t=\Theta(\tfrac{1}{\ell}),~r_{\x}=\Theta(\ell)$.
Combining \cref{sm-gs-1,sm-gs-2}, we conclude that $(\x^{t+1},\y^{t+1})$ is an $\epsilon$-GS.  
Rewriting the bound in terms of~$\epsilon$ shows that achieving an $\epsilon$-GS requires
$
T=\cO\left(\frac{\ell^3 D_{\cY}^2 \Delta_{\Psi_2}}{\epsilon^4}\right),
$
which corresponds to choosing $\beta^t=\Theta\left(\tfrac{\epsilon^{2}}{\ell^2 D_{\cY}^{2}}\right)$.

For the OS analysis, \cref{lemma2} allows us to replicate the argument in the proof of \cref{thm:pgda-nc-c-epi}~(ii), with \cref{sm-gs-1,sm-gs-2} plugged in accordingly, which yields, 
\[2\ell\left\|\prox_{\frac{1}{2\ell} \Phi}(\x^{t+1})-\x^{t+1}\right\|= \cO\left(\sqrt[4]{\frac{\ell^3D_{\cY}^2 \Delta_{\Psi_2}}{ T}}\right). \]
Following the similar argument, we will need at least $T=\cO\left(\frac{\ell^3D_{\cY}^2\Delta_{\Psi_2}}{\epsilon^4}\right)$ to reach an $\epsilon$-OS if we choose $\beta^t = \Theta\left(\tfrac{\epsilon^2}{\ell^2 D^2_{\mathcal{Y}}}\right)$.
This completes the proof.
\end{proof}

\begin{proof}[Proof of \cref{thm:sgda-epi-tight}.]
Similar to \cref{sec:analysis}, we consider \emph{Smoothed GDA} with constant choice of $r^t_{\y}$ for simplicity and drop iteration superscripts. We write $r_{y}^t = r_{y}$, and $\beta^t = \beta$.

(i) We consider the following hard  instance: 
    \begin{gather*}
		f(x,y)=
		\begin{cases}
            -\frac{r_y}{2}y^2,&  ~\mbox{if}~ x<0,\\
			-\frac{1}{2}\ell x^2+bxy-\frac{r_y}{2}y^2, &  ~\mbox{if}~ 0\le x  \leq \frac{r_y D_{\cY}  }{b}, \\
			-\frac{\ell r_y^2 D_{\cY}^2}{2b^2}+r_yD_{\cY}y-\frac{r_y}{2}y^2, &~\mbox{if}~  x >\frac{r_y D_{\cY} }{b},
		\end{cases}
	\end{gather*}
where $x\in\R$ and $y\in[0,D_{\cY}]$, $r_y=\frac{\epsilon}{D_{\cY}}$, and $b=\sqrt{3\ell r_y}$.
Applying \emph{Smoothed GDA} to $f(x,y)$ is essentially identical to applying \emph{Perturbed Smoothed GDA}
to \cref{eg:gs-alter-tight}, except for the choice of $\beta$. 
Thus we apply the analysis of \cref{thm:tight-psgda} (i) to this problem. 
We initialize the algorithm at the same point
$(x^0,y^0,z^0)$ specified in \cref{gs-psgda-tight}. Since the lower bound proof follows the same argument as in
\cref{subsubsec:tight-psgda}, we omit the details.
As a result, the iteration complexity satisfies 
\begin{align*}
    T&=\Omega\left(\frac{\Delta_{\Psi_2}}{\ell\beta\epsilon^2}\right)=\Omega\left(\frac{\ell^3 D_{\cY}^2 \Delta_{\Psi_2}}{\epsilon^4}\right),
\end{align*}
which matches \cref{thm:sgda-epi}~(i).

(ii) We consider the same hard instance as in \cref{eg:os-gda}. 
We initialize \emph{Smoothed GDA} by selecting $(x^0,z^0)$ as an eigenvector associated with the eigenvalue $\lambda_3$. 
Specifically,
$x^0=\frac{3 D_{\mathcal{Y}}+ 2}{\ell D_{\mathcal{Y}}}\epsilon$ and $y^0=D_{\cY}$. 
Under this initialization, the resulting recursion coincides with \cref{eq:new_update_rule}.


Indeed, for all $x\in\cX$, we have $\nabla_y F(x,y,z)\ge 0$, and therefore the
$y$-iterate generated by \emph{Smoothed GDA} is nondecreasing. Since $y^0=D_{\cY}$, it follows that
$y^t=D_{\cY}$ for all $t\ge 0$. Consequently, the updates reduce to \cref{eq:new_update_rule}. By \cref{subsubsec:pgda-os-tight}, the recursion \cref{eq:new_update_rule} requires
$\Omega\left(\frac{\ell^3D_{\cY}^2\Delta_{\Psi_2}}{\epsilon^4}\right)$ iterations to reach an $\epsilon$-OS point on
\cref{eg:os-gda}. Hence, \emph{Smoothed GDA} also needs
$\Omega\left(\frac{\ell^3D_{\cY}^2\Delta_{\Psi_2}}{\epsilon^4}\right)$ iterations to find an $\epsilon$-OS point. This completes the proof.
\end{proof}
\end{appendices}
\end{document}